\documentclass[hidelinks,onefignum,onetabnum]{siamart220329}
\usepackage{tia_vpal}
\usepackage[symbol]{footmisc}

\newlength{\wdth}
\usepackage{hyperref}

\title{
The Variable Projected Augmented Lagrangian Method\thanks{Submitted to the editors \today.
\funding{This work was partially supported by the National Science Foundation (NSF) under grants  DMS-1723005, DMS-2152661 (Chung) and DMS-1913136, DMS-2152704 (Renaut)}}}

\author{Matthias Chung\thanks{Department of Mathematics, Emory University, Atlanta, GA
  (\email{matthias.chung@emory.edu}, \url{http:w ww.math.emory.edu/\string~mchun45/}).}
\and Rosemary Renaut\thanks{School of Mathematical and Statistical Sciences, Arizona State University, Tempe, AZ
  (\email{renaut@asu.edu}, \url{https://isearch.asu.edu/profile/85017}).}}

\begin{document}

\maketitle

\begin{abstract}
Inference by means of mathematical modeling from a collection of observations remains a crucial tool for scientific discovery and is ubiquitous in application areas such as signal compression, imaging restoration, and supervised machine learning. The inference problems may be solved using variational formulations that provide theoretically proven methods and algorithms. With ever-increasing model complexities and growing data size, new specially designed methods are urgently needed to recover meaningful quantifies of interest. We consider the broad spectrum of linear inverse problems where the aim is to reconstruct quantities with a sparse representation on some vector space; often solved using the (generalized) least absolute shrinkage and selection operator (lasso).  The associated optimization problems have received significant attention, in particular in the early 2000's, because of their connection to compressed sensing and the reconstruction of solutions with favorable sparsity properties using augmented Lagrangians, alternating directions and splitting methods. We provide a new perspective on the underlying $\ell^1$ regularized inverse problem by exploring the generalized lasso problem through variable projection methods. We arrive at our proposed variable projected augmented Lagrangian ({\tt vpal}) method. We analyze this method and provide an approach for automatic regularization parameter selection based on a degrees of freedom argument. Further, we provide numerical examples demonstrating the computational efficiency for various imaging problems.
\end{abstract}

\section{Introduction}\label{sec:intro}

Many scientific problems are modeled as
\begin{equation} \label{eq:invprob}
  \bfb = \bfA \bfx_{\rm true} + \bfvarepsilon,
\end{equation}
where $\bfx_{\rm true} \in \bbR^n$ is a desired solution, $\bfA\in\bbR^{m\times n}$ denotes a forward process that includes a physical model and the mapping onto observations $\bfb \in \bbR^m$, and $\bfvarepsilon$ is, for simplicity, some unknown additive Gaussian noise $\bfvarepsilon\sim\calN({\bf0}, \sigma^2\bfI_m)$. Given observations $\bfb$ and the forward process $\bfA$, in \emph{inverse problems} we aim to obtain an approximate solution $\widehat\bfx$ to $\bfx_{\rm true}$~\cite{hansen2010discrete}. We assume the inverse problem is either ill-posed (i.e., a solution does not exist, is not unique, or does not depend continuously on the data~\cite{hadamard1923lectures}) or that we are given specific prior knowledge about the unknown $\bfx_{\rm true}$. Inverse problems arise in many different fields, such as medical imaging, geophysics, and signal processing to name a few~\cite{hansen2006deblurring, meju1994geophysical, santamarina2005discrete}.

\emph{Regularization}
is included to incorporate the prior knowledge and/or to stabilize the inversion process to obtain a meaningful approximation $\widehat \bfx$.  Many types of regularization exist; here we focus on the variational regularization problem
\begin{equation}\label{eq:general_lasso}
    \widehat \bfx \in \argmin_\bfx \ \thf \norm[2]{\bfA\bfx -\bfb}^2 + \mu \norm[1]{\bfD\bfx},
\end{equation}
where $\norm[p]{\mdot}$ denotes the $\ell^p$-norm, $\bfD\in \bbR^{\ell \times n}$ is a predefined matrix and $\mu>0$ is the regularization parameter that weights the relevant contributions of the $\ell^2$-data loss and the $\ell^1$-regularizer. This optimization problem has gained major attention in particular in the early 2000's due to its connection to compressed sensing and the reconstruction of solutions $\widehat\bfx$ with sparsity properties in the range of $\bfD$ \cite{candes2005decoding}. The $\ell^1$-regularization in \cref{eq:general_lasso} is also referred to as (generalized, if $\bfD\ne \bfI_n)$) \emph{least absolute shrinkage and selection operator} (lasso) regression or \emph{basis pursuit denoising} (BPDN), and has major use in various applications such as signal compression, denoising, deblurring, and dictionary learning~\cite{tibshirani1996regression, TibshiraniTaylor:2011}.

While a necessary condition for the existence of a unique solution of \cref{eq:general_lasso} is given by $\mfN(\bfA) \cap \mfN(\bfD) = \{\bf0\}$ ($\mfN({\bfA})$ denoting the null space of the matrix $\bfA$) \cite{tibshirani2017sparsity}, a unique solution is ensured when $\rank{\bfA} = n$, because \cref{eq:general_lasso} is strictly convex.  Note that \cref{eq:general_lasso} also has a Bayesian interpretation, where $\widehat\bfx$ represents the maximum a-posteriori (MAP) estimate of a posterior with linear model, Gaussian likelihood and a particular Laplace prior~\cite{calvetti2007introduction}.

The variational regularization problem \cref{eq:general_lasso} is solved via an optimization algorithm, and as such it is computationally challenging to find the solution, particularly in the large-scale setting.
The introduction of methods including the Alternating Direction Method of Multipliers (ADMM), Split Bregman, and the Fast Iterative Shrinkage-Thresholding Algorithm \cite{goldstein2014fast,getreuer,parikh2014proximal}, have, however, made it computationally feasible to solve large-scale problems described by \cref{eq:general_lasso}, see for example \cite{esser2009applications,fukushima_1992,goldstein2009split,beck_teboulle_2009}. Nevertheless, compared to $\ell^2$-regularized linear least-squares (Tikhonov)
\begin{equation}\label{eq:Tik}
    \min_\bfx \ \thf \norm[2]{\bfA\bfx -\bfb}^2 + \tfrac{\mu^2}{2} \norm[2]{\bfD\bfx}^2,
\end{equation}
also referred to as \emph{ridge regression}, solving $\ell^1$-regularization remains significantly more computationally expensive because it requires nonlinear optimization methods to solve \cite{tibshirani1996regression}.

Another challenge for inverse problems is the selection of an appropriate regularization parameter $\mu$. Within real-world applications, the regularization parameter $\mu$ remains a priori unknown. Various methods have been proposed for finding efficient estimators of $\mu$ for $\ell^2$-regularization \cite{hansen2010discrete,bardsley2018computational}. Generally, the determination of such $\mu$ is posed as finding the root,  or the minimum, of a function that depends on the solution of  \cref{eq:Tik}. Evaluation of the specific function relies on the solver for finding \cref{eq:Tik} but can generally be designed efficiently for $\ell^2$-regularization problems, see same references as above. In contrast, regularization parameter selection methods for $\ell^1$-regularization are generally more computationally demanding, because the methods require a full optimization solve of \cref{eq:general_lasso} for each choice of $\mu$. Dependent on whether the underlying noise distribution is known, a few approaches have been proposed for defining an optimal choice of $\mu$, including the use of the discrepancy principle (DP), computationally expensive cross-validation, and supervised learning techniques, \cite{osher2005iterative,afkham2021learning}. In the statistical community, arguments based on the \emph{degrees of freedom} (DF) in the obtained solutions are prevalent for finding $\mu$, e.g., \cite{TibshiraniTaylor,Mead_2020}.

Here, we present a new variable projected augmented Lagrangian (VPAL) method for solving \cref{eq:general_lasso}. There are two main contributions of this work. (1) We introduce and analyze the VPAL method. We provide numerical evidence of the low computational complexity of the corresponding {\tt vpal} algorithm. Thus, this algorithm is efficient to run for a selection of regularization parameters. (2) In the context of image deblurring, {\tt vpal} is augmented with an automated algorithm for selecting the regularization parameter $\mu$. This uses an efficient implementation of a $\chi^2$ test based on an optimal \emph{degrees of freedom} (DF) argument for generalized lasso problems, as applied to problems for image restoration, but here is also applied successfully for tomography and projection problems. Our findings are corroborated with numerical experiments on various imaging applications.

This work is organized as follows. In~\cref{sec:background} we introduce further notation and provide background on $\ell^1$-regularization methods. A discussion on our regularization parameter selection approach is given in~\cref{sec:reg}. We present our numerical method in~\cref{sec:vpal} and discuss convergence results followed by various numerical investigations in~\cref{sec:numerics}. We conclude our work with a discussion in~\cref{sec:discussion}.

\section{Background}\label{sec:background}
We briefly elaborate on approaches that have been used to solve generic separable optimization problems in \cref{sec:separable} and on approaches specifically for \cref{eq:general_lasso} in \cref{sec:ADMM}. Our discussion leads to the description of a standard algorithm for such problems in \cref{alg:admm}.

\subsection{Separable optimization problems}\label{sec:separable}
Consider a generic nonlinear but \emph{separable} optimization problem
\begin{equation}\label{eq:genopt}
    \min_{(\bfx, \bfy) \in \bbR^{n+\ell}} \ f(\bfx,\bfy),
\end{equation}
where separability refers to functions $f$ where the arguments can be decomposed into two sets of independent variables $\bfx\in\bbR^n$ and $\bfy\in\bbR^\ell$. We assume, for simplicity, that $f$ is strictly convex, has compact lower level sets, and is sufficiently differentiable to ensure convergence to a unique solution $(\widehat \bfx,\widehat \bfy)$ \cite{stefanov2001separable}. To address the solution of \cref{eq:genopt} there are various numerical optimization techniques, which may be classified into \emph{four} main approaches.\\

If we are \emph{not} taking advantage of the separable structure of the arguments $\bfx$ and $\bfy$ of $f$ by setting $\bfz = (\bfz,\bfy)$ and $\min_\bfz \, f(\bfz)$, standard optimization methods such as gradient-based, quasi-Newton, and Newton-type methods can be employed to jointly optimize for $\bfx$ and $\bfy$,~\cite{nocedal2006numerical}, i.e., $\bfz_{k+1} = \bfz_k + \alpha_k \bfs_\bfz(f,\bfz_k)$ with appropriate descent direction $\bfs_\bfz$ and step size $\alpha_k$.  Thus the numerical approach may not benefit from any inherent separability.\\

An \emph{alternating direction} approach may leverage the structure of a separable problem. In this setting the numerical solution of~\cref{eq:genopt} is found by repeatedly alternating the optimization with respect to one of the variables, while keeping the other fixed. Specifically, given an initial guess $\bfy_0 \in \bbR^\ell$, the \emph{alternating direction} method generates iterates
\begin{subequations}
    \begin{align}\label{eq:ad}
        \bfx_{k+1} &= \argmin_{\bfx\in \bbR^{n}} \ f(\bfx,\bfy_k) \\
        \bfy_{k+1} &= \argmin_{\bfy\in \bbR^{\ell}} \ f(\bfx_{k+1},\bfy)
    \end{align}
\end{subequations}
until convergence is achieved. Since each optimization problem is handled separately, specifically tailored numerical schemes may be independently utilized for optimization with respect to each $\bfx$ and $\bfy$, \cite{boyd2011distributed,goldstein2009split}.
Various constrained optimization problems can be approximated by unconstrained optimization problems through  introducing  additional (slack) variables. This circumvents the challenge of using computationally expensive constrained optimization solvers and leverages the benefits of applying alternating direction methods for the primary and slack variables \cite{nocedal2006numerical}.

On the other hand, an alternating direction optimization may lead to artificially introduced computational inefficiencies, such as zig-zagging.\\

Another perspective that yields an alternating approach, referred to as \emph{block coordinate descent}, arises by taking an optimization \emph{step} in each variable direction $\bfx$ and $\bfy$ while alternating over those directions. In this case the iterates, initialized with arbitrary $\bfx_0,\bfy_0$, proceed until convergence via the alternating steps
\begin{subequations}
    \begin{align}\label{eq:bcd}
        \bfx_{k+1} &= \bfx_k + \alpha_k\bfs_\bfx(f,\bfx_{k},\bfy_k)\\
        \bfy_{k+1} &= \bfy_k + \beta_k\bfs_\bfy(f,\bfx_{k+1},\bfy_k).\label{eq:cdy}
    \end{align}
\end{subequations}
Here, $\bfs_\bfx$ and $\bfs_\bfy$ refer to appropriate descent directions with corresponding step sizes $\alpha_k$ and $\beta_k$. Within each iteration $k$, the function $f$ is kept constant with respect to one of the variables while the other performs for instance a gradient descent or Hessian step with appropriate step size control~\cite{wright2015coordinate}. While this approach may leverage computational advances, in particular where mixed derivative information is hard to compute, the convergence may again be slow, for many of the same reasons as are seen with algorithms based on alternating directions.\\

Finally, consider the less-utilized \emph{variable projection} technique, which is a hybrid of the previous two approaches \cite{golub1973differentiation, o2013variable}. At each iteration of this numerical optimization scheme, we first optimize over $\bfx$ while keeping $\bfy$ constant, and then perform a descent step with respect to $\bfy$, while keeping $\bfx$ constant. Specifically, again taking $\bfy_0$ to be an arbitrary initial guess, we iterate to convergence via
\begin{subequations}\label{eq:varpro}
    \begin{align}
        \bfx_{k+1} &= \argmin_{\bfx\in \bbR^{n}}  \ f(\bfx,\bfy_k) \\
        \bfy_{k+1} &= \bfy_k + \alpha_k\bfs_\bfy(f,\bfx_{k+1},\bfy_k).
    \end{align}
\end{subequations}
The iterations defined by \cref{eq:varpro} are equivalent to the single iterative step given by
\begin{equation}\label{eq:singlevarpro}
\bfy_{k+1} = \bfy_k + \alpha_k\bfs_\bfy(f,\argmin_{\bfx\in \bbR^{n}}  \ f(\bfx,\bfy_k),\bfy_k),
\end{equation}
from which it is apparent that the key idea behind variable projection is elimination of one set of variables $\bfx$ by  projecting the optimization problem onto a reduced subspace associated with the other set of variables $\bfy$ \cite{golub2003separable,sjoberg1997separable}. Variable projection approaches were specifically developed for \emph{separable nonlinear least-squares} problems, such as  $\min_{(\bfx,\bfy)\in\bbR^{n+\ell}} \ \thf\|\bfA(\bfy)\bfx-\bfb\|_2^2$ where for fixed $\bfy$ the optimization problem exhibits a \emph{linear} least-squares problem in $\bfx$ which may easily be solved. This approach has been successfully applied beyond standard settings, for example, in super-resolution applications, for separable deep neural networks, and within stochastic approximation frameworks, see~\cite{chung2019iterative, newman2020train,newman2021slimtrain} for details. Variable projection methods show their full potential when one of the variables can be efficiently eliminated.\bigskip

\subsection{Methods to solve generalized lasso problems}\label{sec:ADMM}
For the purposes of deriving our new algorithm for the solution of~\cref{eq:general_lasso}, we need to recast the problem in the context of the ADMM algorithm, as first introduced in \cite{glowinski1975approximation} and \cite{gabay1976dual}, and for which extensive details are given in \cite{esser2009applications}.
We first introduce a variable $\bfy\in \bbR^{\ell}$ with $\bfy = \bfD\bfx$. Then problem \cref{eq:general_lasso} is equivalent to
\begin{equation}\label{eq:general_lassosplit}
    \min_{\bfx,\bfy} \ f(\bfx,\bfy) =  \thf \norm[2]{\bfA\bfx -\bfb}^2 +\mu \norm[1]{\bfy} \quad \mbox{subject to }  \bfD\bfx -\bfy = {\bf0}.
\end{equation}
Transforming from unconstrained \cref{eq:general_lasso} to  constrained problem \cref{eq:general_lassosplit} may at first seem counterintuitive, as the dimension of the problem increases from $\bbR^n$ to $\bbR^{n+\ell}$. The advantage, however, lies in the ``simplicity'' of the subsequent subproblems arising from the particular form of $f$, when this constrained problem is solved utilizing common numerical approaches for constrained optimization.

We now follow the well-established \emph{augmented Lagrangian} framework, see~\cite{hestenes1969multiplier, powell1969method} and~\cite[Chapter 17]{nocedal2006numerical} for details, in which the Lagrangian function of the constrained optimization problem~\cref{eq:general_lassosplit} is accompanied by a quadratic penalty term:
\begin{equation}\label{eq:general_lassoaug}
    \calL_{\rm aug}(\bfx,\bfy,\bfz; \lambda) =   f(\bfx,\bfy) + \bfz\t (\bfD\bfx -\bfy) + \tfrac{\lambda^2}{2} \norm[2]{\bfD\bfx -\bfy}^2,
\end{equation}
where $\bfz\in \bbR^\ell$ denotes the vector of Lagrange multipliers and  $\lambda$ is a scalar penalty parameter. Merging the last two quadratic terms in \cref{eq:general_lassoaug} gives
\begin{equation}\label{eq:augL}
    \calL_{\rm aug}(\bfx,\bfy,\bfc; \lambda) =  f(\bfx,\bfy) + \tfrac{\lambda^2}{2} \norm[2]{\bfD\bfx -\bfy + \bfc}^2 -\tfrac{\lambda^2}{2}\norm[2]{\bfc}^2,
\end{equation}
where $\bfc = \bfz/\lambda^2$ is a scaled Lagrangian multiplier.
Now, for fixed $\bfc$, \cref{eq:augL} can be solved using an \emph{alternating direction} approach to minimize the augmented Lagrangian $\calL_{\rm aug}(\bfx,\bfy,\bfc; \lambda)$ with respect to $\bfx$ and $\bfy$. Algorithmically, the minimizers $\bfx_{k},\bfy_{k}$ can be used to warm-start the numerical optimization method for iteration $k+1$. Further, using a sequence of increasing $\lambda_{k}$ reduces and ultimately eliminates the violation of the equality constraint $\bfD\bfx_k- \bfy_k = {\bf0}$ for sufficiently large $\lambda_k$. While adopting an increasing sequence of $\lambda_k$ leads to an ill-conditioned problem that suffers from numerical instability \cite{wang2019global,gabay1976dual}, it has been shown that, within the augmented Lagrangian framework, it is sufficient and computationally advantageous to explicitly estimate the scaled Lagrange multipliers $\bfc_k$, assuming constant $\lambda_k = \lambda$~\cite{bertsekas2014constrained}.

From the first-order optimality condition, we expect
\begin{equation}
    \nabla_{\bfx,\bfy} \ \calL_{\rm aug}(\bfx_{k+1},\bfy_{k+1},\bfc_k; \lambda)\approx {\bf0},
\end{equation}
whenever $\bfx_{k+1},\bfy_{k+1}$ approximately minimizes  $\calL_{\rm aug}(\,\cdot\,,\,\cdot\,,\bfc_k; \lambda)$. But noting that
\begin{equation}
\nabla_{\bfx,\bfy} \ \calL_{\rm aug}(\bfx,\bfy,\bfc_k; \lambda) =  \nabla_{\bfx,\bfy} f(\bfx,\bfy) +
\lambda^2\begin{bmatrix}
      \bfD\t\\
      -\bfI_\ell
\end{bmatrix} (\bfD\bfx -\bfy + \bfc_k),
\end{equation}
we see that $\bfD\bfx_{k+1} -\bfy_{k+1} + \bfc_k$ approximates the corresponding Lagrange multipliers. Hence  we update
\begin{equation}\label{eq:cupdate}
    \bfc_{k+1} = \bfD\bfx_{k+1} -\bfy_{k+1} + \bfc_k,
\end{equation}
and the \emph{method of multipliers} can be summarized as follows \cite{hestenes1969multiplier,powell1969method}. For initial $\bfc_0\in \bbR^\ell$ and a choice of $\lambda$, we iterate over
\begin{subequations}\label{eq:augLag}
\begin{align}
    (\bfx_{k+1}, \bfy_{k+1}) &= \argmin_{\bfx,\bfy} \calL_{\rm aug}(\bfx,\bfy,\bfc_k;\lambda) \label{eq:lagmin}\\
       \bfc_{k+1} &= \bfD\bfx_{k+1} -\bfy_{k+1} + \bfc_k
\end{align}
\end{subequations}
until convergence is achieved.

Certainly, the main computational effort associated with implementing~\cref{eq:augLag} for large-scale problems lies in solving \cref{eq:lagmin}. Hence, splitting the optimization problem \cref{eq:lagmin} with respect to $\bfx$ and $\bfy$ and optimizing with an \emph{alternating direction} approach offers the potential for reduced computational complexity. It is this final step that leads to the widely used \emph{alternating direction method of multipliers} (ADMM), where $\bfx_{k+1}$ and $\bfy_{k+1}$ in \cref{eq:lagmin} are approximately obtained by performing one alternating direction optimization
\begin{subequations}\label{eq:admm}
\begin{align}\label{eq:admm_x}
    \bfx_{k+1} &= \argmin_{\bfx} \ \calL_{\rm aug}(\bfx,\bfy_{k},\bfc_k;\lambda) \\ \label{eq:admm_y}
    \bfy_{k+1} &= \argmin_{\bfy} \ \calL_{\rm aug}(\bfx_{k+1},\bfy,\bfc_k;\lambda).
\end{align}
\end{subequations}
The computational advantages  of the alternating direction method becomes apparent with a closer look at \cref{eq:admm_x} and \cref{eq:admm_y} and considering the specific form of \cref{eq:general_lassosplit}. As for \cref{eq:admm_x}, we notice (ignoring constant terms) that minimizing $\calL_{\rm aug}(\bfx,\bfy_{k},\bfc_k;\lambda)$ with respect to $\bfx$ reduces to a linear least-squares problem of the form
\begin{equation}\label{eq:admm_min_x}
    \bfx_{k+1} = \argmin_\bfx \ \thf \norm[2]{
    \begin{bmatrix}
    \bfA \\ \lambda \bfD
    \end{bmatrix} \bfx -
    \begin{bmatrix}
    \bfb \\ \lambda\left(\bfy_{k} - \bfc_k\right)
    \end{bmatrix}
    }^2.
\end{equation}
There are many efficient methods to solve this least-squares problem either directly or iteratively \cite{greenbaum1997iterative}. To obtain $\bfy_{k+1}$ in \cref{eq:admm_y}, we notice that the minimization of the objective function  $\calL_{\rm aug}(\bfx_{k+1},\bfy,\bfc_k;\lambda)$ with respect to $\bfy$ for the particular choice of $f$ in \cref{eq:general_lassosplit} reduces to
\begin{equation} \label{eq:shrink}
    \bfy_{k+1} = \argmin_\bfy \  \mu \norm[1]{\bfy}+ \tfrac{\lambda^2}{2} \norm[2]{\bfd_{k+1} -\bfy}^2,
\end{equation}
where $\bfd_{k+1} = \bfD\bfx_{k+1} + \bfc_k$. \Cref{eq:shrink} is a well-known \emph{shrinkage} problem that has the explicit solution
\begin{equation}\label{eq:vecshrink}
    \bfy_{k+1} = \sign{\bfd_{k+1}} \hadamard \left(\left|\bfd_{k+1}\right| -  \frac{\mu}{\lambda^2}\bf1_{\ell} \right)_+,
\end{equation}
for the element-wise function $(\mdot)_+$ defined by
\begin{equation}
    (w)_+ =
    \begin{cases}
        w, & \mbox{for }w>0,\\
        0, & \mbox{otherwise,}
    \end{cases}
\end{equation}
where $\hadamard$ denotes the Hadamard product, and $|\,\cdot\,|$ the element-wise absolute value. The update \cref{eq:vecshrink} is of low computational complexity, making the alternating direction optimization efficient. Combining \cref{eq:admm_min_x,eq:vecshrink} with \cref{eq:cupdate}, yields the ADMM algorithm for the solution of \cref{eq:general_lasso} as summarized in \cref{alg:admm}.

\begin{algorithm}[htb!]
\caption{Alternating Direction Method of Multipliers (ADMM) \cite{gabay1976dual}}\label{alg:admm}
    \begin{algorithmic}[1]
    \State \textbf{input} $\bfA$, $\bfb$, $\bfD$, $\mu$, and $\lambda$
    \State initialize $\bfc_0=\bfx_0=\bfy_0={\bf0}$, and set $k = 0$
    \While{not converged}
        \State\label{ln:admm_xk} $\bfx_{k+1} =\argmin_{\bfx} \ \ \thf\norm[2]{\bfA\bfx -\bfb}^2 + \tfrac{\lambda^2}{2} \norm[2]{\bfD\bfx - \bfy_k+\bfc_{k}}^2$
        \State\label{ln:admm_yk} $\bfy_{k+1} =\argmin_{\bfy} \ \ \mu\norm[1]{\bfy} + \tfrac{\lambda^2}{2} \norm[2]{\bfD\bfx_{k+1} - \bfy + \bfc_{k}}^2$
        \State\label{ln:admm_ck} $\bfc_{k+1} = \bfc_k + \bfD\bfx_{k+1} - \bfy_{k+1}$
        \State\label{ln:kupdate} $k = k+1$
    \EndWhile
    \State \textbf{output} $\bfx_k$
    \end{algorithmic}
\end{algorithm}

The computational complexity of a \emph{plain} ADMM algorithm, as described in \cref{alg:admm}, is $\calO(\widetilde K (m+\ell)n^2)$, where $\widetilde K$ refers to the required number of iterations. While various iterative techniques for the solution of \cref{eq:admm_min_x} have been discussed in the literature, including a generalized Krylov approach, e.g., \cite{BucciniADMM}, a standard approach is to use the LSQR algorithm, also based on a standard Krylov iteration \cite{paige1982lsqr}. This is our method of choice for comparison with our new projected algorithm, described in~\cref{sec:vpal}. Stopping criteria for ADMM are outlined in \cite{boyd2011distributed}.

Note that at each alternating direction optimization \cref{eq:admm} the matrices $\bfA$ and $\bfD$ in \cref{eq:admm_min_x} do not change. Hence, computational efficiency for ADMM, with fixed $\lambda$, can be improved by pre-computing a matrix factorization, e.g., a Cholesky decomposition (as implemented in Matlab's {\tt lasso} function) of $\bfA\t\bfA + \lambda^2\bfD\t\bfD$ and using this factorization to solve the optimization problem \cref{eq:admm_min_x} . This initial computational effort is especially effective, making the update of $\bfx$ computationally efficient. The theoretical computational cost reduces to about $\calO(\widetilde K (m+l)n)$. However, the use of a precomputed factorization may come at a price. For instance, for large-scale problems as they appear in feature selection, sparse representation, machine learning, and dictionary learning, as well as total variation (TV) problems \cite{brunton2019data}, it is not immediate that a precomputed matrix decomposition can be performed. In particular, this may not be feasible when the matrices $\bfA$, $\bfD$ are too large ($m$, $n$, or $\ell$ large) or are only available as linear mappings, e.g., $\bfA:\bbR^n \to \bbR^m$. Another drawback of finding solutions of equations with the system matrix $\bfA\t\bfA + \lambda^2\bfD\t\bfD$ is that the ill-posedness of the problem can be accentuated due to the squaring of the condition number as compared to that of $\bfA$ alone, when $\lambda$ is not chosen appropriately. Before presenting an alternative approach based on a variable projection in \cref{sec:vpal} we turn to a discussion on regularization parameter selection methods in \cref{sec:reg}.

\section{Regularization parameter selection}\label{sec:reg}
We briefly overview standard methods to find regularization parameters in \cref{sec:regbackground}. Approaches in the context of image restoration for the generalized lasso formulation are presented in \cref{sec:meadoptmu}, and a  discussion of applying a bisection algorithm for this problem is provided in \cref{sec:bisect}.

\subsection{Background}\label{sec:regbackground}
There is a substantive literature on determining the regularization parameter $\mu$ in \cref{eq:Tik}, for which details are available in texts such as \cite{hansen1998rank,hansen2010discrete,bardsley2018computational}. These range from techniques that do not need any information about the statistical distribution of the noise in the data, such as the L-curve that trades off between the data fit and the regularizer \cite{hansen2010discrete}, the method of generalized cross-validation \cite{GoHeWa} and supervised learning techniques \cite{osher2005iterative,afkham2021learning}.  Some other approaches are  statistically-based and require that an estimate of the variance  of the noise in the data, assumed to be normal, is known. These yield techniques such  as the Morozov Discrepancy Principle (MDP) and Unbiased Predicative Risk Estimation, \cite{Morozov1966,Vogel:2002}. The MDP is based on an assumption that the optimal choice for $\mu$ yields a residual which follows a $\chi^2$ distribution with $m$ degrees of freedom (DF). It has also been shown that another choice for $\mu$ is the one for which the augmented residual follows a $\chi^2$ distribution but with a change in the DF. In this case the DF depends on the relative sizes of $\bfA$ and $\bfD$, and their ranks, for details see \cite{mead2008newton,RHM10}.

While it is generally computationally feasible to find optimal estimates for $\mu$ for the ridge regression problem, the lack of convexity of the generalized lasso, and the associated computational demands for large scale problems \cref{eq:general_lasso} presents challenges in both defining a method for optimally selecting $\mu$ and with finding an optimal estimate efficiently. On the other hand, techniques that use arguments based on the solution's \emph{degrees of freedom}, rather than a residual DF, have paved a way to identifying an optimal $\mu$ for the lasso problem, $\bfD = \bfI_n$, as discussed in the statistical literature,  \cite{Efron2,Efronetal,TibshiraniTaylor:2011,TibshiraniTaylor,dittmer2020regularization,Zouetal}. A DF argument for the solution of the generalized lasso problem when $\bfA$ is a blurring operator of an image and $\bfD\ne \bfI_n$ was also recently provided by Mead, \cite{Mead_2020}. Given an efficient algorithm to solve the generalized lasso problem \cref{eq:general_lasso}, it becomes computationally feasible to find an optimal $\mu$ provided the conditions under which the DF arguments apply are satisfied. Here, we focus on implementing an efficient estimator for $\mu$ for the generalized lasso problem, equipped with our efficient {\tt vpal} solver.

\subsection{The optimal \texorpdfstring{$\mu$ }{mu }for total variation image deblurring}\label{sec:meadoptmu}

As observed in \cite{Green:02}, for naturally occurring images, the TV functional defined by $\bfy = \bfD\bfx$ is Laplace distributed\footnote{A random variable $y$ follows a  Laplace  distribution with mean $\theta$ and variance $2\beta^2$, denoted $y \sim \calL(\theta, 2\beta^2 )$, if its probability density function is
$y = \tfrac{1}{2\beta} \exp\left(\tfrac{|y-\theta|}{\beta}\right)$.}~with mean $\bftheta\in \bbR^\ell$ and variance $2\beta^2\bfI_\ell$, i.e., $\bfy \sim \calL(\bftheta, 2\beta^2 \bfI_\ell)$;  the underlying image  $\bfx$ is said to be  \emph{differentially Laplacian}.  Then, under the assumption $\bfvarepsilon\sim\calN({\bf0}, \sigma^2\bfI_m)$, the maximum a posteriori estimator (MAP) for $\widehat \bfx(\mu)$, here denoted by $\widehat\bfx(\mu_{\textrm{map}})$, is given by \cref{eq:general_lasso} when $\mu=\mumap=\sigma^2/\beta $ \cite[eq.~(9)]{Mead_2020}. This result relies on first determining that the image pixels are differentially Laplacian, and then in estimating a value for $\beta$, in both cases using $\bfb$ since $\bfx$ is not known. The statistical  distribution can be tested by applying $\bfD$ to a given image and forming its histogram, as suggested in \cite{Green:02}. Moreover, $\beta$ can be estimated using $\beta=\mathrm{std}(\bfD \bfb)/\sqrt{2}$ (or from a suitable subset of the image, or adjustment of $\bfD$ to the size of the image, when the dimensions are not consistent) \cite[Algorithm~1]{mead2008newton}, where $\rm std$ denotes the standard deviation. The difficulty with using $\mumap$  is that one  does need to find $\beta$ from the data, and even for naturally occurring images, it may not be effective to find $\beta$ using $\bfb$. Specifically,  for images that are significantly blurred, it is unlikely  that $\beta$ obtained from $\bfb$ is a good estimate for the true $\beta$ associated with the unknown image $\bfx$.

With the observed limitation in identifying $\beta$ from any given image,  \cite{Mead_2020} suggested the alternative direction that uses the DF argument of the solution. Briefly, and referring to \cite{Mead_2020} for more details, the argument relies on the connection between  independent Laplacian random variables and the $\chi^2$ distribution. Specifically, for $\bfy \sim \mathcal{L}({\bf 0}, 2\beta^2 \bfI_\ell)$, we have  $\tfrac{2}{\beta}\|\bfD \bfx\|_1 \sim \chi^2_{2\ell}$\footnote{We use $\chi^2_m$ to denote a sum that follows a $\chi^2$ distribution with $m$ degrees of freedom} \cite[Proposition 2 to 4]{Mead_2020}. While we know $\norm[2]{\bfA \bfx -\bfb}^2 \sim \chi^2_m$, the two terms in \cref{eq:general_lasso} are not independent.  To use the $\chi^2$ distributions of both terms together requires the result on the DF from  \cite{TibshiraniTaylor}, which then yields
\begin{align}\label{eq:muchi2}
 \tfrac{1}{\sigma^2}\norm[2]{\bfA\bfx -\bfb}^2 + \tfrac{2}{\beta} \norm[1]{\bfD\bfx} \sim \chi^2_m,
\end{align}
\cite[Theorem 4]{Mead_2020}. This still uses $\beta$, but given $\sigma^2$ from the data, and without calculating $\beta$, \cref{eq:muchi2} suggests that $\mu$ can be chosen based on the $\chi^2$ test that fits the $\chi^2$ degrees of freedom
\begin{equation}\label{eq:dofmead}
\norm[2]{\bfA \widehat\bfx(\mu) -\bfb}^2 + \mu \|\bfD \widehat\bfx(\mu)\|_1 \cong  m \sigma^2,
\end{equation}
when both $\bfA$ and $\bfD$ have full column rank.  We define $\muchi$ to be the $\mu$ that satisfies \cref{eq:dofmead}.

While the results presented in \cite{Mead_2020} demonstrate that use of the TV-DF estimator in \cref{eq:dofmead} is more effective than use of $\mumap$, \cite[Algorithm~2]{Mead_2020}, as we have also confirmed from results not presented here, it must be noted that we cannot immediately apply this estimator for matrices $\bfA$ that arise in other models, such as projection, i.e., $\bfA$ is not a smoothing operator. On the other hand, for the restoration of blurred images, a bisection algorithm can be used to find $\muchi$ using \cref{eq:dofmead}.
Here we give only the necessary details, but note that the  approach provided is feasible when the covariance of the noise in the data is available. We assume that the noise is Gaussian normal. For colored noise a suitable whitening transform can be applied to modify \cref{eq:general_lasso}, effectively, using a weighted norm in the data fit term, see e.g., \cite{mead2008newton}. We reiterate that the approach is directly applicable only for the restoration of naturally occurring images, \cite{Green:02}.

\subsection{A bisection algorithm to find \texorpdfstring{$\mu$}{mu}}\label{sec:bisect}
Although the importance of \cref{eq:dofmead} was presented in \cite{Mead_2020}, no efficient algorithm to find $\muchi$ was given. Defining
\begin{equation}\label{eq:dof}
F(\mu)=\norm[2]{\bfA \widehat\bfx(\mu) -\bfb}^2 + \mu \norm[1]{\bfD \widehat\bfx(\mu)}
\end{equation}
implies $F(\mu)>0$. But, under the assumption that $\norm[1]{\bfD \widehat \bfx(\mu)}$ is bounded, as $\mu\rightarrow 0$ the function $F(\mu)$ decreases, see \cite[Proposition~4]{Mead_2020}.  Consequently, for any choice of $p>0$,  $F(\mu)-p\sigma^2$ decreases, and may become negative if $p\sigma^2$ is sufficiently large.  When the estimates  for $p$ and $\sigma^2$ are accurate we can expect, for sufficiently large $\mu$,  that $F(\mu)-p\sigma^2>0$. Noting now, given an algorithm to find solutions $\widehat\bfx(\mu)$, that we can also calculate $F(\mu)$ as given by \cref{eq:dof}, then we can also seek $\muchi$  either as the root of $F(\mu)-p\sigma^2=0$ or by minimization of $\left| F(\mu)-p\sigma^2\right|$.

Here we adapt a standard root-finding algorithm by bisection to find $\muchi$,  assuming that the estimates for both DF and $\sigma^2$ are accurate, and that $\norm[1]{\bfD \widehat \bfx(\mu)}$ is bounded. Given initial estimates for $\mumin$ and $\mumax$ such that $F(\mumin)-p\sigma^2 < 0 < F(\mumax) -p\sigma^2 $, we may apply bisection to find $0< \mumin<\muchi<\mumax$.  In general we use $\mumax=2\|\bfA\t\bfb\|_{\infty}$. This is now where an estimate of $\mumap$ becomes helpful, even when it may be generally inadequate as the actual estimator for a good choice of $\mu$, it can be used to suggest an interval in which a suitable  $\muchi$ exists. Thus given   $\mumap$ we can define an interval for $\muchi$ by using $\mumin=10^{-q}\mumap$ and $\mumax=10^q\mumap$ for a suitably chosen $q$ such that $F(\mumin)-p\sigma^2 <0< F(\mumax)-p\sigma^2 $.  In our bisection algorithm we utilize a logarithmic bisection.  Empirically we observe that this is more efficient for getting close to a suitable root, given that the range for suitable $\mu$ may be large.  We further note, as observed in \cite[Remarks~1 and 2]{Mead_2020}, that the bisection relies not only on finding a good interval for $\mu$ such that $F(\mu)-p\sigma^2$ goes through zero, but also, as already stated, on the availability of good estimates for $p$ and $\sigma^2$.
These limitations (i.e., providing statistics on the data and an estimate of the degrees of freedom) are the same as the ones arising when using a DF argument in the standard MDP, or for the augmented residual in \cref{eq:Tik}, \cite{Morozov1966,mead2008newton}.

Bisection is terminated when one of the following conditions is satisfied: (i) $\mumax-\mumin < \tau_2(1+|\mumin|)$;  (ii) $|\widetilde{F}(\mumax)-\widetilde{F}(\mumin)|<\tau_2$; (iii) $\mumax-\mumin< \tau_1$ or (iv) a maximum number of total function evaluations are reached.  First, note that the limit on function evaluations is important because each evaluation requires finding $\widehat\bfx(\mu)$, i.e., solving \cref{eq:general_lasso} for a given choice of $\mu$. Second, the estimate  $\tau_1$ determines whether or not a tight estimate for $\muchi$ is required.

It has been shown in \cite{Efronetal} that the family of solutions for \cref{eq:general_lasso} is piecewise linear, namely the regularization path for solutions as $\mu$ varies has a piecewise linear property in $\mu$. Indeed, path-following on $\mu$ has been used to analyze the properties of the solution with $\mu$ and to design algorithms that determine the intervals \cite{ali2019generalized,Dossaletal:13,Zouetal,TibshiraniTaylor:2011,4407767}. Here, we do not propose to apply any path-following approach, rather we just rely on the existence of the intervals, to assert that refining the interval $\mu$ will have little impact on the solution.

The parameter $\tau_2$, in contrast, provides a relative error bound on $\mu$ that is relevant for large $\mu$. It also  permits adjustment of a standard bisection algorithm to reflect the confidence in the knowledge of $p$ and/or $\sigma^2$. For example, the MDP is a $\chi^2$ test on the satisfaction of the DF~in the data fit term for Tikhonov regularization, and is often adjusted by the introduction of a safety parameter $\eta$. Then, rather than seeking $\norm[2]{\bfA\bfx_{\mathrm{tik}}-\bfb}^2 = m\sigma^2$, the right-hand side is adjusted to $\eta m \sigma^2$ (Here $\bfx_{\mathrm{tik}}$ is the solution of \cref{eq:Tik} dependent on the parameter $\mu$).  In the same manner, we may adjust $\tau_2$ using the safety parameter $\eta$  as a degree of confidence on the variance $\sigma^2$ or the DF.  Since we solve relative to $p \sigma^2$, we can adjust using  $\eta p\sigma^2$ for safety on $\sigma^2$ or using $\eta=(p+\zeta)/p$ when considering a confidence interval on the DF. For example,  replacing $p$ by $p+ \zeta$ for $\zeta=\textrm{confidence}(p,0.05)$ represents a $95\%$ confidence interval in the $\chi^2$ distribution, as used when applying the $\chi^2$ estimate for the augmented residual \cite{mead2008newton}. A choice based on a confidence interval is more specific than an arbitrarily chosen  $\eta$.

It should also be noted that the solution $\widehat\bfx(\mu)$ is defined within the algorithm for a given choice of the shrinkage parameter $\gamma$, which should not be changed during the bisection. This means that as $\mu$ increases, with soft shrinkage parameter $\gamma=\mu/\lambda^2$ fixed, $\lambda$ also increases, and hence the result for a specific $\muchi$ is dependent on a given shrinkage threshold.

\section{Variable Projected Augmented Lagrangian}\label{sec:vpal}

As discussed in~\cref{sec:ADMM} the computational efficiency of~\cref{alg:admm} is dominated by the effort of repeatedly solving the least-squares problem \cref{eq:admm_min_x} at step~\ref{ln:admm_xk} of \cref{alg:admm}. The focus, therefore, of our approach is to  significantly reduce the computational cost that arises in obtaining the updates in~\cref{eq:lagmin}. Specifically, for the solution of \cref{eq:general_lasso}, we propose to utilize the augmented Lagrangian approach as derived in \cref{eq:general_lassosplit}--\cref{eq:augLag}.  However, rather than  performing an alternating direction optimization as in \cref{eq:admm}, we adopt \emph{variable projection} techniques as presented in \cref{eq:varpro} to solve \cref{eq:lagmin}.

Calculating $\bfy$ exactly while keeping $\bfx$ constant is computationally tractable due to the use of the shrinkage step~\cref{eq:vecshrink}. On the other hand, gains in computational efficiency can be achieved by exploiting an inexact solve for $\bfx$. Consequently, we propose a variable projection approach with an \emph{inexact} solve to update $\bfx$ while optimizing for $\bfy$. This new perspective of utilizing variable projection resonates with inexact solves of the linear system within ADMM and other common methods \cite{goldstein2009split}.  While there are various efficient update strategies that may be utilized to find $\bfx$, such as LBFGS or Krylov subspace type methods \cite{nocedal2006numerical,paige1982lsqr}, we propose updating $\bfx$ by performing a single conjugate gradient (CG) step. Specifically, we use the update
\begin{equation}\label{eq:xupdate}
    \bfx^{(j+1)} = \bfx^{(j)} - \alpha_j \bfg_j,
\end{equation}
where $\bfg_j$ is a vector of length $n$ given by
\begin{equation}\label{eq:gupdate}
   \bfg_j = \begin{bmatrix}
   \bfA\t & \lambda\bfD\t\end{bmatrix} \bfr_j \qquad \mbox{with} \qquad \bfr_j = \begin{bmatrix} \bfA\bfx^{(j)} - \bfb\\ \lambda\bfD\bfx^{(j)} - \lambda \left(\bfy^{(j)}-\bfc_k\right)\end{bmatrix},
\end{equation}
and the optimal step length $\alpha_j$ may be computed by
\begin{equation}\label{eq:alphaupdate}
    \alpha_j = \argmin_\alpha \, f_{\rm proj}(\bfx^{(j)}-\alpha \bfg_j),
\end{equation}
where the projected function $f_{\rm proj}:\bbR^n\to\bbR$ is defined in \cref{eq:fproj} \cite{hestenes1952methods,paige1982lsqr}. Adopting this update strategy yields the new \emph{variable projected augmented Lagrangian} (VPAL) method for solving \cref{eq:general_lasso} as summarized in~\cref{alg:vpal}.

\begin{algorithm}[htb!]
\caption{Variable Projected Augmented Lagrangian (VPAL)}\label{alg:vpal}
    \begin{algorithmic}[1]
    \State \textbf{input} $\bfA$, $\bfb$, $\bfD$, $\mu$, $\lambda$
    \State initialize $\bfc_0=\bfx_0=\bfy_0={\bf0}$, and set $k = 0$
    \While{not converged}\label{ln:outer_start_vpal}
        \State set $j = 0$, $\bfx^{(0)} = \bfx_{k}$, $\bfy^{(0)} = \bfy_{k}$
        \While{not converged}\label{ln:inner_start_vpal}
            \State calculate residual $\bfr_j = \begin{bmatrix} \bfA\bfx^{(j)} - \bfb\\ \lambda\bfD\bfx^{(j)} - \lambda \left(\bfy^{(j)}-\bfc_k\right)\end{bmatrix}$,
            \State calculate direction $\bfg_j = \begin{bmatrix}
   \bfA\t & \lambda\bfD\t\end{bmatrix} \bfr_j$
            \State
            set $\alpha_j = \argmin_\alpha \, f_{\rm proj}(\bfx^{(j)}-\alpha \bfg_j)$
            \vspace*{0.5ex}
           \State update  $\bfx^{(j+1)} = \bfx^{(j)} - \alpha_j \bfg_j$\label{ln:inner_x_vpal}\vspace*{0.5ex}
            \State\label{ln:vpal_yk}update  $\bfy^{(j+1)} = \sign{\bfD\bfx^{(j+1)} + \bfc_k} \hadamard \left(\left|\bfD\bfx^{(j+1)} + \bfc_k\right| -  \mu/\lambda^2 \bf1_{\ell} \right)_+$
            \State$j = j+1$
        \EndWhile   \label{ln:inner_end_vpal}
        \State set $\bfx_{k+1} = \bfx^{(j)}$ and $\bfy_{k+1} = \bfy^{(j)}$
        \State set $\bfc_{k+1} = \bfc_k + \bfD\bfx_{k+1} - \bfy_{k+1}$
        \State $k = k+1$
    \EndWhile\label{ln:outer_end_vpal}
    \State \textbf{output} $\bfx_k$
    \end{algorithmic}
\end{algorithm}

\smallskip

{\sc Remarks.} Within VPAL our ansatz for applying a variable projection technique is ``\emph{inverted}'' as compared to the standard approaches. Specifically, a typical variable projection method would seek to optimize over the variable that determines the linear least squares problem, here it would be $\bfx$, and would update the variable occurring nonlinearly, here $\bfy$, using a standard gradient or Hessian update, \cite{golub1973differentiation,o2013variable}. Here, we reverse the roles of the variables.
The idea to use an inexact solve for the update step \cref{eq:admm_min_x} within an ADMM algorithm is not new. For instance using just a few CG steps, or using other iterative methods, has been proposed and analyzed \cite{gol1979modified, eckstein1992douglas, hager2019inexact, teng2016admm}. In contrast, VPAL uses the variable projection \cref{eq:vecshrink} at each CG step.
\smallskip

Our convergence result for \cref{alg:vpal} is summarized in \cref{thm:vpal}. This result relies on the well-established convergence analysis of ADMM and the augmented Lagrangian method for the sequence of solves for $\bfx_k$  \cite{boyd2011distributed,hestenes1952methods,powell1969method,bertsekas1995nonlinear,bertsekas2014constrained}. Thus, the focus of the proof is to show that the solves obtained via variable projection corresponding to the inner loop, steps \ref{ln:inner_start_vpal} to \ref{ln:inner_end_vpal} of \cref{alg:vpal}, are sufficient to solve \cref{eq:lagmin}. Note, since $\bfc$ and $\lambda$ in \cref{eq:lagmin} are considered constant, an equivalent objective function to $\calL_{\rm aug} (\bfx,\bfy,\bfc;\lambda)$ is given by $f_{\rm joint}:\bbR^n\times \bbR^\ell \to \bbR$ with
\begin{equation}
\begin{aligned}\label{eg:fjt}
    f_{\rm joint}(\bfx,\bfy) & = \thf \norm[2]{\bfA\bfx -\bfb}^2 + \tfrac{\lambda^2}{2} \norm[2]{\bfD\bfx - \bfy + \bfc}^2+\mu \norm[1]{\bfy}. \\
\end{aligned}
\end{equation}

\begin{lemma}\label{lem:jtunique}
Suppose that $\bfA$ has full column rank. Then, for arbitrary but fixed $\bfc$ and $\mu, \lambda>0$, $f_{\rm joint}(\bfx,\bfy)$ is strictly convex and has a unique minimizer $(\widehat\bfx,\widehat\bfy)$.
\end{lemma}

\begin{proof}
We may rewrite \cref{eg:fjt} as
\begin{equation}
\begin{aligned}
    f_{\rm joint}(\bfx,\bfy) & = \thf \norm[2]{
        \begin{bmatrix}
            \bfA & {\bf0}_{m\times \ell} \\ \lambda \bfD & -\lambda \bfI_{\ell}
        \end{bmatrix}
        \begin{bmatrix}
            \bfx \\ \bfy
        \end{bmatrix}
        -
        \begin{bmatrix}
            \bfb \\ -\lambda \bfc
        \end{bmatrix}
        }^2 + \mu \norm[1]{
        \begin{bmatrix}
            {\bf0}_{\ell \times n} & \bfI_{\ell}
        \end{bmatrix}
        \begin{bmatrix}
            \bfx \\ \bfy
        \end{bmatrix}
        }.
\end{aligned}
\end{equation}
Now, since $\bfA$ is assumed to have  full column rank, $[\bfA , {\bf0}_{m\times \ell} ; \lambda \bfD , -\lambda \bfI_{\ell}]$ also has full rank. Hence,  $\thf \norm[2]{\bfA\bfx -\bfb}^2 + \tfrac{\lambda^2}{2} \norm[2]{\bfD\bfx - \bfy + \bfc}^2$ is strictly convex. Moreover, $\mu \norm[1]{\bfy}$ is convex in $\bfx$ and $\bfy$. Therefore the sum as given by  $f_{\rm joint}$ is strictly convex with a unique minimizer denoted by  $(\widehat\bfx,\widehat\bfy)$.
\end{proof}

We now introduce the continuous mapping $\bfZ:\bbR^n \to \bbR^\ell$ for the shrinkage step \cref{eq:vecshrink}, for fixed but arbitrary $\bfc$ and $\mu,\lambda>0$, given by
\begin{equation}\label{eq:shrinkmap}
    \bfZ(\bfx) = \sign{\bfD\bfx+\bfc} \hadamard \left(\left|\bfD\bfx+\bfc\right| -  \frac{\mu}{\lambda^2}\bf1_{\ell} \right)_+.
\end{equation}
Then the projected function that corresponds to $f_{\rm joint}$ is defined to be $f_{\rm proj}:\bbR^n \to \bbR$ with
\begin{equation}\label{eq:fproj}
        f_{\rm proj}(\bfx) =\thf \norm[2]{\bfA\bfx -\bfb}^2 +  \tfrac{\lambda^2}{2}\norm[2]{\bfD\bfx - \bfZ(\bfx) +\bfc}^2+\mu \norm[1]{\bfZ(\bfx)}.
\end{equation}

\begin{lemma}\label{lem:globalmini}
The following statements are true.
\begin{enumerate}
    \item\label{lem:globalmini_1} For any $\bfx$ there exists a $\bfy$ such that $f_{\rm proj}(\bfx) = f_{\rm joint}(\bfx,\bfy)$.
    \item\label{lem:globalmini_2} For any $\bfy$ and arbitrary $\bfx$ the inequality $f_{\rm proj}(\bfx) \leq f_{\rm joint}(\bfx,\bfy)$ holds.
    \item Let $(\widehat\bfx,\widehat\bfy)$ be the (unique) global minimizer of $f_{\rm joint}$. Then $\widehat\bfx$ is the (unique) global minimizer of $f_{\rm proj}$.
\end{enumerate}
\end{lemma}

\begin{proof}\,
\begin{enumerate}
    \item The first statement follows directly by the definitions, \cref{eg:fjt,eq:shrinkmap,eq:fproj}, i.e., if $\bfy = \bfZ(\bfx)$, then $f_{\rm proj} (\bfx) = f_{\rm joint}(\bfx,\bfy)$.
    \item
    Due to optimality of $\bfZ(\bfx) = \argmin_\bfy \ f_{\rm joint}(\bfx,\bfy)$ for arbitrary $\bfx$, we have $f_{\rm joint}(\bfx,\bfZ(\bfx))\leq f_{\rm joint}(\bfx,\bfy)$.
    \item By \cref{eg:fjt,eq:fproj} we have $f_{\rm proj}(\bfx) = f_{\rm joint}(\bfx,\bfZ(\bfx)) \geq f_{\rm joint}(\widehat\bfx,\widehat\bfy)$ for any $\bfx\in \bbR^n$. Thus  a global minimizer for $f_{\rm proj}$ is obtained at $\widehat\bfx$, i.e., $f_{\rm proj}(\widehat\bfx)=f_{\rm joint}(\widehat\bfx,\bfZ(\widehat\bfx))$. To show that the minimizer is unique we assume that there exists an $\widebar \bfx \neq \widehat\bfx$ that is also a global minimizer, i.e.,  $f_{\rm proj}(\widebar \bfx) = f_{\rm proj}(\widehat\bfx)$. By this assumption we have  that $f_{\rm joint}(\widehat\bfx,\widehat\bfy) = f_{\rm joint}(\widehat\bfx,\bfZ(\widehat\bfx)) = f_{\rm proj}(\widehat\bfx)= f_{\rm proj}(\widebar \bfx) = f_{\rm joint}(\widebar \bfx,\bfZ(\widebar \bfx))$.  Hence, $(\widebar \bfx,\bfZ(\widebar \bfx))\neq(\widehat\bfx,\widehat\bfy)$ is a global minimizer of $f_{\rm joint}$. But this contradicts the assumption that $(\widehat\bfx,\widehat\bfy)$ is the unique global minimizer of $f_{\rm joint}$, and thus $\widehat\bfx$ is the unique global minimizer of $f_{\rm proj}$.
\end{enumerate}
\end{proof}

It remains to be shown that there do not exist local minimizers of $f_{\rm proj}$, other than the one global minimizer given by $\widehat\bfx$.

\begin{lemma}\label{lem:nolocal}
Assume $\bfA$ has full column rank, $\bfc$ is fixed, and $\mu$, $\lambda>0$; then all minimizers of $f_{\rm proj}$ are global.
\end{lemma}

\begin{proof}
Note, $f_{\rm proj}$ has a local minimum at $\widebar\bfx$ if there exists a neighborhood $\calD(\widebar\bfx)$ of $\widebar\bfx$ such that $f_{\rm proj}(\bfx)\geq f_{\rm proj}(\widebar\bfx)$ for all $\bfx$ in $\calD(\widebar\bfx)$. Let $(\widehat\bfx,\widehat\bfy) = (\widehat\bfx, \bfZ(\widehat\bfx))$ denote the unique global minimum of $f_{\rm joint}$ according to \cref{lem:jtunique}. Since $f_{\rm joint}$ is strictly convex, $f_{\rm joint}$ is inevitably strictly convex in each of its components $\bfx$ and $\bfy$. Let $\widebar\bfy$ be arbitrary but fixed; then for any $\widebar\bfx\in \bbR^n$ with $\widebar\bfx \neq \widehat\bfx$ and $\epsilon>0$, there exists an $\bfx^*\neq\widebar\bfx$ with $\norm[2]{\bfx^*-\widehat\bfx}<\epsilon$ such that
\begin{equation}
    f_{\rm joint} (\bfx^*,\widebar \bfy) < f_{\rm joint} (\widebar\bfx,\widebar\bfy) = f_{\rm proj} (\widebar\bfx)
\end{equation}
due to the strict convexity in $\bfx$.  Further, by \cref{lem:globalmini}, \cref{lem:globalmini_1} we have  $f_{\rm proj}(\bfx^*) = f_{\rm joint} (\bfx^*,\bfZ(\bfx^*))$, and, by \cref{lem:globalmini}, \cref{lem:globalmini_2},   $f_{\rm joint} (\bfx^*,\bfZ(\bfx^*))\leq f_{\rm joint}(\bfx^*,\widebar\bfy)$. Together, we have  $$f_{\rm proj}(\bfx^*)= f_{\rm joint} (\bfx^*,\bfZ(\bfx^*)) \leq f_{\rm joint}(\bfx^*,\widebar\bfy) <
f_{\rm joint} (\widebar\bfx,\widebar\bfy) = f_{\rm proj} (\widebar\bfx). $$

Hence, there does \emph{not} exist a neighborhood $\calD(\widebar\bfx)$ around $\widebar\bfx$ for which $f_{\rm proj}(\bfx)\geq f_{\rm proj}(\widebar\bfx)$ for all $\bfx \in \calD(\widebar\bfx)$. Hence $\widebar\bfx\neq \widehat\bfx$ cannot be a local minimizer. For $\widebar\bfx = \widehat\bfx$ we have $f_{\rm proj}(\widehat\bfx) = f_{\rm joint}(\widehat\bfx,\bfZ(\widehat\bfx))$ and there is nothing to show.
\end{proof}

Taking the results of \cref{lem:globalmini,lem:nolocal} together we arrive at the main convergence result, in which we now consider the minimization at step $k$ of \cref{alg:vpal}.

\begin{theorem}\label{thm:vpal}
    Given optimization problem \cref{eq:general_lasso}, where $\bfA$ has full column rank, $\mu>0$, and $\lambda>0$ sufficiently large, \cref{alg:vpal} converges to the unique minimizer $\widehat\bfx$ of \cref{eq:general_lasso}.
\end{theorem}

\begin{proof}
As noted above, convergence of the outer loop (lines \ref{ln:outer_start_vpal} to \ref{ln:outer_end_vpal} of \cref{alg:vpal}) is well established for augmented Lagrangian methods, see \cite{bertsekas1995nonlinear,bertsekas2014constrained} for details. It remains to be shown that the variable projection corresponding to the inner loop, steps \ref{ln:inner_start_vpal} to \ref{ln:inner_end_vpal} of \cref{alg:vpal}, solves \cref{eq:lagmin}. \cref{lem:jtunique,lem:nolocal} ensure uniqueness of the minimizer $\widehat\bfx$ of $f_{\rm proj}$. Since
\begin{align}
    \mfg_\bfy= \lambda^2\left(-\bfD\bfx + \bfy -\bfc\right) + \mu\, (\sign{\bfy} \hadamard \bf1)
\end{align}
is the subgradient of $f_{\rm joint}$ with respect to $\bfy$ and due to the optimality for $\bfy = \bfZ(\bfx)$ we have
\begin{equation}
    \lambda^2\left(-\bfD\bfx + \bfZ(\bfx) -\bfc\right) + \mu\, (\sign{\bfZ(\bfx)} \hadamard \bf1) = \bf0.
\end{equation}
Therefore,
\begin{equation*}
\bfs
   = -\left(\bfA\t(\bfA\bfx -\bfb) + \lambda^2 \bfD\t\left(\bfD\bfx + \bfc -\bfZ(\bfx)\right) \right)
\end{equation*}
is a descent direction of $f_{\rm proj}$ corresponding to $-\bfg_j$ in \cref{alg:vpal}. Paired with an optimal step size $\alpha_j$ as defined in \cref{eq:alphaupdate}, an efficient descent is ensured until the subgradient optimality conditions are fulfilled.
\end{proof}

Notice that an optimal step size for the projected function $f_{\rm proj}$ as computed in \cref{eq:alphaupdate} is not required. Rather it may be sufficient to use an efficient step size selection that also provides a decrease along the gradient direction. Since $\bfy^{(j)} \approx \bfZ(\bfx^{(j)}-\alpha_j\bfg_j)$, especially in later iterations, we may replace this term in $f_{\rm proj}$, resulting in a linear least squares problem where the (linearized) optimal step size can be computed efficiently by
\begin{equation}\label{eq:alphaupdateLin}
    \alpha_j = \frac{\bfg_j\t\bfg_j}{\bfh_j\t\bfh_j}, \qquad \mbox{where} \qquad  \bfh_j= \begin{bmatrix}\bfA \\ \lambda\bfD \end{bmatrix}\bfg_j.
\end{equation}
Empirically, we observe that the linearized step size selection is very close to the optimal step size $\widehat \alpha$, approximately underestimated by about $10\%$, and is computationally less expensive (experiment not shown).

Hence, the dominant costs to obtain the update $\bfx^{(j+1)}$ in \cref{eq:xupdate} are the matrix-vector products needed to generate $\bfg_j$ and $\bfh_j$. There are two matrix-vector multiplications with $\bfA$, two with $\bfD$, and one each with their respective transposes. This corresponds to two multiplications with matrices of sizes $(m +\ell) \times n$ and $n \times (m+\ell)$ yielding a computational complexity of $\calO((m+\ell)n)$.

Hence, the computational complexity of \cref{alg:vpal} is $\calO(KJ(m+\ell)n)$, where $J$ refers to an average of the number of inner iterations for the inner while loop (steps \ref{ln:inner_start_vpal} to \ref{ln:inner_end_vpal} in \cref{alg:vpal}) and $K$ is the number of outer iterations.

In the practical implementation of the variable projection augmented Lagrangian method \vpal, we eliminate the while loop over $\bfx^{(j)}$ and $\bfy^{(j)}$ (steps \ref{ln:inner_start_vpal} to \ref{ln:inner_end_vpal} in \cref{alg:vpal}) and only perform a single update of $\bfx_k$ and $\bfy_k$. This reduces the complexity within the algorithm of determining suitable stopping conditions for carrying out the inexact solve, and reduces the computational complexity of a practical implementation of \vpal~to $\calO(K(m+\ell)n)$. We observe fast convergence as highlighted in \cref{sec:numerics}.
To terminate the outer $k$ iteration we use stopping criteria that are adapted from standard criteria for terminating iterations in unconstrained optimization, i.e., given a user-specified tolerance $\tau>0$, we terminate the algorithm when both $f(\bfx_{k},\bfy_{k}) - f(\bfx_{k+1},\bfy_{k+1}) \leq \tau (1+ f(\bfx_{k+1},\bfy_{k+1}))$ and $\norm[\infty]{\bfx_{k} - \bfx_{k+1}}\leq \sqrt{\tau}(1+\norm[\infty]{\bfx_{k+1}})$~\cite{gill2019practical}.
Matlab source code of the \vpal~method is available at \href{www.github.com/matthiaschung/vpal}{\texttt{www.github.com/matthiaschung/vpal}}\footnote{{\tt vpal} code and demo available upon acceptance.}.

\section{Numerical experiments}\label{sec:numerics}

In the following we perform various large-scale numerical experiments to demonstrate the benefits of {\tt vpal}. We first demonstrate the performance of {\tt vpal} in comparison to a standard {\tt admm} method on a denoising example in \cref{sec:denoise}. We investigate its scalability in 3D medical tomography inversion in \cref{sec:tomo}, and discuss regularization parameter selection methods in \cref{sec:regpara}. Key metrics for our numerical investigations are the relative error and relative residual defined by  $e(\bfx_k) = \frac{\norm[2]{\bfx_k - \bfx_{\rm true}}}{\norm[2]{\bfx_{\rm true}}}$ and  $r(\bfx_k) = \frac{\norm[2]{\bfA\bfx_k - \bfb}}{\norm[2]{\bfb}}$. Likewise when using $e(\widehat\bfx(\mu,\lambda))$ this refers to the converged (or final) value for the error in the solution $\widehat\bfx$ for a given parameter set $(\mu,\lambda)$.

\subsection{Denoising experiment}\label{sec:denoise}
\begin{figure}
    \centering
    \includegraphics[width=0.495\textwidth]{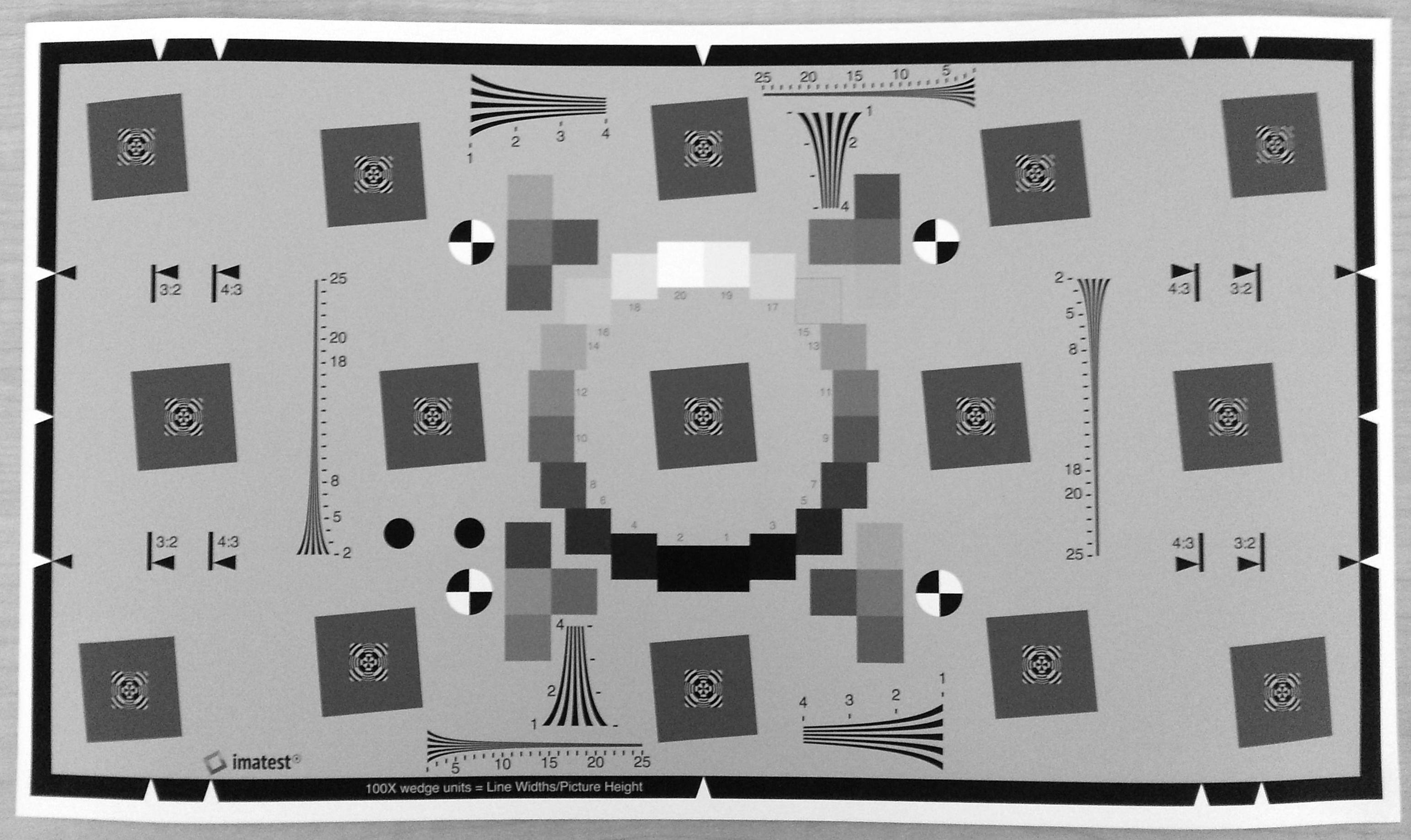}
    \includegraphics[width=0.495\textwidth]{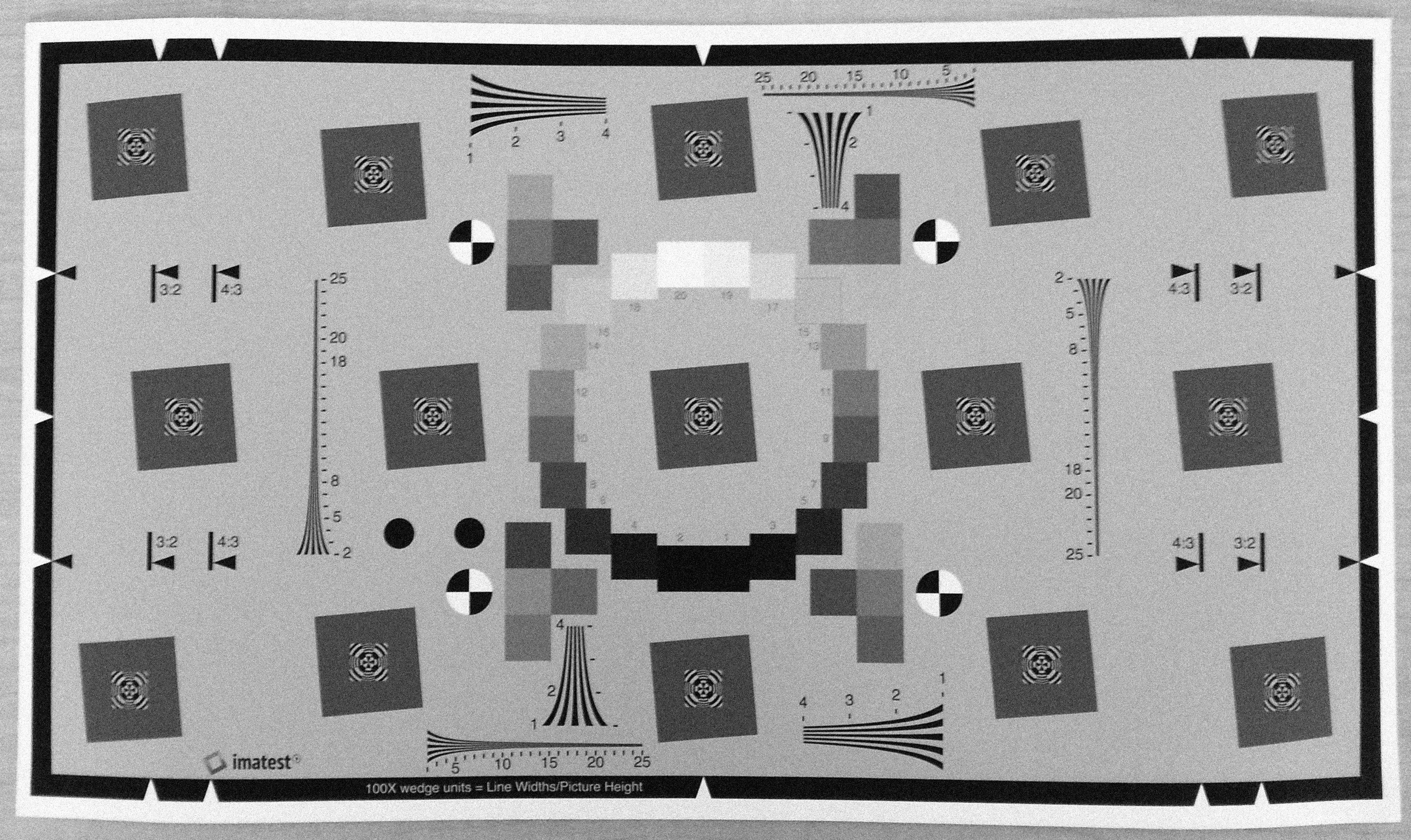}
    \caption{On the left we show Matlab's test image {\tt eSFRTestImage.jpg} representing $\bfx_{\rm true}$, on the right we depict noisy observations represented by $\bfb$ with $10$\% Gaussian white noise which yields an SNR of $20$.}
    \label{fig:experiment1a_images}
\end{figure}
In a first experiment we investigate the convergence of {\tt vpal}, compared to {\tt admm}. We consider an image denoising problem with $\bfA = \bfI_n$, while $\bfb$ is a Matlab test image {\tt eSFRTestImage.jpg} of size $1,\!836 \times 3,\!084$, so that $\bfx_{\rm true}$ is of size $n = 16,\!986,\!672$, see \cref{fig:experiment1a_images}, left panel. We contaminate $\bfx_{\rm true}$ with $10$\% Gaussian white noise, which yields an SNR of $20$,  see \cref{fig:experiment1a_images} right panel. Here, we utilize a total variation regularization where $\bfD$ is the 2D finite difference matrix. We fix the regularization parameter as $\mu = 10$, select a stopping tolerance $\tau = 10^{-4}$, and observe the relative error   $e(\bfx_k)$ and relative residual $r(\bfx_k)$ within iterations $k$ for both methods. The final approximations are denoted by $\bfx_{\tt vpal}$ and $\bfx_{\tt admm}$ with reconstruction errors $e(\bfx_{\tt vpal}) = 2.890\cdot 10^{-2}$ and $e(\bfx_{\tt admm}) = 2.887\cdot 10^{-2}$, respectively. Results are displayed in \cref{fig:experiment1_err}, where it is shown that {\tt admm} reaches the solution in fewer iterations than {\tt vpal}, $28$ to $38$ iterations, respectively. The number of outer iterations used  in \cref{alg:admm,alg:vpal} is, however, an ambiguous computational currency. Following the discussion above, the main computational cost of {\tt admm} is the number of LSQR iterations utilized for each outer iteration $k$.

Notice {\tt vpal} has the equivalent computational complexity of only one LSQR solve per iteration $k$. Factoring in these computational costs, we notice that {\tt admm} overall requires $141$ LSQR iterations while {\tt vpal} requires $38$ computations, indicating almost a factor of four improvement in the total number of LSQR iterations. Consequently we expect {\tt vpal} to converge more rapidly.

\begin{figure}
    \centering
    \includegraphics[width=0.495\textwidth]{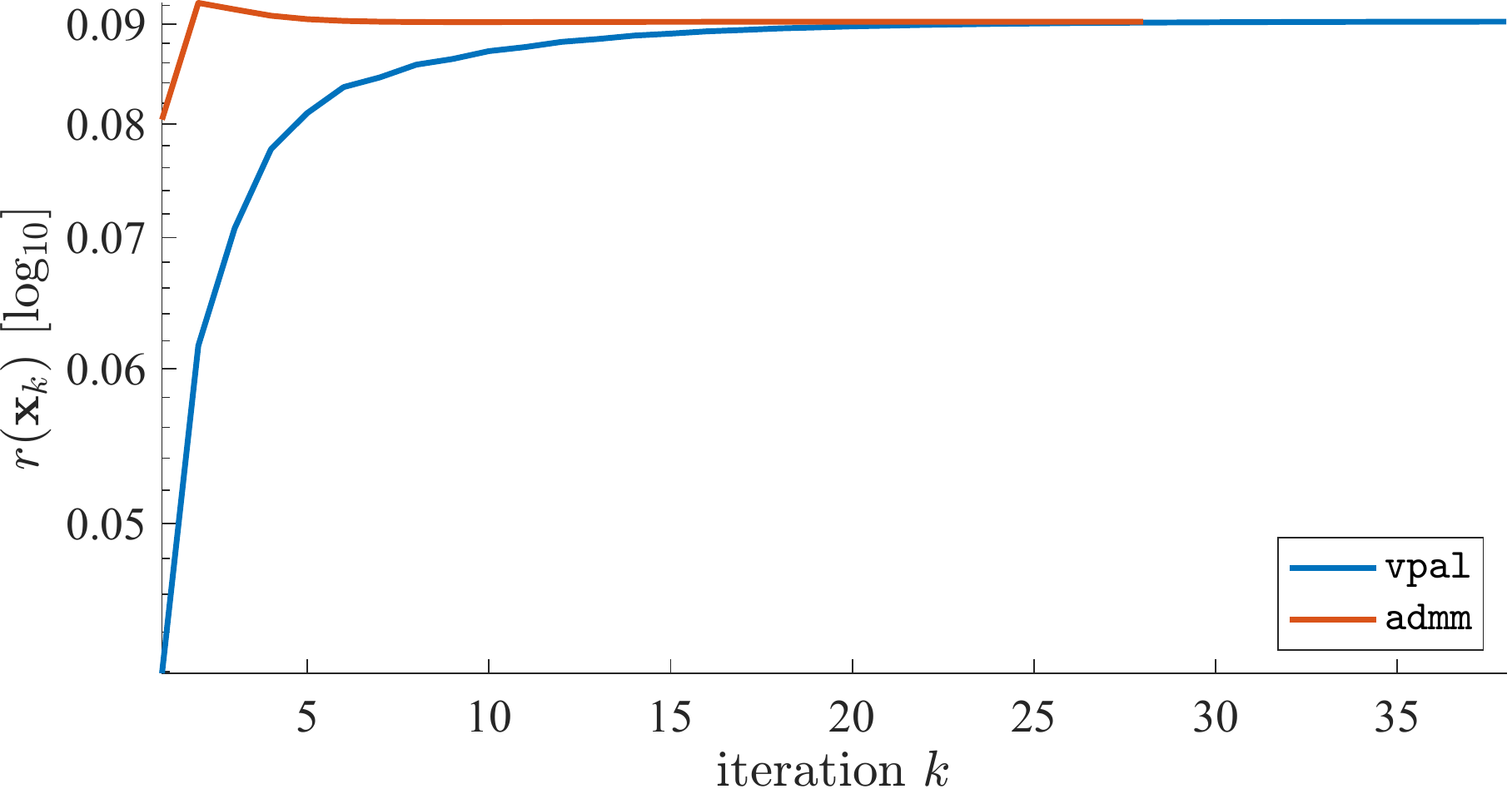}
    \includegraphics[width=0.495\textwidth]{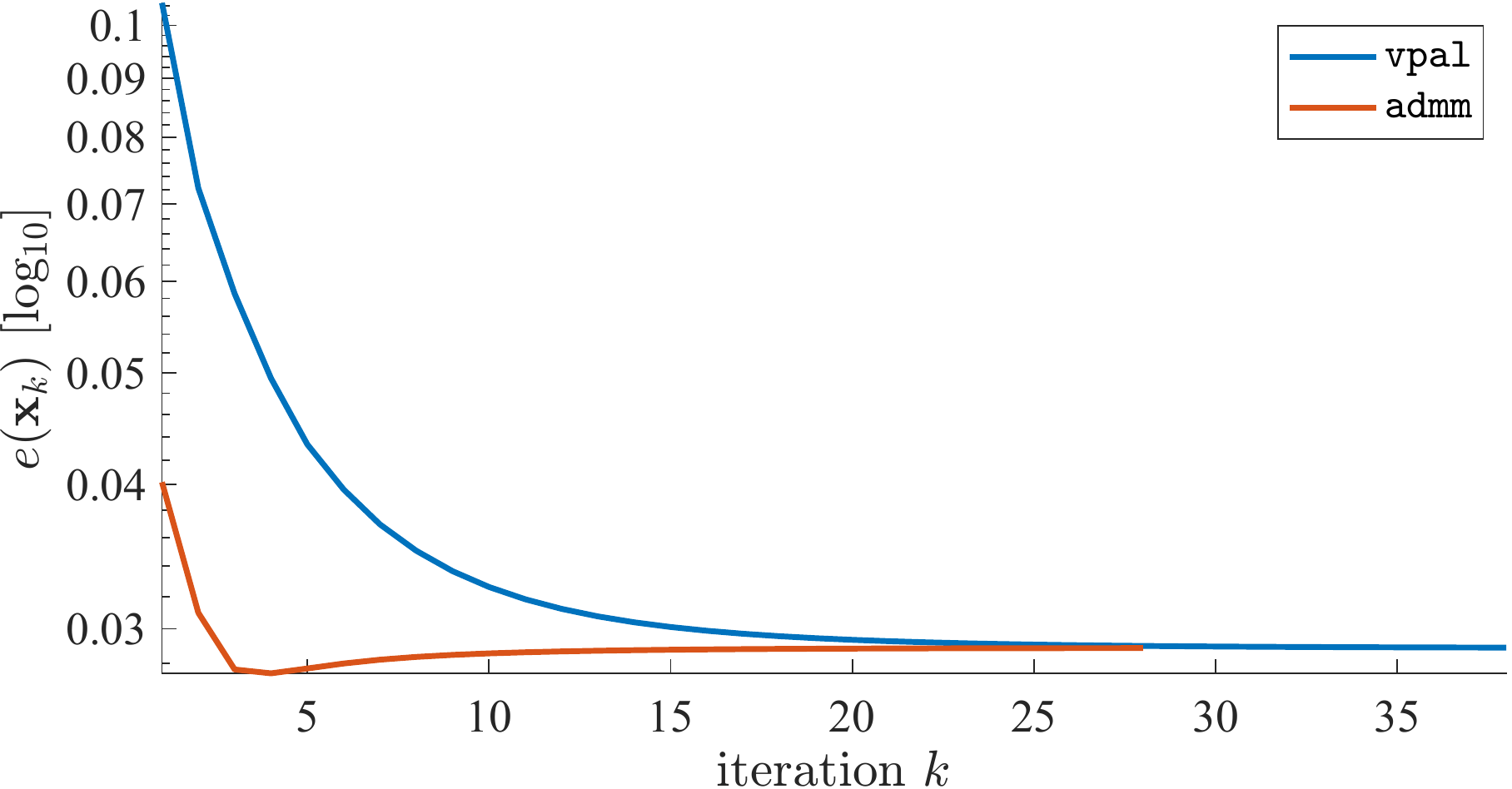}
    \caption{The left panel displays the relative residual  , while the right panel depicts the relative error   for {\tt vpal} in blue and {\tt admm} in red. Both methods converge to an approximation of $\bfx_{\rm true}$ where the relative difference between {\tt admm} and {\tt vpal} is negligible.}
    \label{fig:experiment1_err}
\end{figure}

Along this line, we extend our investigations and examine wallclock times, see \cref{fig:experiment1b_timing}. Using the same computational setup as before, we consider varying sizes of the image in \cref{fig:experiment1a_images}, left panel. Note that in this setup we also slightly vary the regularization parameter $\mu \in [1,20]$ to obtain more realistic timings for ``near optimal'' regularization parameters. Additionally, in our comparison we include timings for a generalized Tikhonov approach of the form \cref{eq:Tik}. Here this is referred to as {\tt tik}, and can be see as a computational lower bound for {\tt vpal} and {\tt admm} (left). The Tikhonov problem is solved using LSQR. Notice our comparison to a generalized Tikhonov approach comes with an asterisk, because (1) the range of good regularization parameters is largely different (to obtain close to optimal regularization parameters we choose $\widetilde \mu \in [0.1,0.4]$) and (2) the Tikhonov approach leads to inferior reconstructions. Nevertheless, we confirm that despite its computational superiority, {\tt vpal} does not lack numerical accuracy in comparison to {\tt admm}. We depict the relative errors $e(\bfx_{{\tt vpal}})$, $e(\bfx_{\tt admm})$, and $e(\bfx_{\tt tik})$ in \cref{fig:experiment1b_timing} right panel.

\begin{figure}
    \centering
    \includegraphics[width=0.495\textwidth]{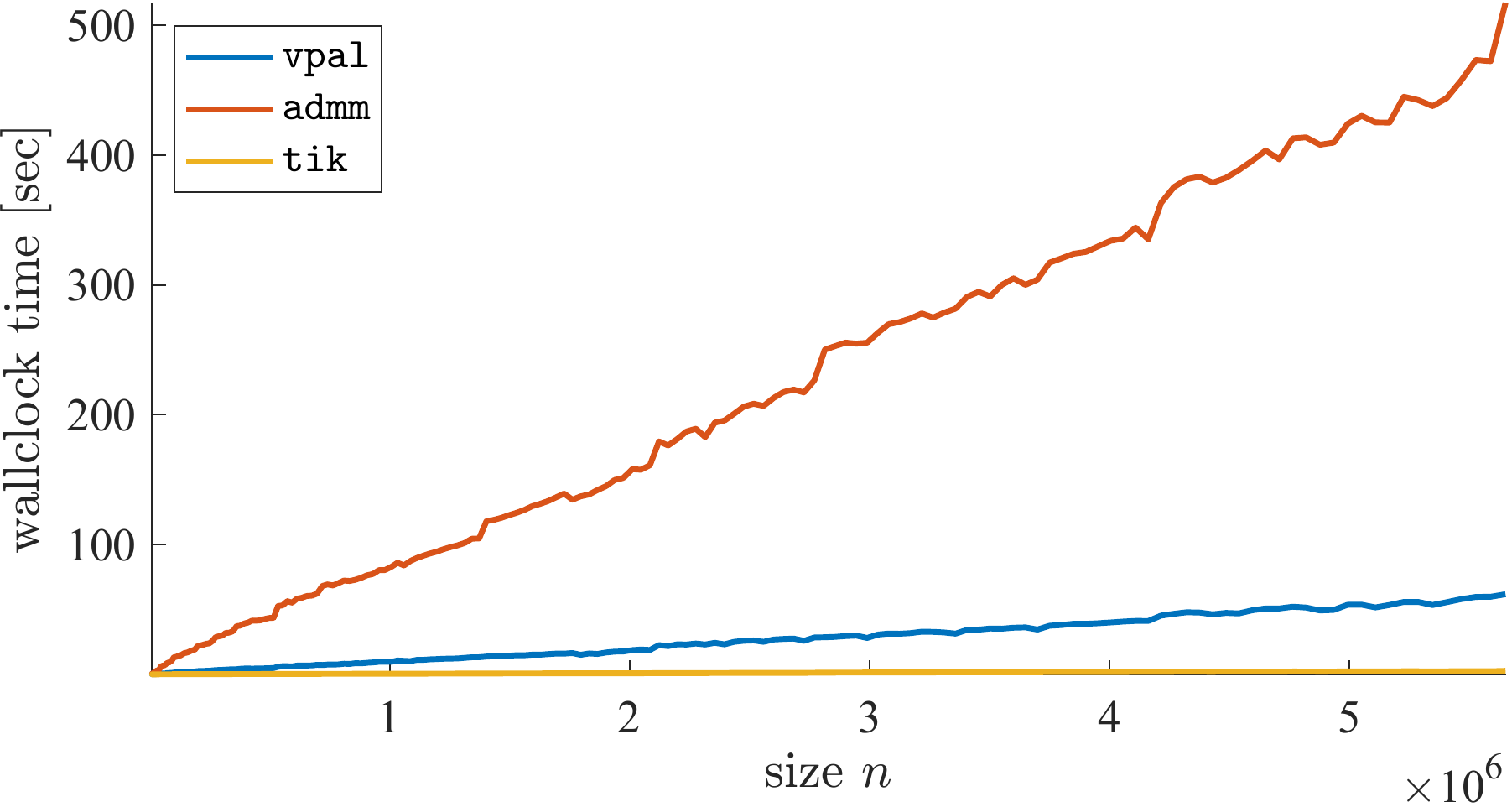}
    \includegraphics[width=0.495\textwidth]{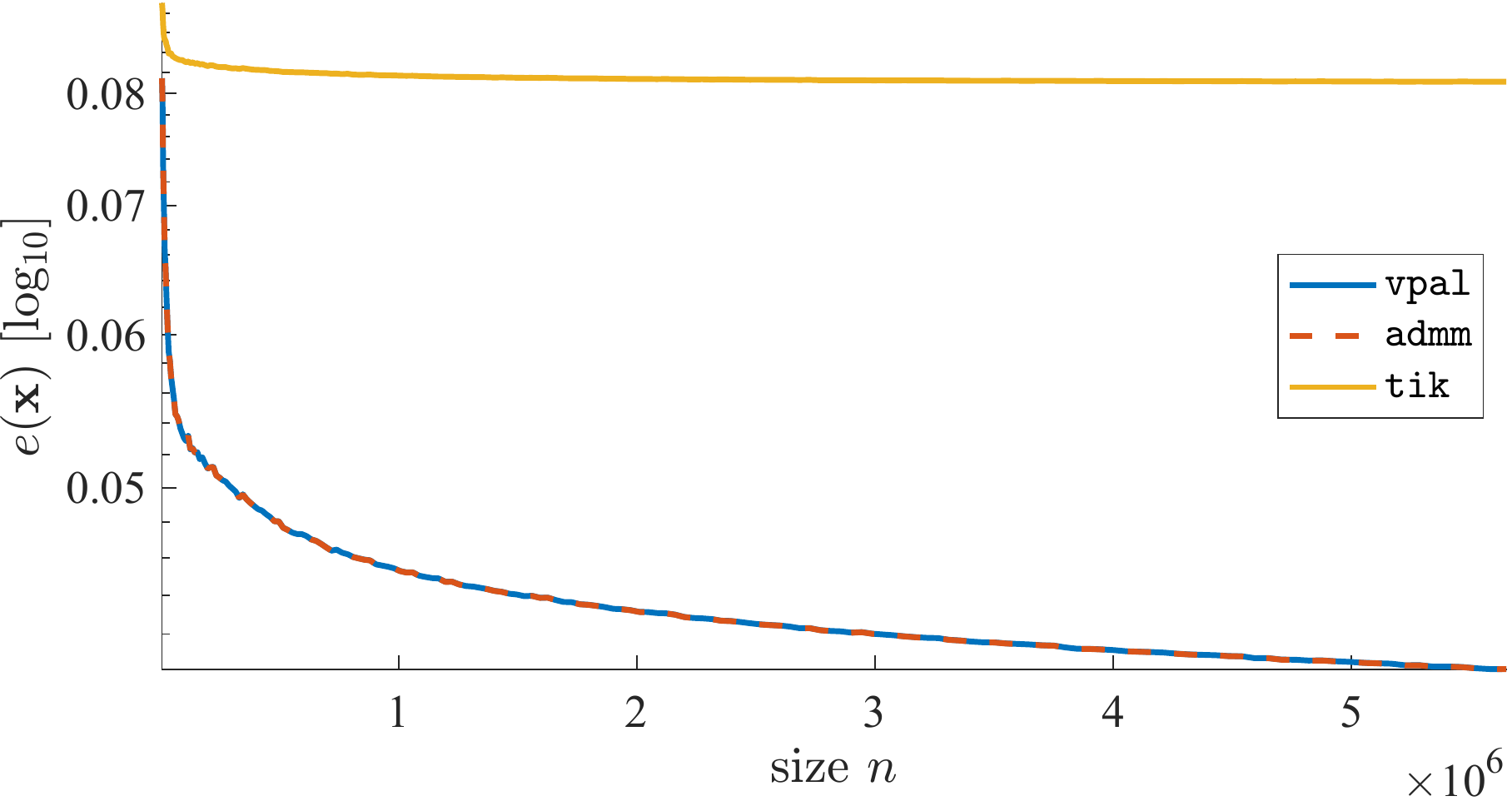}
    \caption{The left plot shows the wallclock timing of {\tt vpal} (blue), {\tt admm} (red), and {\tt tik} (yellow) with increasing image size $n$. Depicted on the right are the corresponding relative reconstruction errors. Despite their differences in computational complexity, both {\tt vpal} and {\tt admm} produce nearly identical reconstructions and are virtually indistinguishable, while {\tt tik} generates inferior reconstructions.}
    \label{fig:experiment1b_timing}
\end{figure}

In \cref{fig:experiment1b_gain} we show the average gain in computational speed-up of {\tt vpal} as compared to {\tt admm}. Depicted is the ratio of the wallclock timing ({\tt admm}$/${\tt vpal}) in black, confirming the computational advantage of {\tt vpal} with an approximate average speed-up of about 8. Additionally, we recorded the ratio of how many $\calO((m+\ell)n)$ operations (the main computational costs of LSQR equivalent steps) on average each of the methods requires ({\tt admm}$/${\tt vpal}) highlighted in purple. Note, due to the similar curves in \cref{fig:experiment1b_gain}, we empirically confirm that the computational costs are dominated by the LSQR solves. Hence, this denoising experiment illustrates that {\tt vpal} compared to {\tt admm} reconstructs images accurately at reduced computational cost. We further continue our numerical investigations on tomography applications.

\begin{figure}
    \centering
    \includegraphics[width=0.495\textwidth]{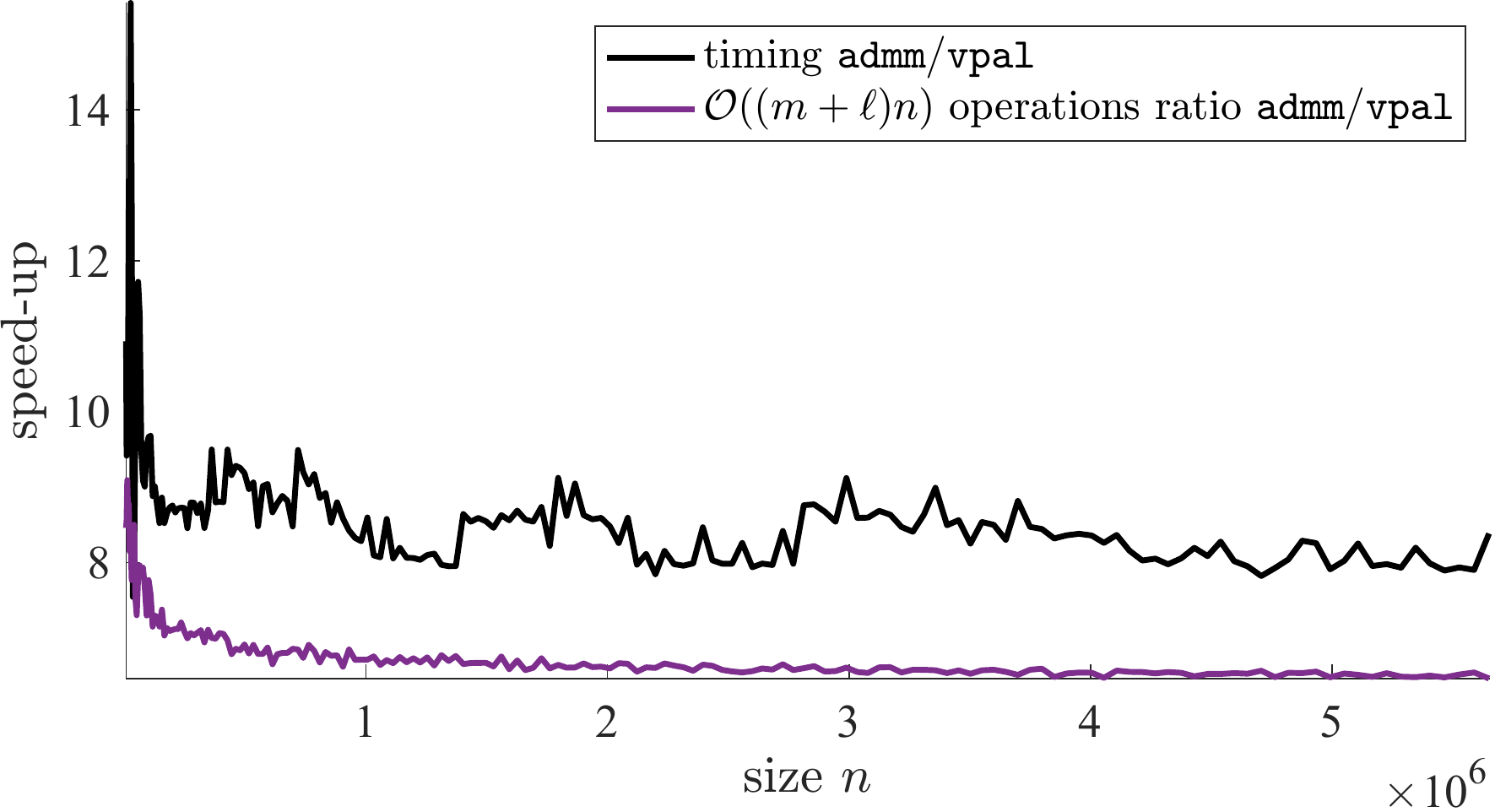}
    \caption{
    Depicted is the computational gain of {\tt vpal} compared to {\tt admm}. In black we show the timing ratio {\tt vpal}/{\tt admm} and in violet the ratio of $\calO((m+\ell)n)$ operations (main computational cost) of {\tt admm} compared to {\tt vpal}.}
    \label{fig:experiment1b_gain}
\end{figure}

\subsection{3D Medical Tomography}\label{sec:tomo}
To demonstrate the scalability of our {\tt vpal} method, we consider a 3D medical tomography application. As a ground truth we consider the 3D Shepp-Logan phantom, illustrated in \cref{fig:experiment3dtomo_xtrue} (left panel) by slice planes. For the discretization of the Shepp-Logan phantom we use $255\times 255 \times 255$ uniform cells, corresponding to $\bfx_{\rm true}\in \bbR^{16,581,375}$. Data is generated by a parallel beam tomography setup using the {\tt tomobox} toolbox \cite{tomobox}. We constructed projection images from $100$ random directions each of size $255 \times 255$. The projection images are contaminated with $5\%$ white noise, an SNR of approximately $26$, see four sample projection images in \cref{fig:experiment3dtomo_b}. Consequently, the ray-tracing matrix $\bfA$ is of size $6,\!502,\!500 \times 16,\!581,\!375$. Note that $\bfA$ is underdetermined and therefore is outside the scope of \cref{thm:vpal} without convergence guarantees. As the regularization operator we use total variation. Here $\bfD$ is of size $49,\!549,\!050 \times 16,\!581,\!375$ and uses zero boundary conditions. The regularization parameter is set to $\mu = 5$.

 \begin{figure}
    \centering
    \includegraphics[width=0.545\textwidth]{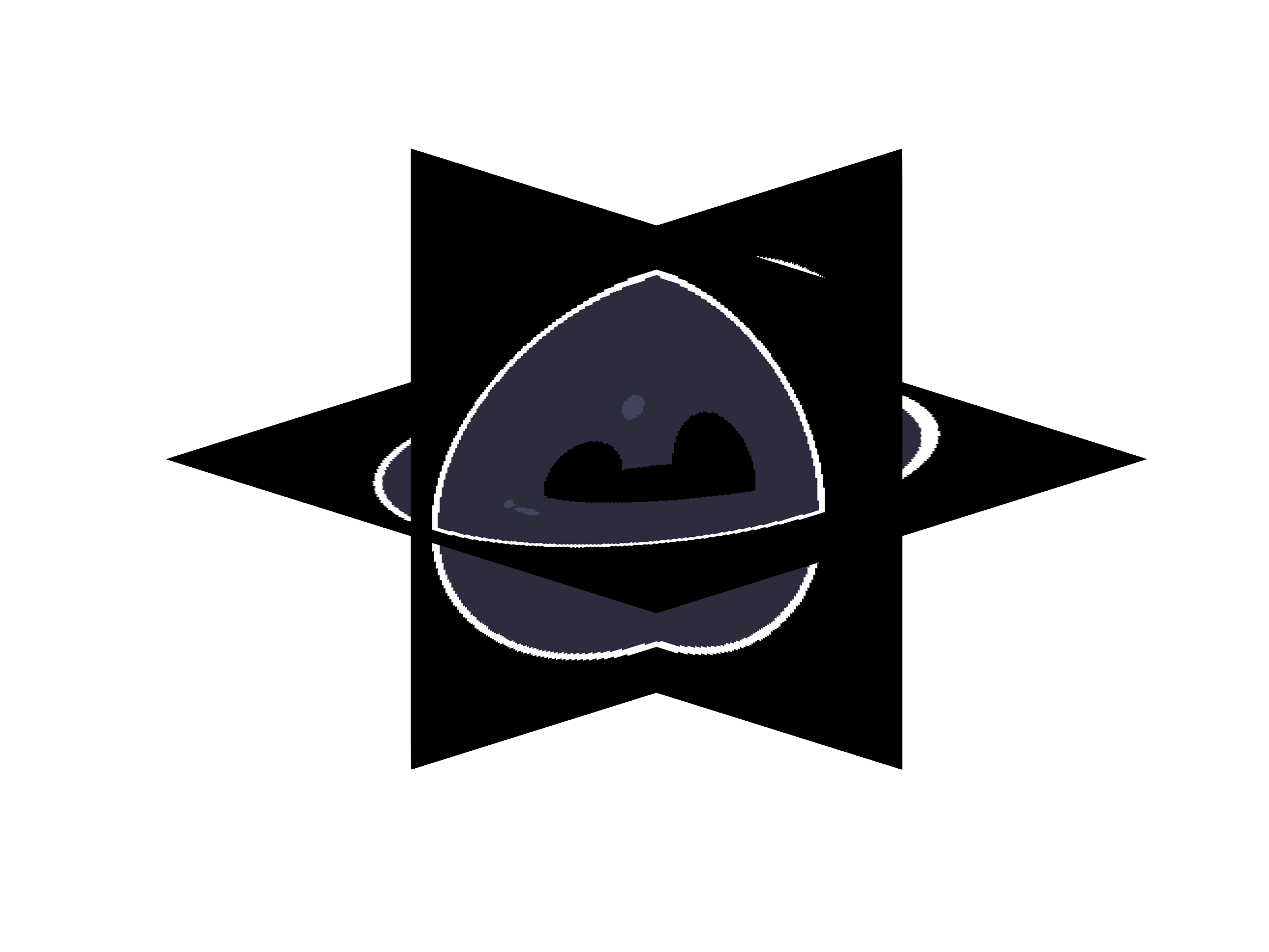}\hspace*{-10ex}
    \includegraphics[width=0.545\textwidth]{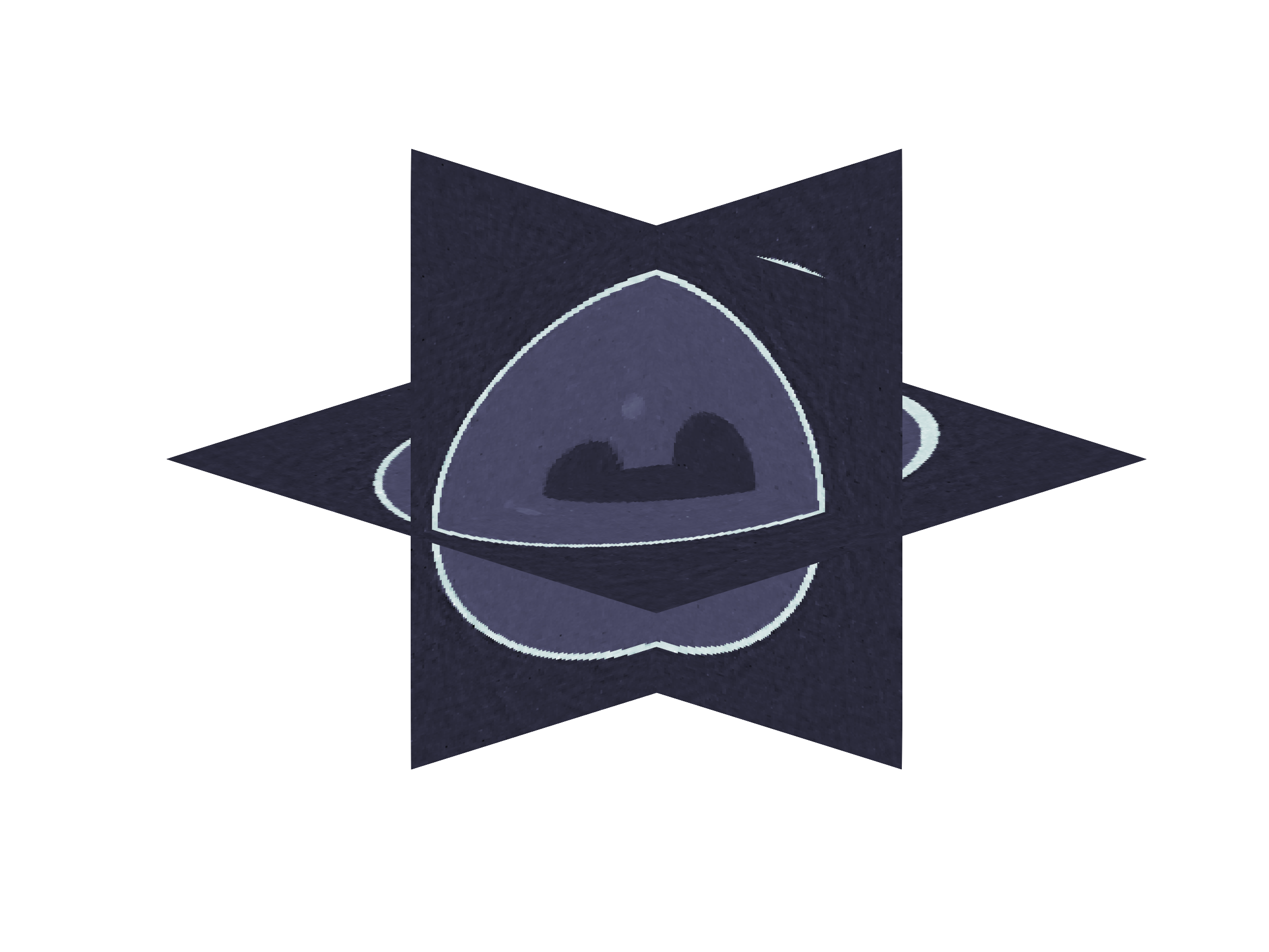}
    \caption{The left panel illustrates the 3D Shepp-Logan phantom using slice planes, while the right panel shows the reconstructed Shepp-Logan phantom using {\tt vpal}.}
    \label{fig:experiment3dtomo_xtrue}
\end{figure}

 \begin{figure}
    \centering
        \includegraphics[width=0.243\textwidth]{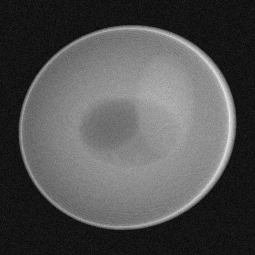}
        \includegraphics[width=0.243\textwidth]{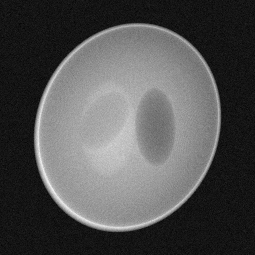}
        \includegraphics[width=0.243\textwidth]{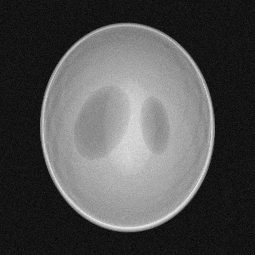}
        \includegraphics[width=0.243\textwidth]{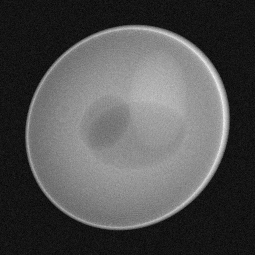}
    \caption{Shown four projection images with $5\%$ white noise and an SNR of approximately $26$.}
    \label{fig:experiment3dtomo_b}
\end{figure}

Our computations were performed on a 2013 MacPro with a 2.7 GHz 12-core Intel Xeon E5 processor with 64 GB Memory 1866 MHz DDR3 running with macOS Big Sur and Matlab 2021a. We use {\tt vpal} with its default tolerance set to $\tau = 10^{-6}$. Our method requires $414$ iterations to converge within the given tolerance with a wallclock time of about $2.4$ hours. The reconstructed 3D Shepp-Logan phantom $\bfx_{\tt vpal}$ is illustrated in \cref{fig:experiment3dtomo_xtrue} (right panel) by slice planes. The corresponding relative error is $e(\bfx_{\tt vpal}) = 0.1167$. A reconstruction error below $12\%$ is noteworthy, considering that the matrix $\bfA$ is significantly underdetermined. Again, notice that a standard Tikhonov regularization (using $\bfD = \bfI_n$) using LSQR is significantly faster (about $3.5$ minutes); however, the relative reconstruction error is inferior. The Tikhonov approach does not generate a relative error below $35\%$ even when utilizing an optimal regularization parameter (data not shown).

\begin{figure}
 \begin{center}
    \begin{tabular}{ccc}
        \includegraphics[width=0.301\textwidth]{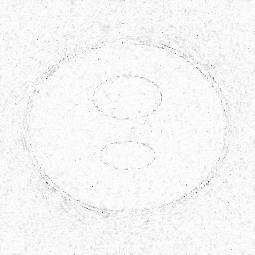} &
        \includegraphics[width=0.301\textwidth]{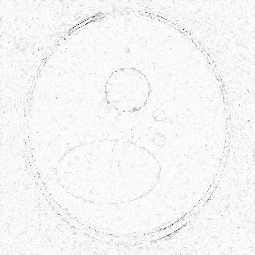} &
        \includegraphics[width=0.301\textwidth]{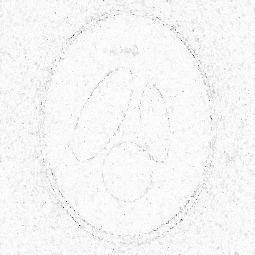} \\
        $x$-slice & $y$-slice & $z$-slice
        \end{tabular}
\end{center}
    \caption{Absolute error of 3D tomography image slices of the reconstruction $\bfx_{\rm vpal}$ in each principal ($x$, $y$, and $y$) coordinate direction with inverted colormap, where darker gray values correspond to larger absolute errors with largest absolute error of about $0.5346$.}
    \label{fig:experiment3dtomo_xvpal}
\end{figure}

\begin{figure}
    \centering
    \includegraphics[width=0.495\textwidth]{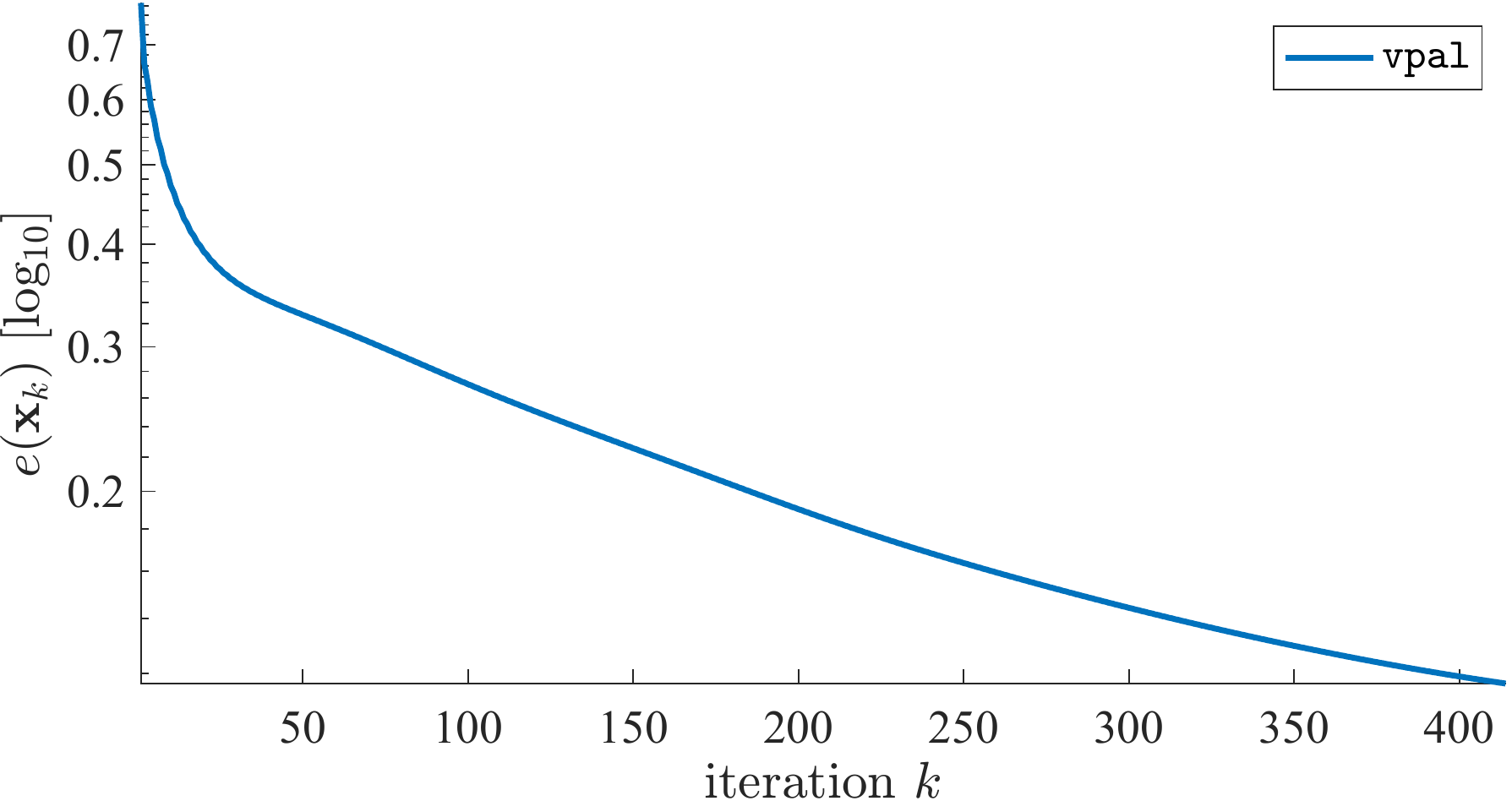}
    \includegraphics[width=0.495\textwidth]{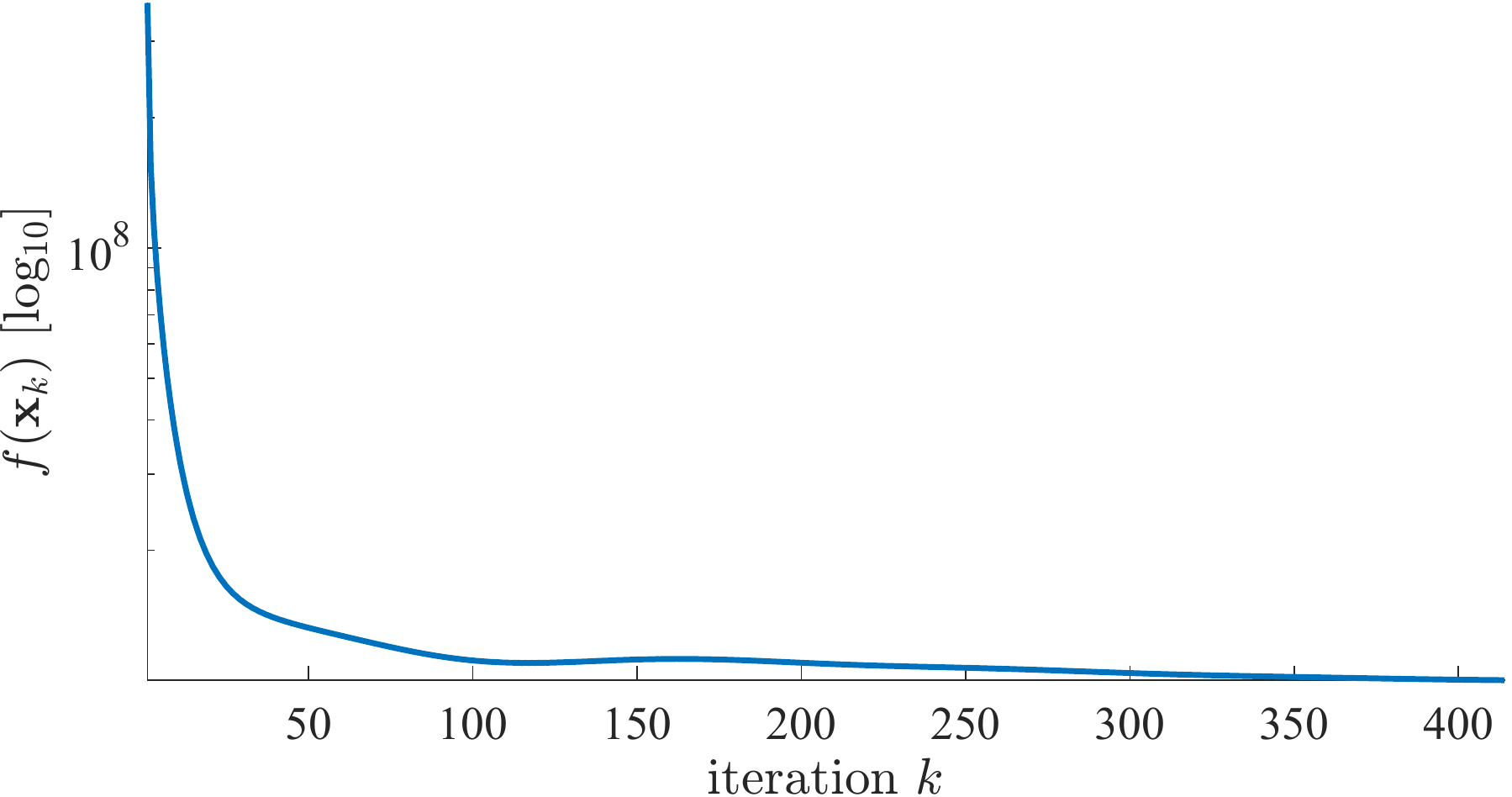}
    \caption{The left graph shows the relative error  of the reconstructions $\bfx_k$ at each iteration of {\tt vpal} of the 3D tomography problem. The right graph shows the corresponding objective function value $f(\bfx_k)$ at each iteration. Note that with this under-determined $\bfA$ the relative error is monotonically decreasing, while the objective function value is not monotonically decreasing.}
    \label{fig:experiment3dtomo_relerr}
\end{figure}

Reconstruction results are further illustrated in \cref{fig:experiment3dtomo_xvpal,fig:experiment3dtomo_relerr}, where \cref{fig:experiment3dtomo_xvpal} shows  the absolute error image, and \cref{fig:experiment3dtomo_relerr} depicts the relative reconstruction errors $\bfx_k$ of {\tt vpal} and also the corresponding objective function value $f(\bfx_k)$ at each iteration. We notice fast but not necessarily monotonically decreasing objective function values $f(\bfx_k)$.

\subsection{Experiments on parameters}\label{sec:regpara}
In the following we discuss experiments with our approach for regularization parameter selection using the degrees of freedom argument. As a test bed we investigate a deblurring ({\tt blur}), 2D medical tomography ({\tt tomo}), and seismic ({\tt seismic}) inversion problem. For each problem we utilize the Matlab Toolbox {\tt IRtools}; see \cite{gazzola2019ir} for further details on our choices. Here, in all cases we select Gaussian noise, and for {\tt blur} we use a {\tt severe shake blur} for the Hubble space telescope,  for {\tt tomo} we use the 2D Shepp-Logan phantom, while for {\tt seismic} we select the {\tt tectonic} phantom. In all usages of the {\tt vpal} algorithm we select a stopping tolerance of $\tau = 10^{-4}$ and limit the number of outer iterations to $1,\!000$.  For the bisection algorithm we set $\tau_1 = 0.01$, $\tau_2 = 0.02$,  limit the number of bisection steps to $10$, and always use the safety parameter  $\eta=1$; no confidence interval is used. Note that each bisection is a full solve of {\tt vpal} and thus the total maximum number of solves could be as high as $10,\!000$ with these settings. There is a trade-off on the accuracy of the bisection desired and the potential for high computational overhead by carrying out the bisection.

We investigate each of these experiments for various image and data sizes. Our results are presented for four problem sizes with $N$ ranging from $64$ to $512$, corresponding to image sizes $n=N^2$ ($4,\!096$, $16,\!384$, $65,\!536$, and $262,\!144$) for problem {\tt blur}, {\tt tomo} and {\tt seismic}. The size of the data vector $\bfb$ varies with the application, where $m=n$ for {\tt blur}, $m=16,\!380$, $32,\!580$, $65,\!160$, and $130,\!320$ for {\tt tomo} and $m=8,\!192$, $32,\!768$, $131,\!072$, and $524,\!288$ for {\tt seismic}. Notice these experiments correspond to $m=n$ for all the {\tt blur} cases, $m>n$ for {\tt seismic} (over-determined) and with {\tt tomo} $m<n$ (under-determined) for the two larger experiments. In each of these experiments we investigate white noise levels of $10\%$ and $20\%$, corresponding to SNR of $20$ and $13.98$ in each case.

In \cref{fig:experimentmu_deblur1,fig:experimentmu_tomo1,fig:experimentmu_seismic1} we evaluate the choice of $\mu$ determined using the $\chi^2$-DF test to give $\muopt$ as compared to that obtained as the MAP estimator $\mumap$. To find $\muopt$ we pick a value for the shrinkage parameter $\gamma$ and estimate $\muopt$ by the $\chi^2$-DF~test. (Results show that the $\muopt$ is virtually independent of $\gamma$.)  The relative errors for the solutions obtained in this way are indicated on the plots using the solid red circles with the legend $(\muopt, \gamma)$. We also show the values of the relative error that are obtained by taking $\muopt$ and $\mumap$, and using the DP to find an optimal $\lambda$, denoted $\lambda_{\textrm{opt}}$ and $\lambda_{\textrm{map}}$ respectively. For the DP we again use the bisection algorithm with all the same settings but now for the function $H(\lambda)= \tfrac{\norm[2]{\bfA\bfx(\lambda) - \bfb}^2}{m \sigma^2} -1$. The blue solid circles and green open circles show the  errors calculated using the pairs  $(\mumap, \lambda_{\textrm{map}})$ and $(\muopt, \lambda_{\textrm{opt}})$, respectively, that are generated by $\mumap$ and $\muopt$,  with  the DP  to find the relevant $\lambda$. Then, given the respective $(\mu,\lambda)$ pairs, we fix $\mu$ but use the $\lambda$ value to provide a range of $\lambda$ logarithmically spaced between  $\lambda_{\textrm{map}}/100$ and $100 \lambda_{\textrm{map}}$ at $50$ points, with $\lambda_{\textrm{opt}}$ for the $\muopt$ case. The resulting relative errors  obtained for the range of $\lambda$, with fixed $\mumap$ and $\muopt$, are given in the two curves, blue and red, respectively. In these figures the oscillations in the relative error curves occur if the algorithm did not converge within $1,\!000$ iterations, which is more prevalent for small values of $\lambda$.  Given that $\mu$ is fixed on these curves, this corresponds to taking larger values of the shrinkage parameter $\gamma$. Flat portions of the curves indicate the relative lack of sensitivity to the choice of $\lambda$ (respectively $\gamma$) for a fixed $\mu$, equivalently confirming that the optimal $\mu$ is largely independent of  $\gamma$ within a suitably determined range, dependent on the data and the problem.

We see immediately that  finding $\muopt$ by the $\chi^2$-DF test yields smaller relative errors  in the solutions than when the solution is generated using $\mumap$, open green circles and solid red circles are lower than solid blue circles, and except where there are issues with convergence, the red curves lie below blue curves. This is notwithstanding that the $\chi^2$-DF test and the MAP estimators do not immediately apply for the {\tt tomo}  and {\tt seismic} problems, since neither meets the criterion of being differentially Laplacian natural images. Indeed, the Hubble space telescope image is not a perfect example of such an image, but the results demonstrate that the approach still works reasonably well. On the other hand, comparing now red solid and green open circles, contrasts the impact of finding a $\muopt$ by the $\chi^2$-DF test (red solid)  and then assessing whether the standard DP (green open) on $\lambda$ might be a better option. It is particularly interesting that the $\chi^2$-DF test does uniformly well on {\tt tomo} and {\tt seismic} problems, but there are a few cases with {\tt blur} in which the DP finds a $\lambda$ that yields a smaller relative error. Even in these cases, the results are good using the $\chi^2$-DF result. In all situations it is clear that we would not expect to find an optimal $\lambda$ that is at the minimum point of the respective relative error curve, but in general the results are acceptably close to these minimum points. Overall, these results support the approach in which we pick a shrinkage parameter $\gamma$ and find $\muopt$ using the $\chi^2$-DF test by bisection. The results presented for the DP approach to find $\lambda$ were provided to contrast the two directions for estimating the parameters. Indeed, it is clear that finding a $\mu$ to fit $F(\mu)$ given by \cref{eq:dof} to $m\sigma^2$ and then fitting the residual term there alone, also to $m\sigma^2$, by applying DP, will necessarily increase the value of $F(\mu)$. This set of results demonstrates that there is no need to use the DP principle; rather, optimizing based on \cref{eq:dofmead} is appropriate for all the test problems.

To further assess the validity of this assertion we also performed a parameter sweep over a grid of values for $(\mu, \lambda)$ for the experiments that use $10\%$ noise, corresponding to an SNR of $20$. For each point on the grid we calculated both the relative error $e(\bfx(\mu,\lambda))$ and  the $\chi^2$ value, $|F(\mu,\lambda)/m\sigma^2 -1|$. The left panels and right panels in \cref{fig:blur_lambda_mu,fig:tomo_lambda_mu,fig:seismic_lambda_mu} show the contour plots for the relative errors and $\chi^2$ estimates, respectively. The red dots correspond to the points with minimum error and the black dots to the points with minimum $\chi^2$ value. If they are close we would assert that the $\chi^2$ is optimal for finding a good value for $\mu$.
In general, the contours are predominantly vertical, confirming that the solutions are less impacted by the choice of $\lambda$ (and hence shrinkage $\gamma$) than of $\mu$. Further, it is necessary to examine the values for the contours, given in the colorbar, in order to assess whether the $\chi^2$ is not giving a good solution. From \cref{fig:blur_lambda_mu,fig:tomo_lambda_mu} we can conclude that the difference in the relative error from using the $\chi^2$ estimate for $\mu$ rather than the optimal in terms of the minimal error is small; all values lie within the blue contours. Even for the {\tt seismic} case shown in \cref{fig:seismic_lambda_mu} the contour level for the relative error changes only from about $0.06$ to $0.09$. The captions give the actual calculated relative errors for the red and black dots.

\begin{figure}
    \centering
    \begin{tabular}{cccc}
         $n = 64^2$ & $n = 128^2$ & $n = 256^2$ & $n = 512^2$ \\
         \includegraphics[width=0.220\textwidth]{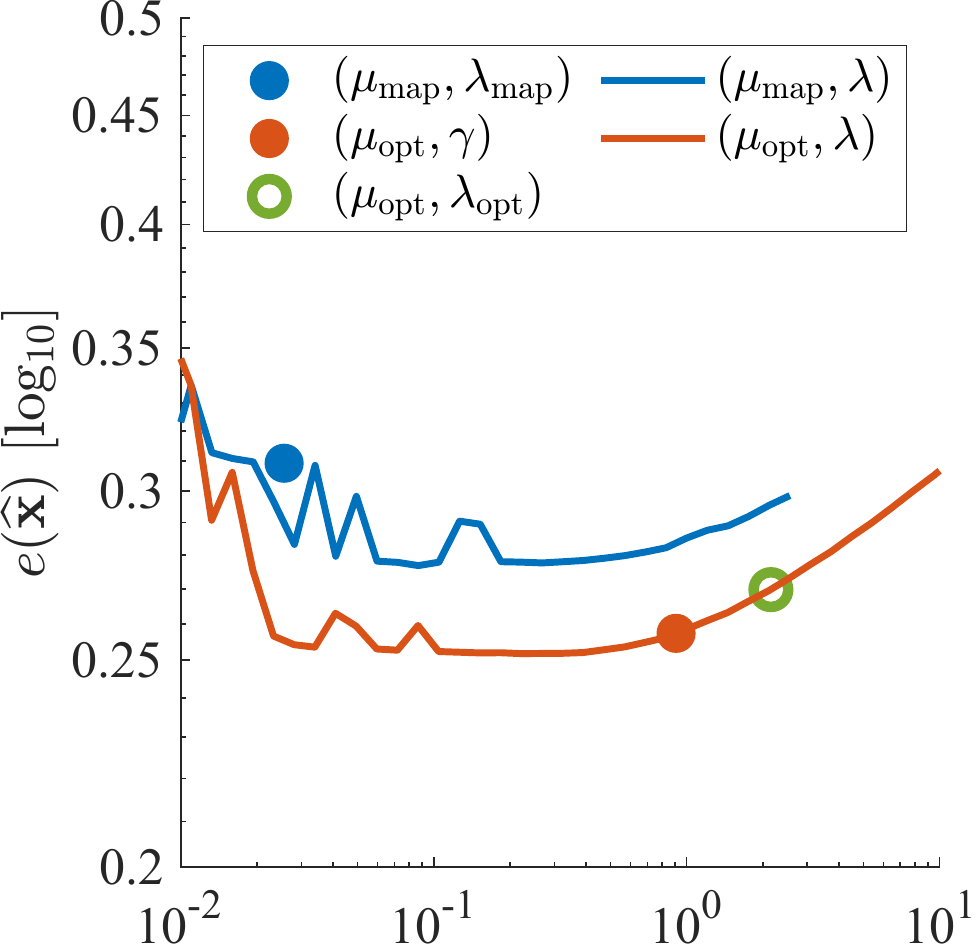} &
         \includegraphics[width=0.215\textwidth]{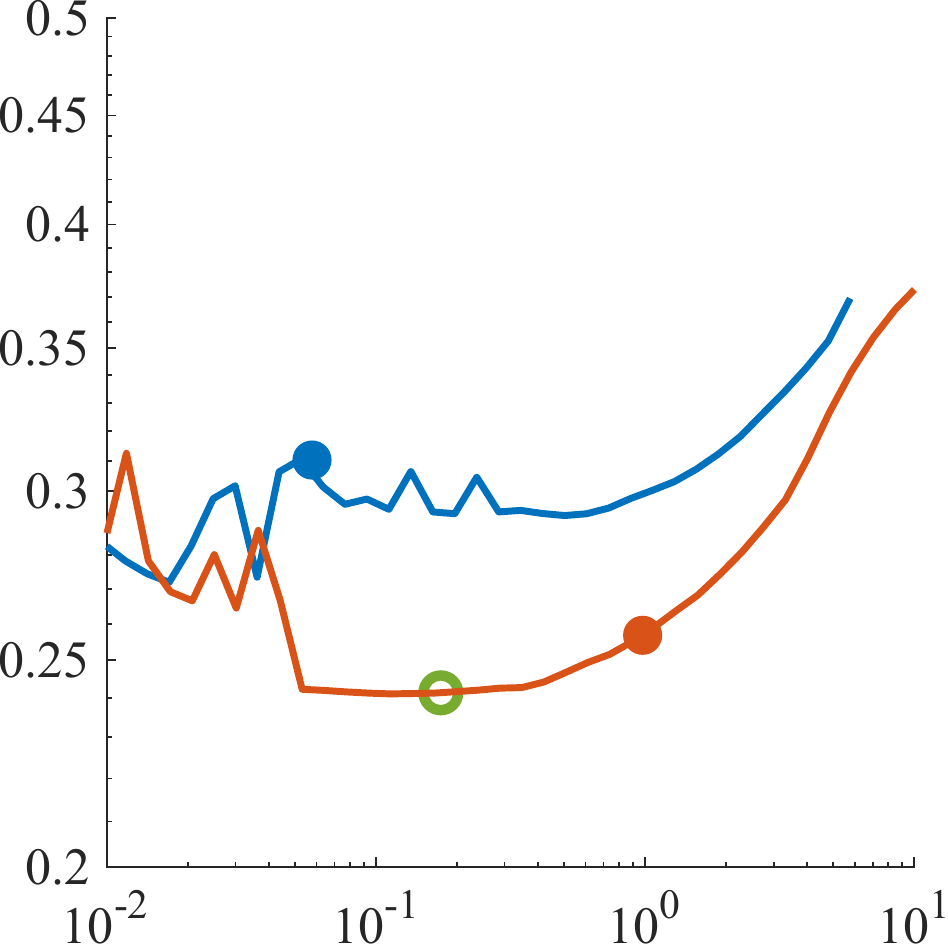} &
         \includegraphics[width=0.215\textwidth]{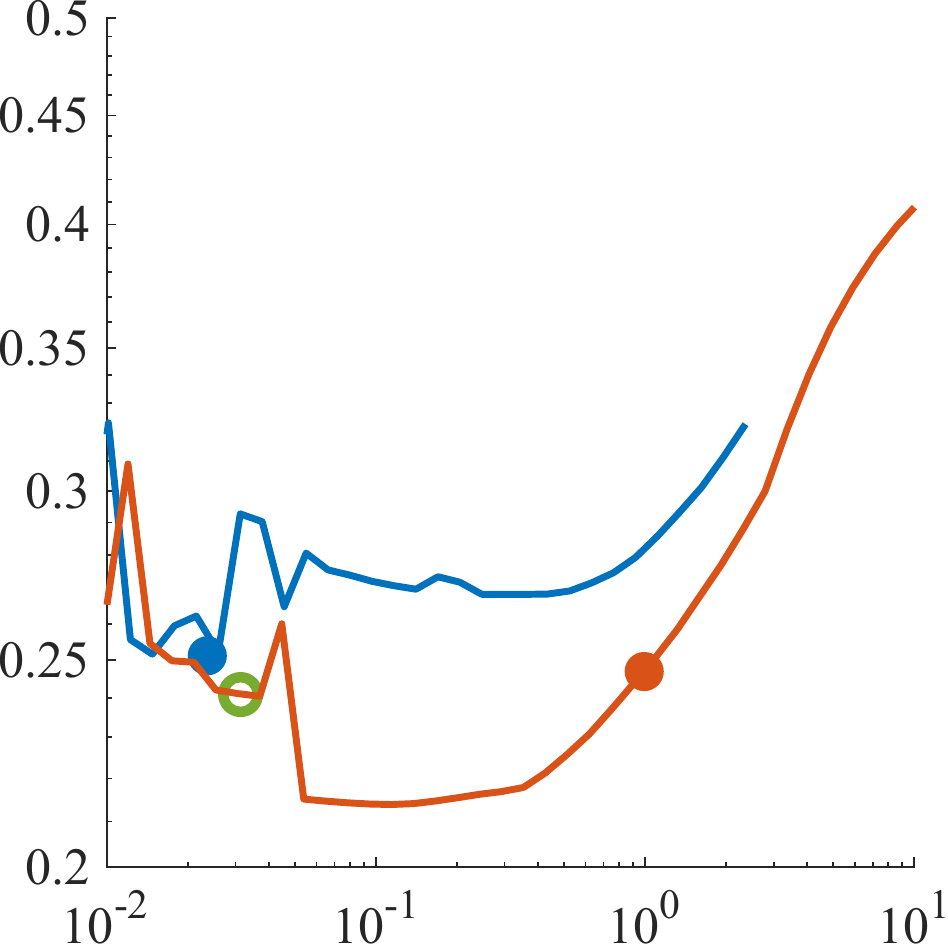} &
         \includegraphics[width=0.215\textwidth]{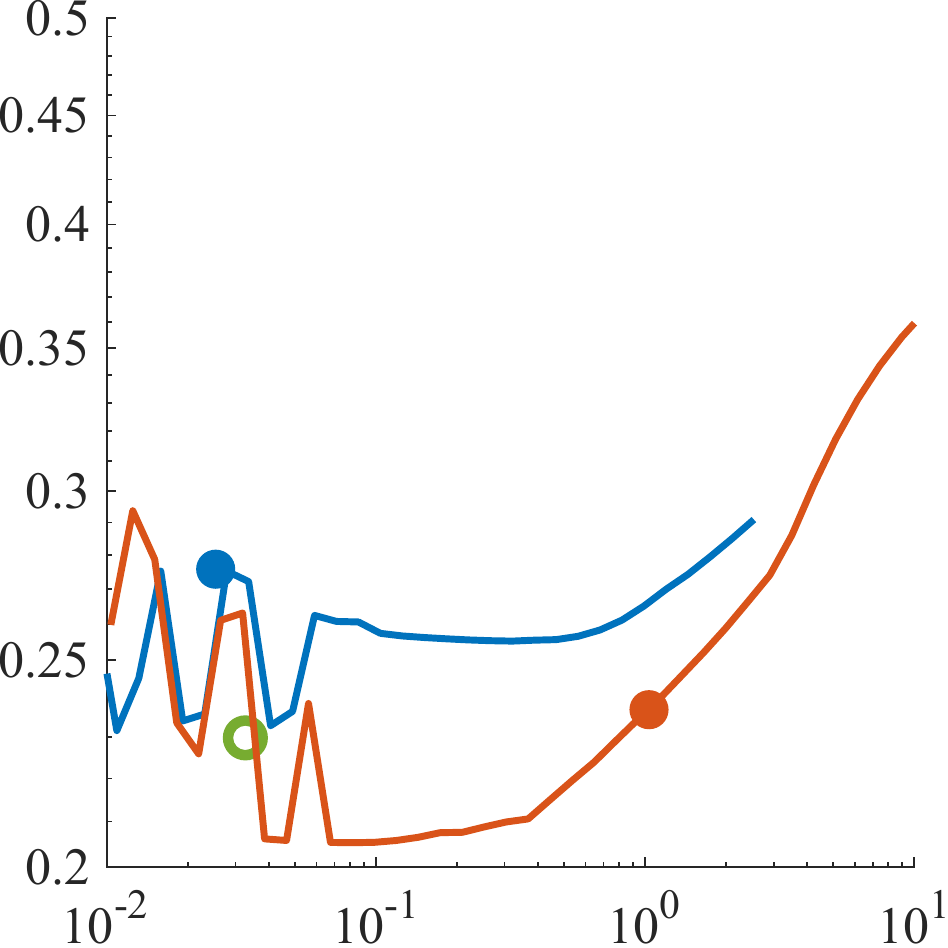} \\
         \includegraphics[width=0.220\textwidth]{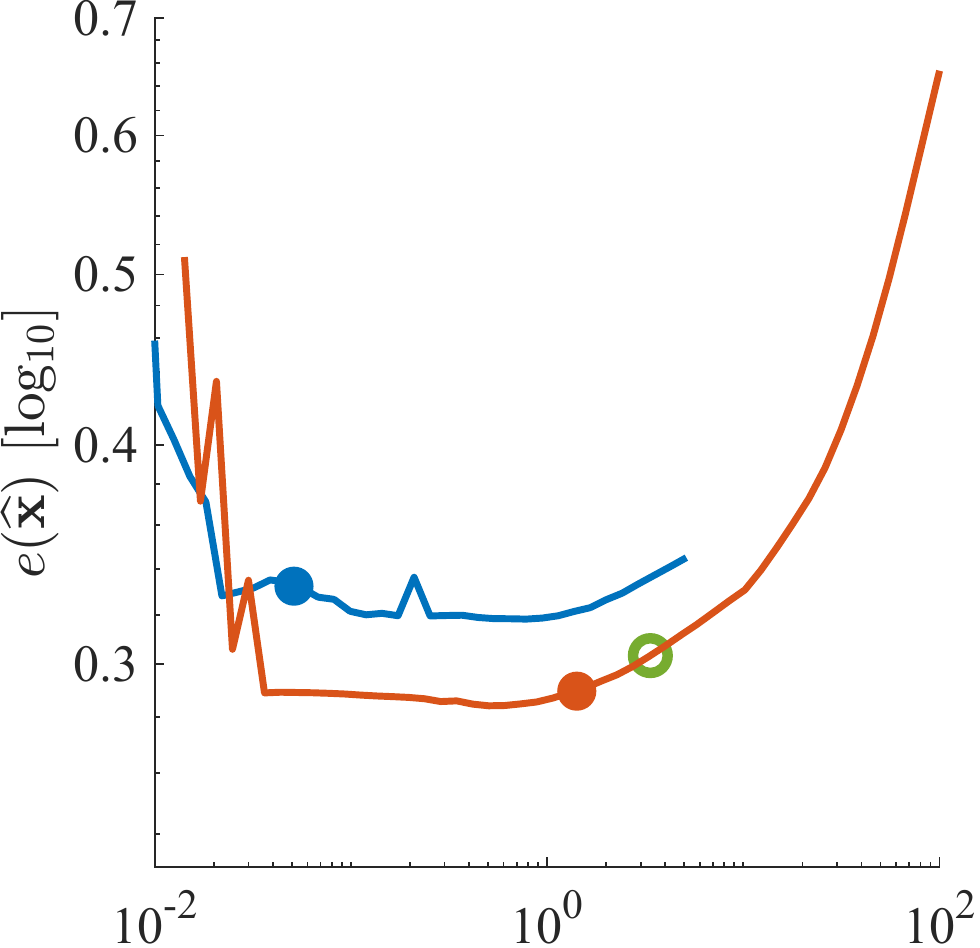} &
         \includegraphics[width=0.215\textwidth]{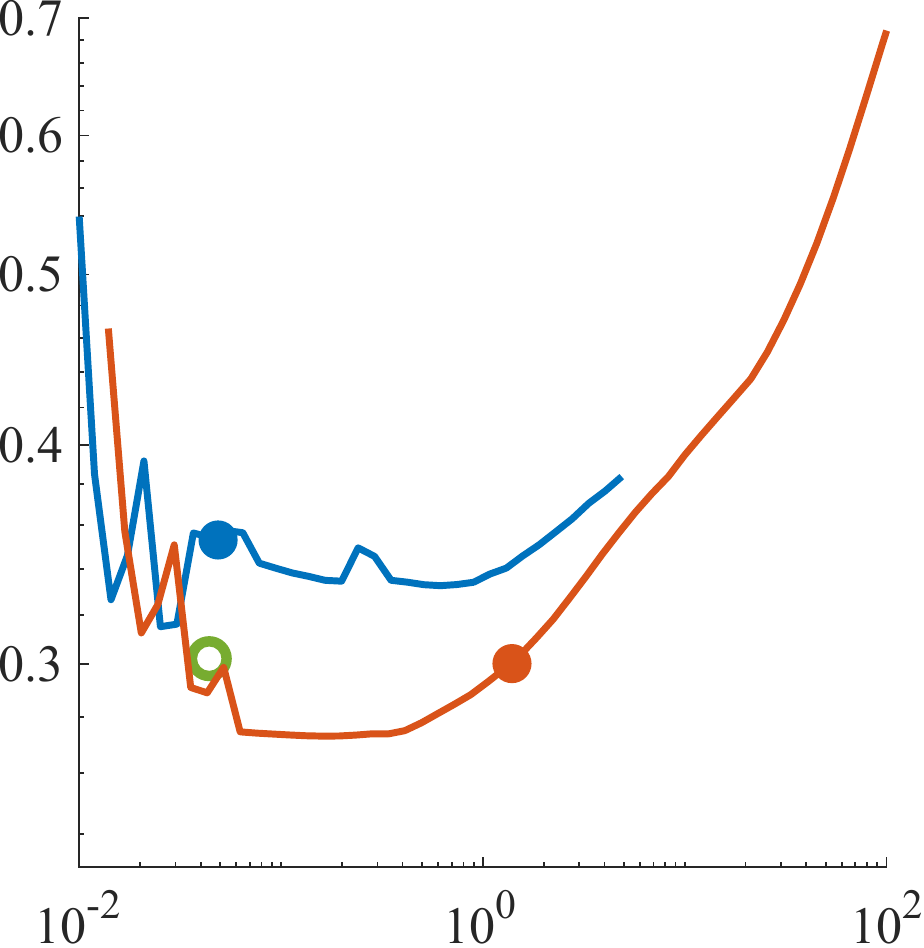} &
         \includegraphics[width=0.215\textwidth]{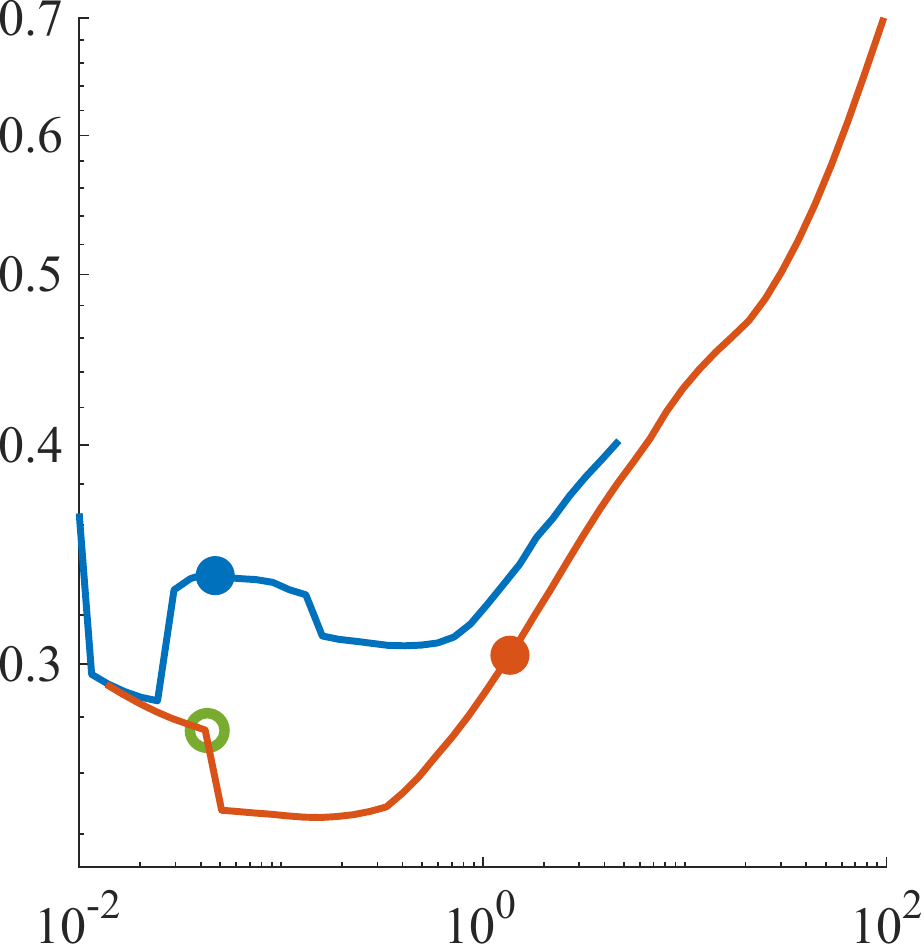} &
         \includegraphics[width=0.215\textwidth]{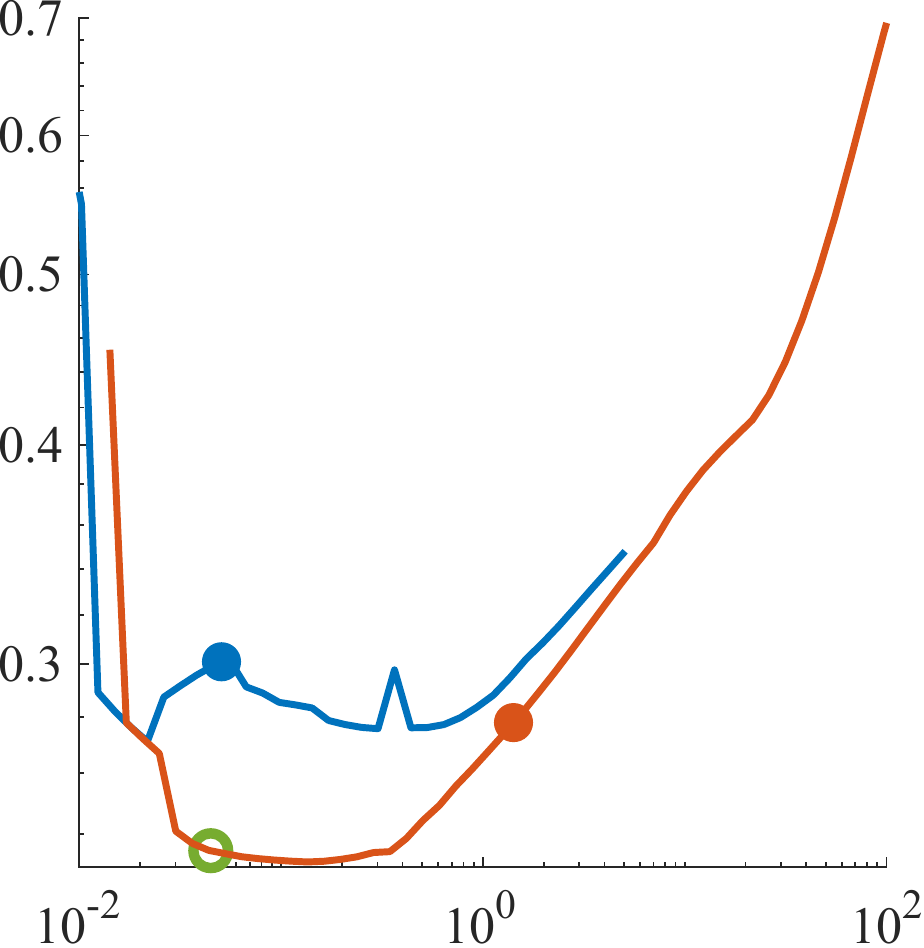} \\[-1.5ex]
         \tiny $\lambda$ [log$_{10}$] & \tiny $\lambda$ [log$_{10}$] & \tiny $\lambda$ [log$_{10}$] & \tiny $\lambda$ [log$_{10}$] \\
    \end{tabular}
    \caption{Results for {\tt blur} example for varying $\lambda$ and fixed regularization parameter $\mu$, $\mumap$ and $\muopt$ for the blue and red curves, respectively. Also marked for the given fixed $\mu$ choices are the optimal $\lambda$ found using the DP.  First row shows relative errors for different image sizes ($n = 64^2, 128^2, 256^2$, and $512^2$) with noise level of $10\%$ corresponding to a signal to noise ratio (SNR) of $20$. Second row shows relative errors with noise level of $20\%$ corresponding to an SNR of $13.98$.
    \label{fig:experimentmu_deblur1}}
\end{figure}

\begin{figure}
    \centering
    \begin{tabular}{cccc}
         $n = 64^2$ & $n = 128^2$ & $n = 256^2$ & $n = 512^2$ \\
         \includegraphics[width=0.226\textwidth]{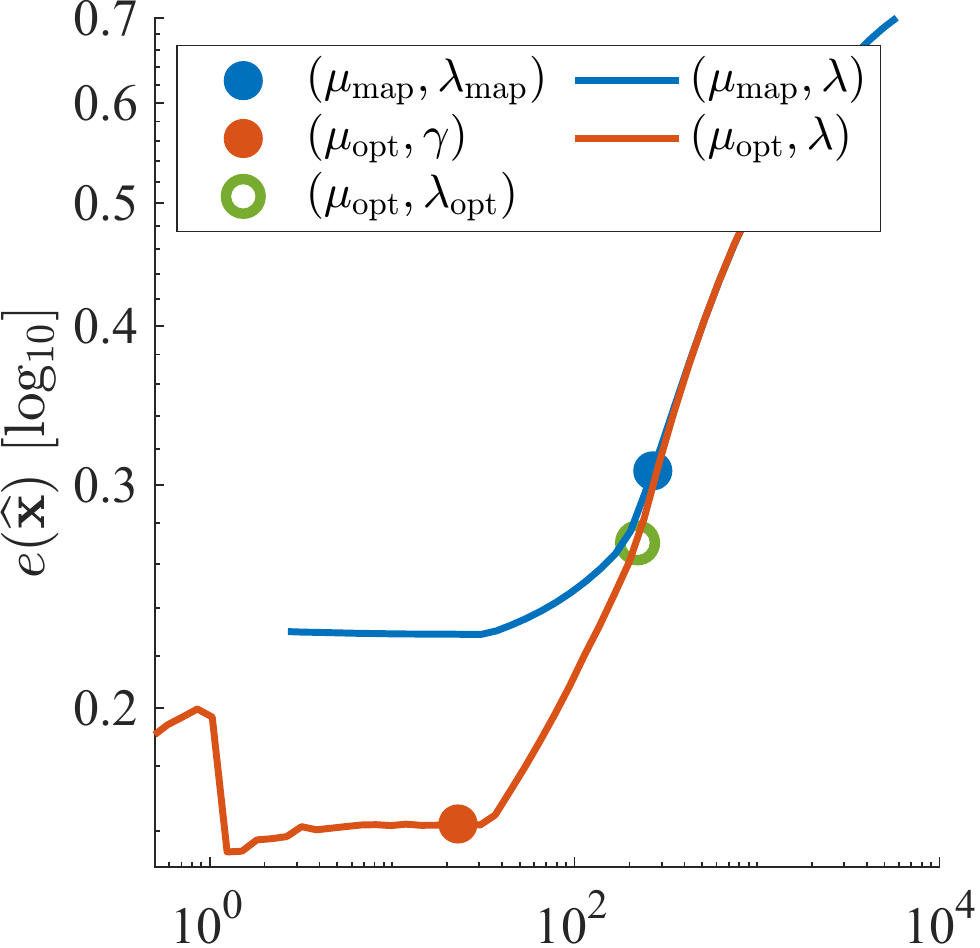} &
         \includegraphics[width=0.215\textwidth]{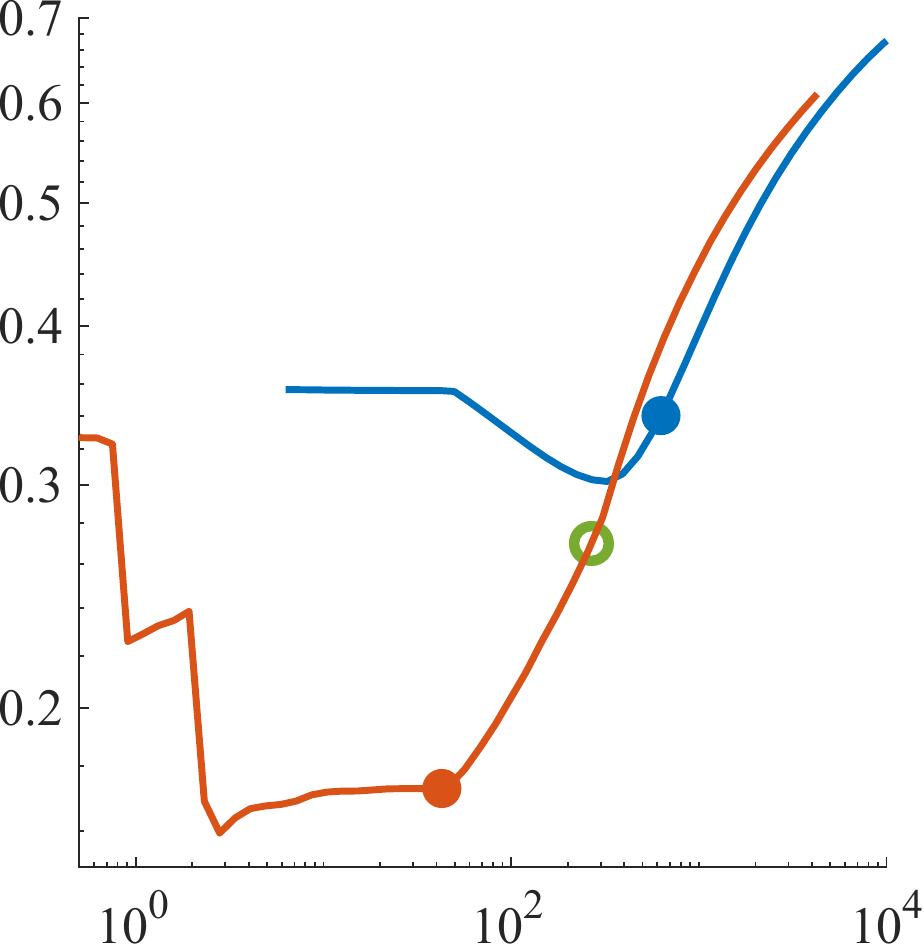} &
         \includegraphics[width=0.215\textwidth]{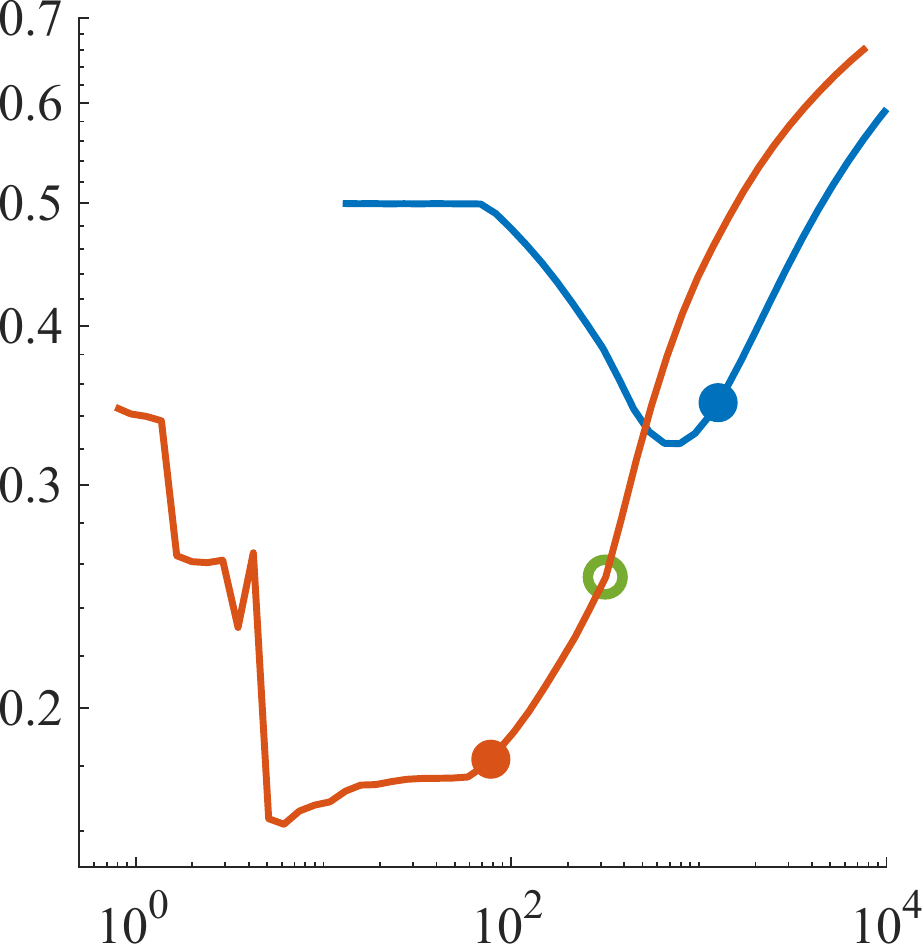} &
         \includegraphics[width=0.215\textwidth]{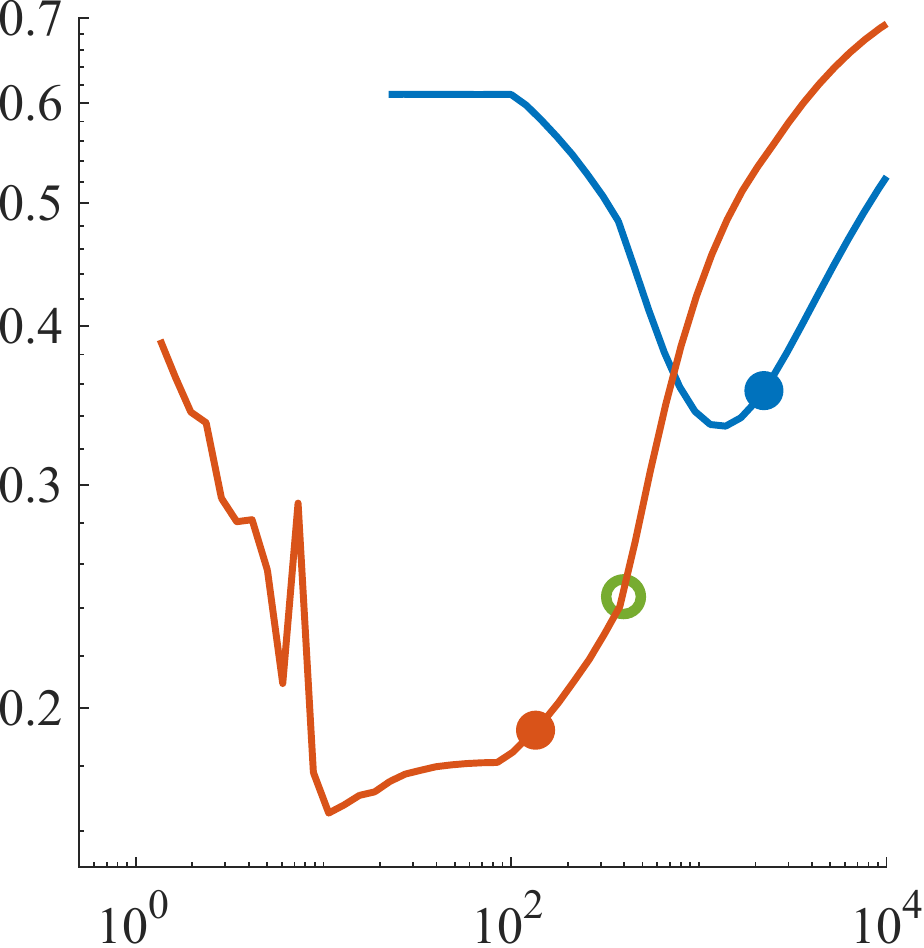} \\
         \includegraphics[width=0.226\textwidth]{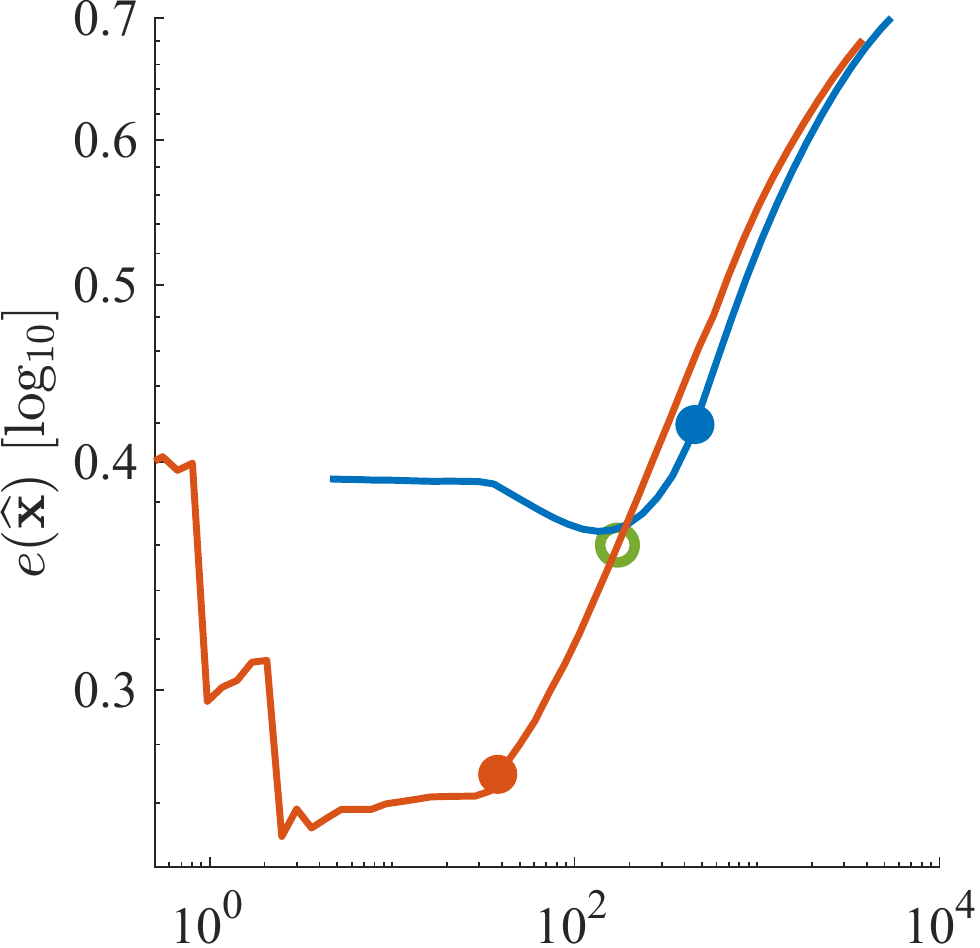} &
         \includegraphics[width=0.215\textwidth]{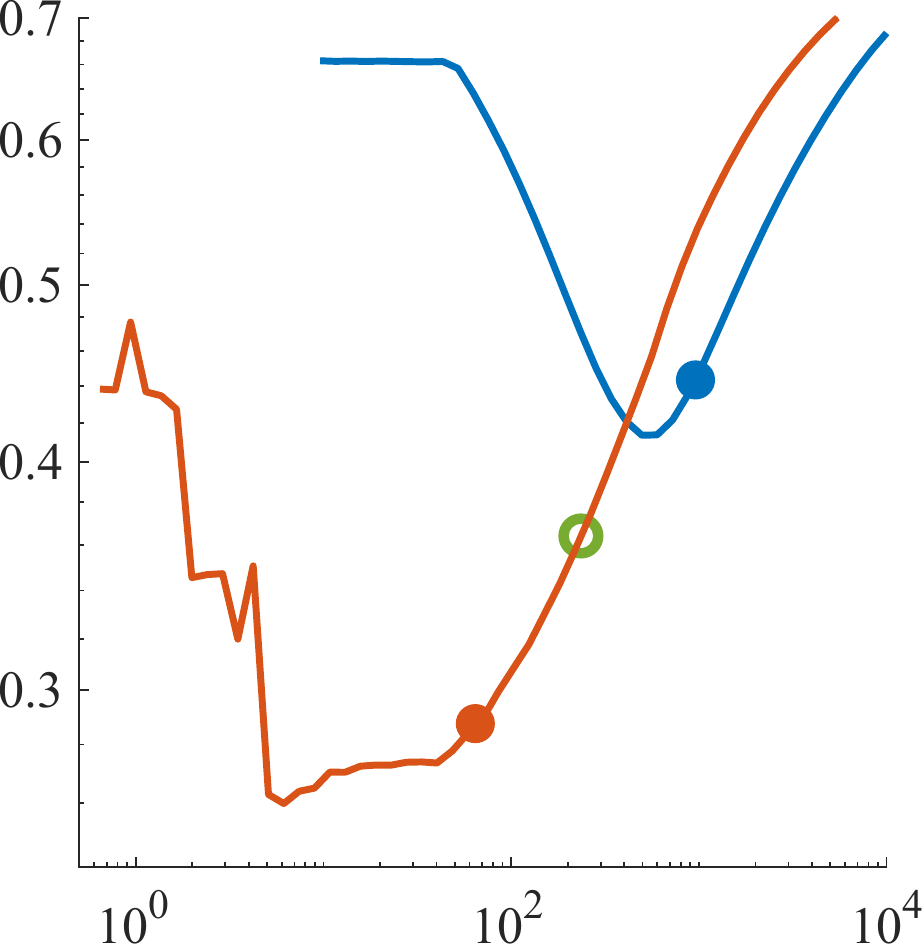} &
         \includegraphics[width=0.215\textwidth]{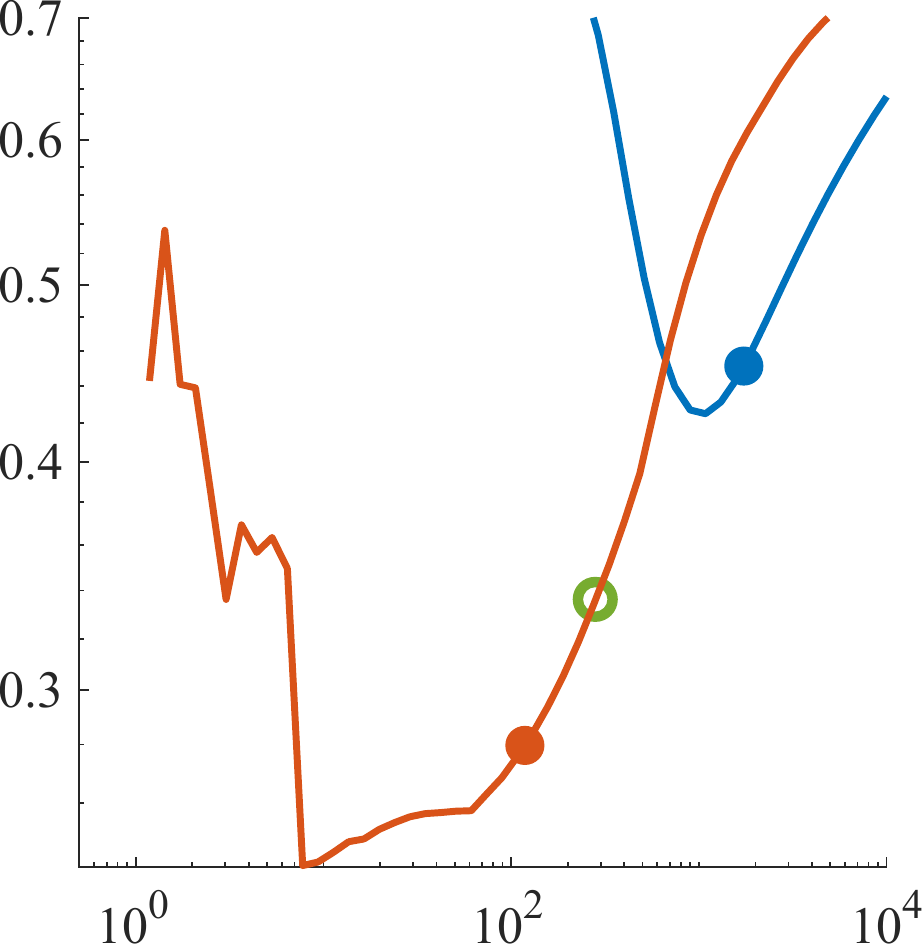} &
         \includegraphics[width=0.215\textwidth]{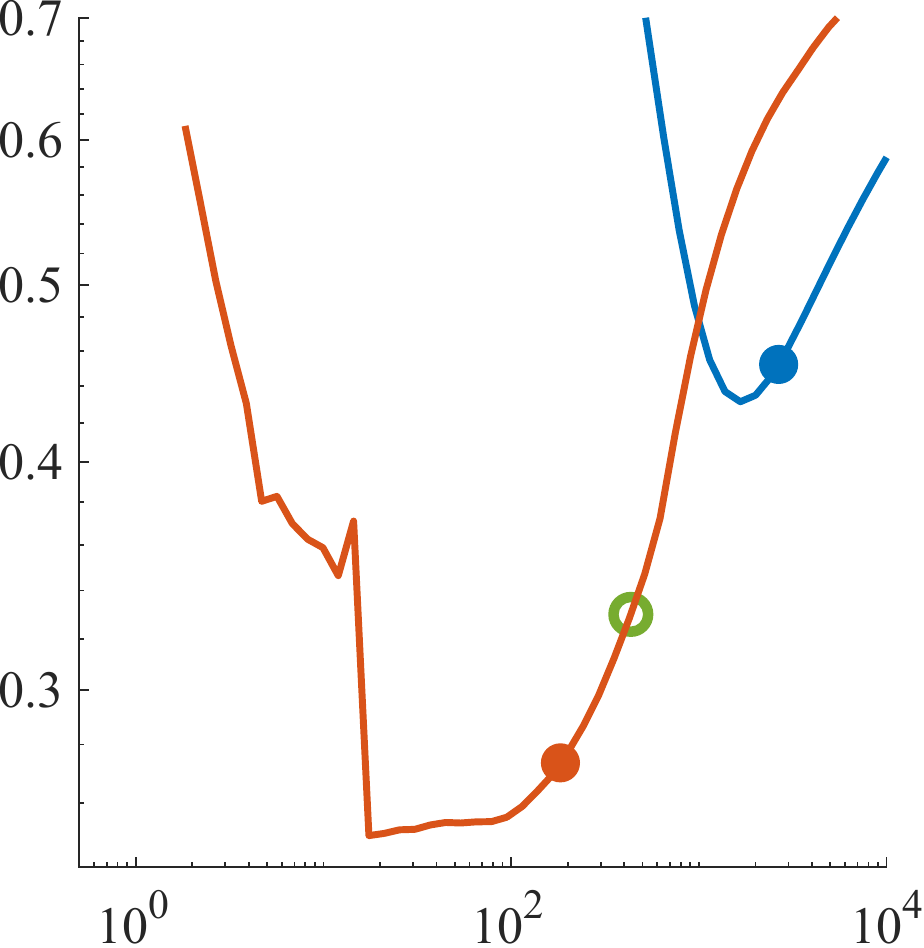} \\[-1.5ex]
         \tiny $\lambda$ [log$_{10}$] & \tiny $\lambda$ [log$_{10}$] & \tiny $\lambda$ [log$_{10}$]  & \tiny $\lambda$ [log$_{10}$] \\
    \end{tabular}
    \caption{Results for {\tt tomo} example for varying $\lambda$ and fixed regularization parameter $\mu$, $\mumap$ and $\muopt$ for the blue and red curves, respectively. Also marked for the given fixed $\mu$ choices are the optimal $\lambda$ found using the DP.  First row shows relative errors for different image sizes ($n = 64^2, 128^2, 256^2$, and $512^2$) with noise level of $10\%$ corresponding to a signal to noise ratio (SNR) of $20$. Second row shows relative errors with noise level of $20\%$ corresponding to an SNR of $13.98$.
    \label{fig:experimentmu_tomo1}}
\end{figure}

\begin{figure}
    \centering
    \begin{tabular}{cccc}
         $n = 64^2$ & $n = 128^2$ & $n = 256^2$ & $n = 512^2$ \\
         \includegraphics[width=0.226\textwidth]{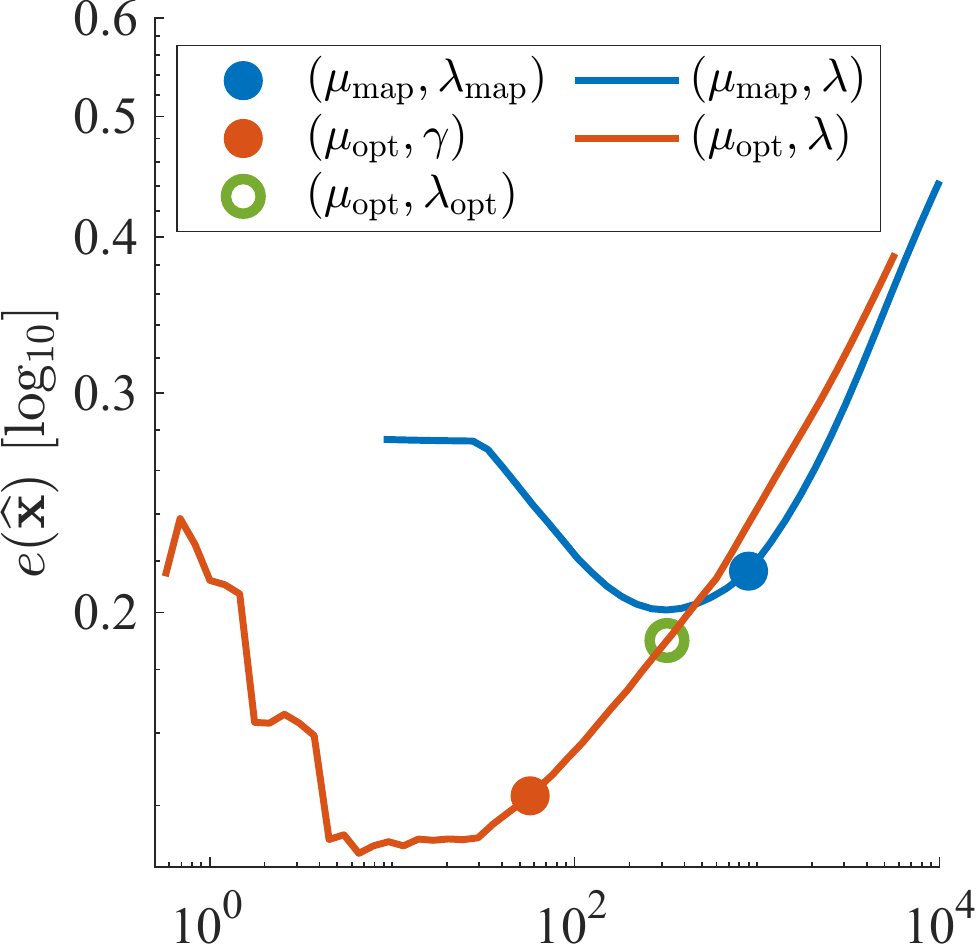} &
         \includegraphics[width=0.215\textwidth]{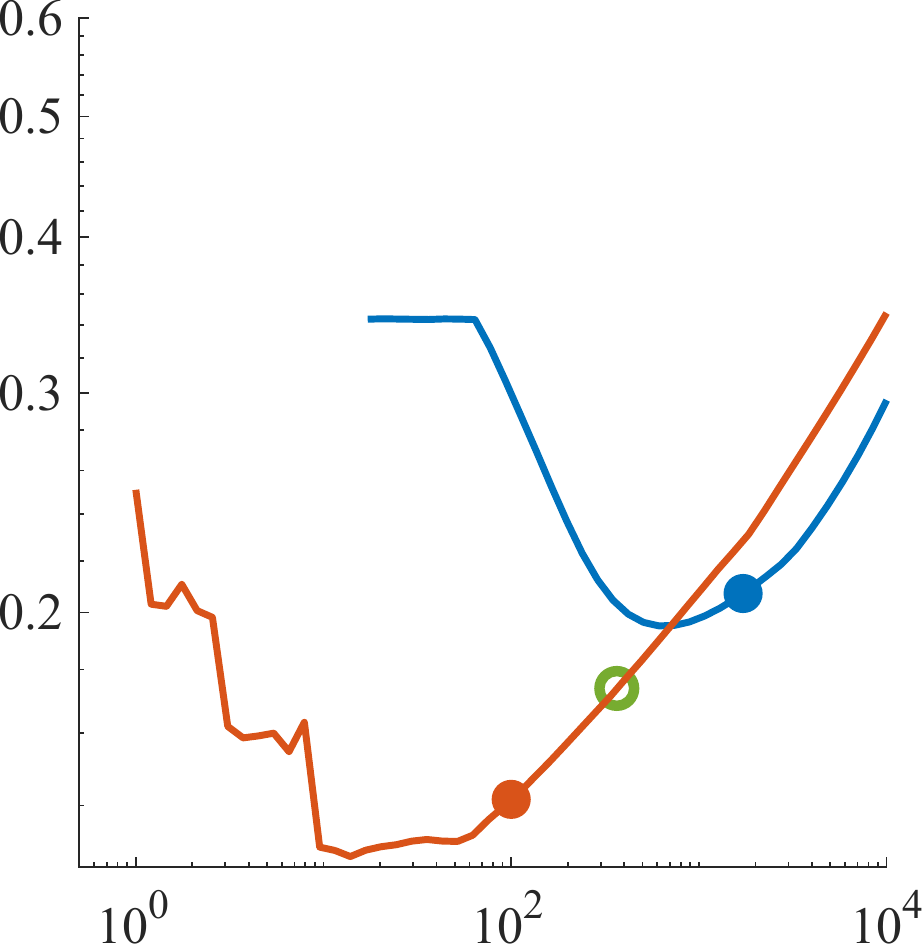} &
         \includegraphics[width=0.215\textwidth]{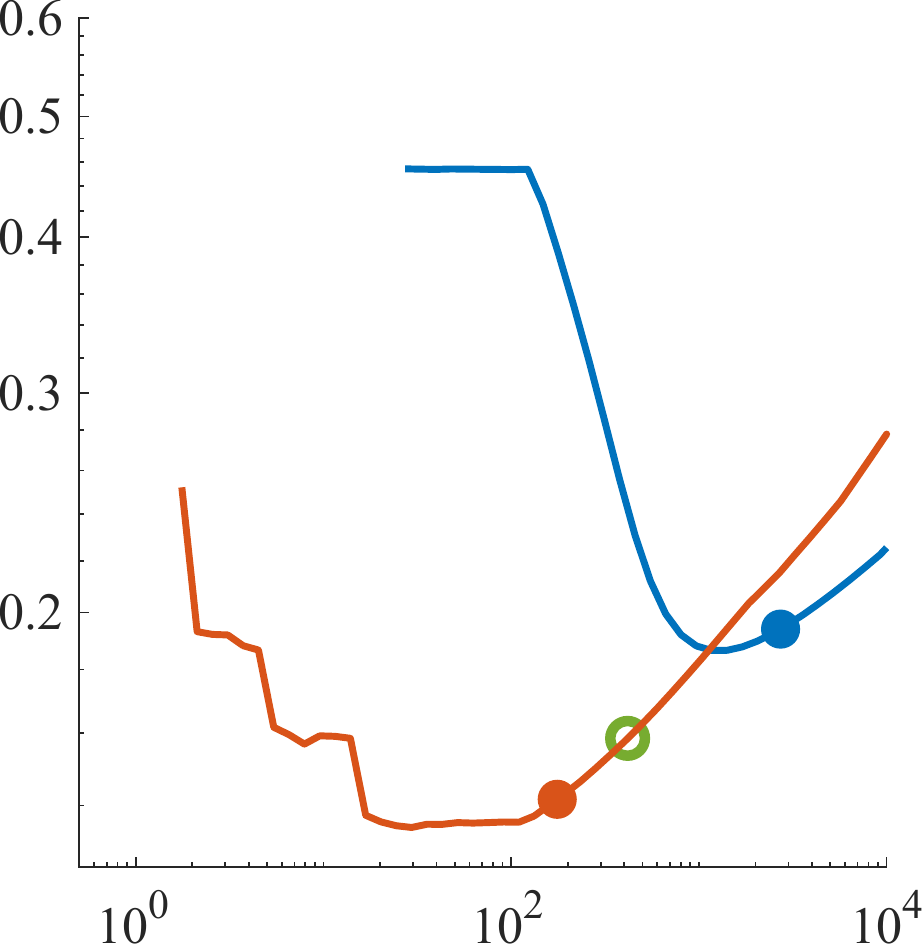} &
         \includegraphics[width=0.215\textwidth]{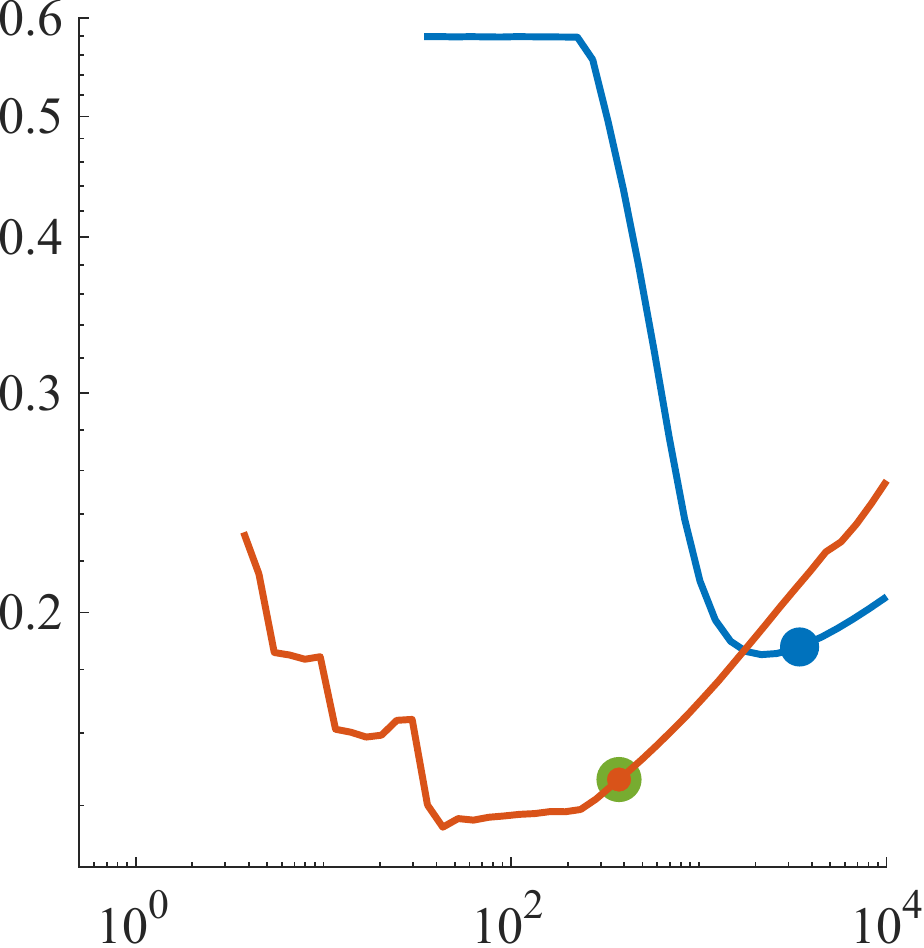} \\
         \includegraphics[width=0.226\textwidth]{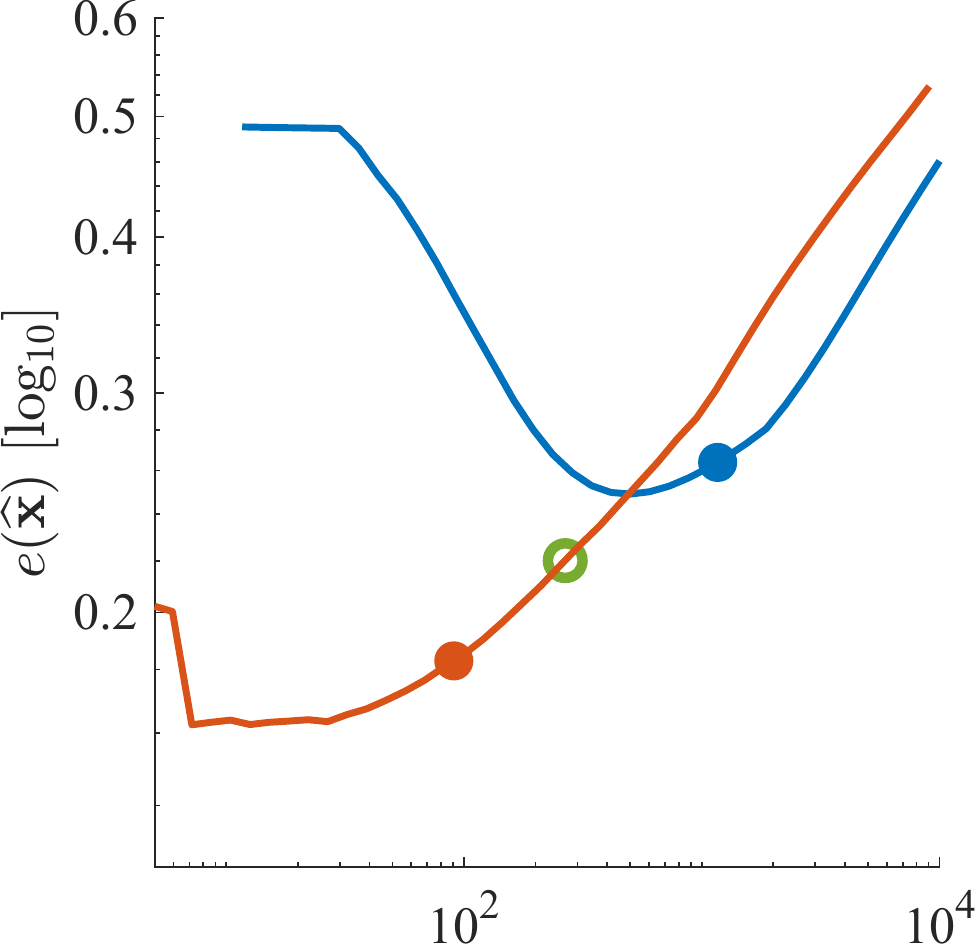} &
         \includegraphics[width=0.215\textwidth]{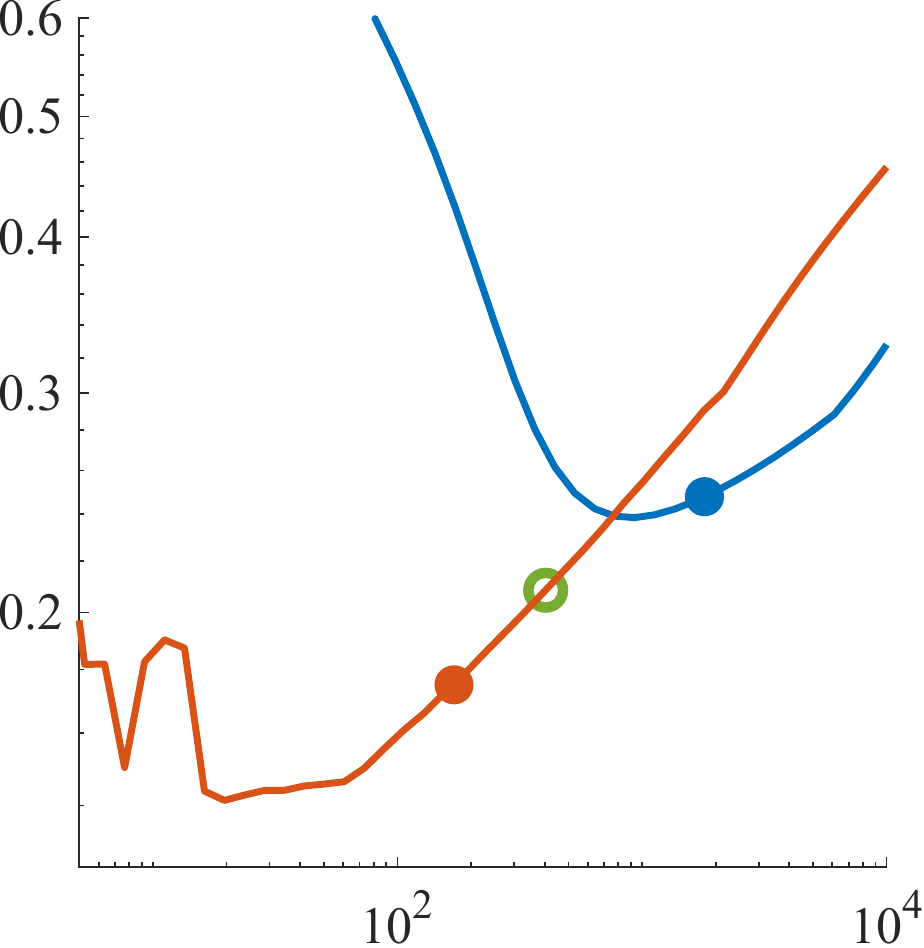} &
         \includegraphics[width=0.215\textwidth]{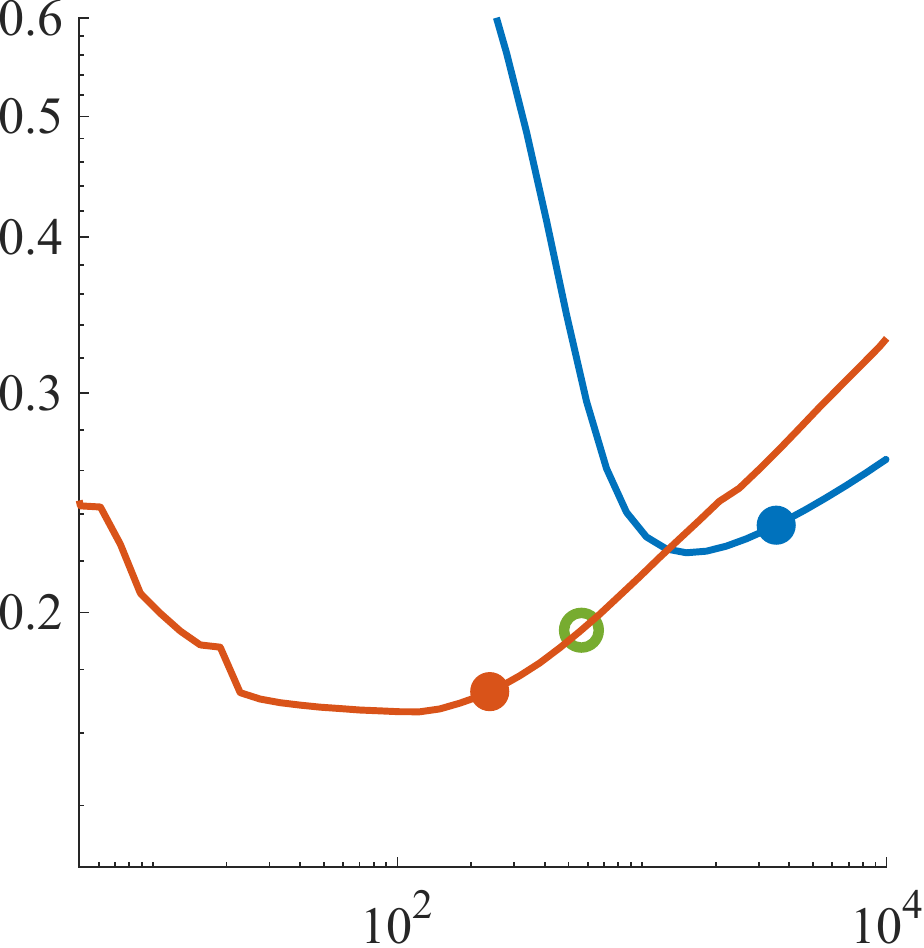} &
         \includegraphics[width=0.215\textwidth]{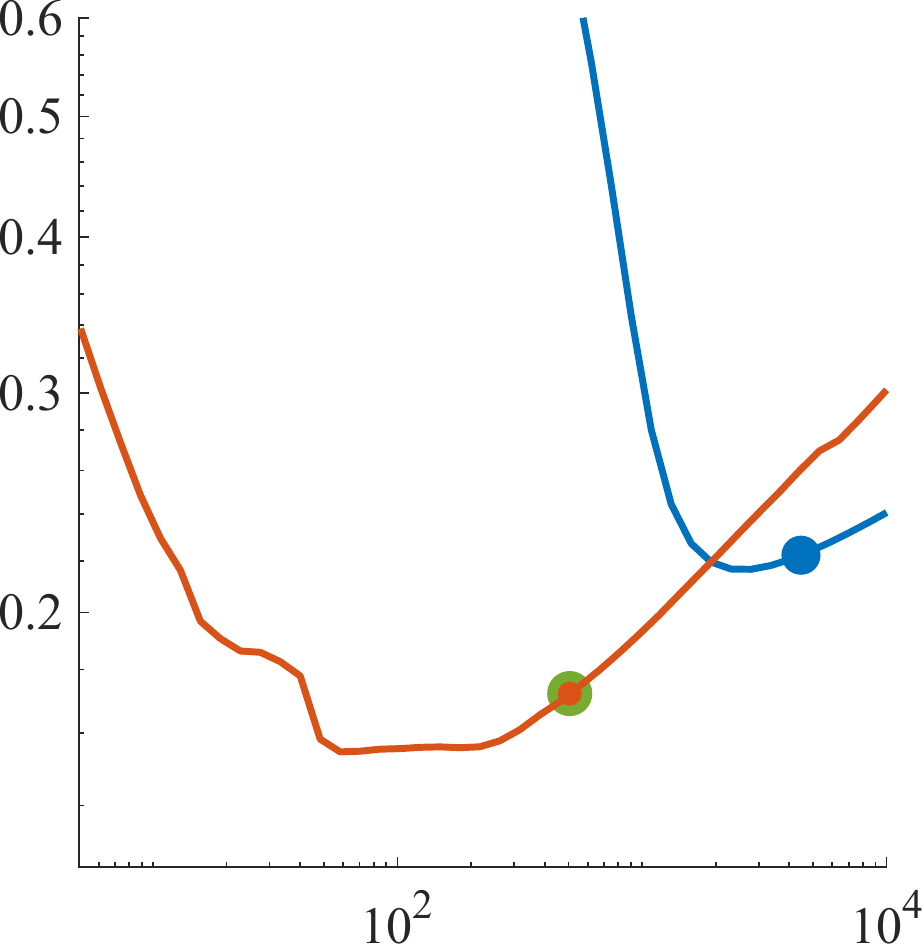} \\[-1.5ex]
         \tiny $\lambda$ [log$_{10}$] & \tiny $\lambda$ [log$_{10}$] & \tiny $\lambda$ [log$_{10}$]  & \tiny $\lambda$ [log$_{10}$] \\
    \end{tabular}
    \caption{Results for {\tt seismic} example for varying $\lambda$ and fixed regularization parameter $\mu$, $\mumap$ and $\muopt$ for the blue and red curves, respectively. Also marked for the given fixed $\mu$ choices are the optimal $\lambda$ found using the DP.  First row shows relative errors for different image sizes ($n = 64^2, 128^2, 256^2$, and $512^2$) with noise level of $10\%$ corresponding to a signal to noise ratio (SNR) of $20$. Second row shows relative errors with noise level of $20\%$ corresponding to an SNR  of $13.98$.
    \label{fig:experimentmu_seismic1}}
\end{figure}

\begin{figure}
    \centering
    \includegraphics[width=0.48\textwidth]{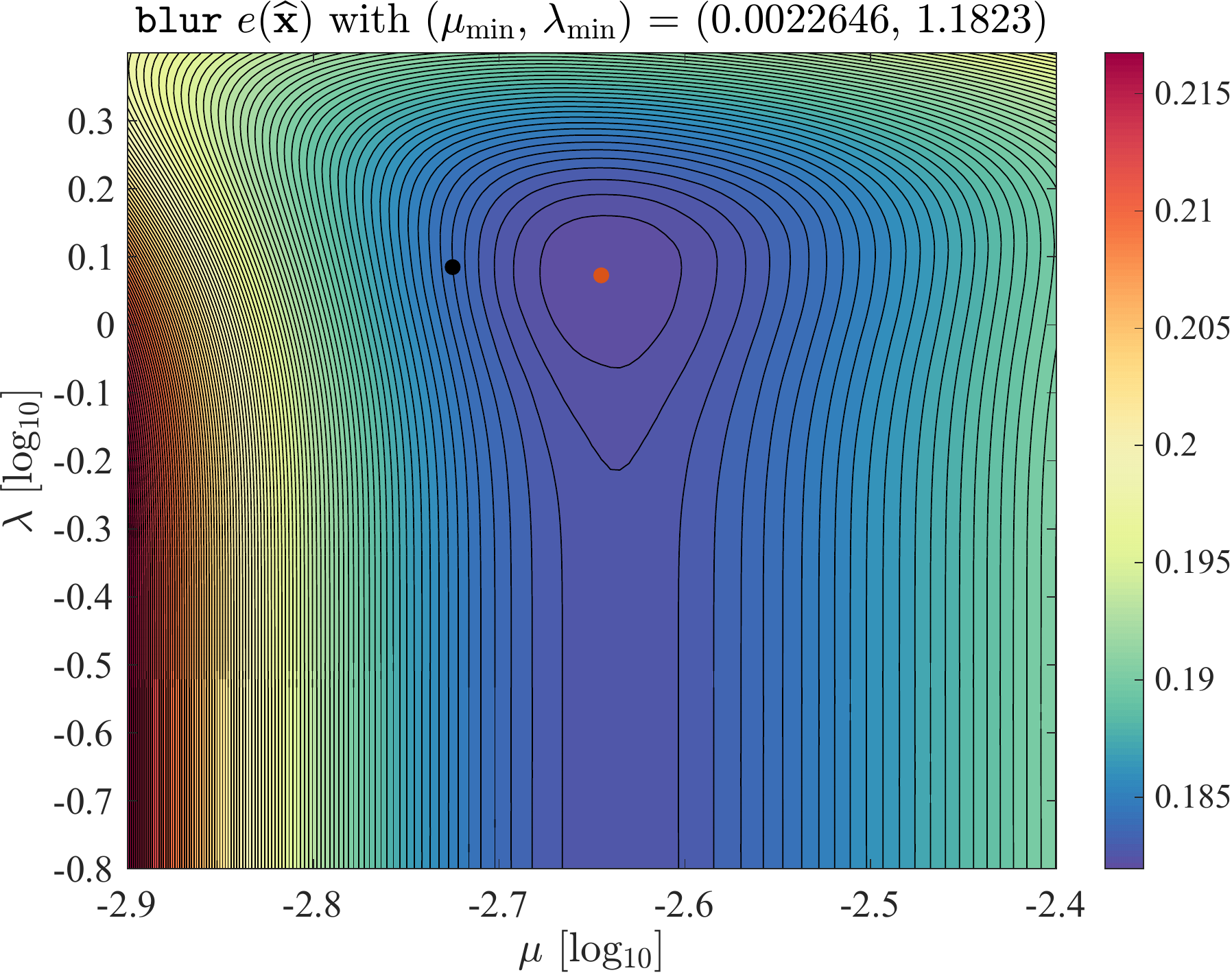}
    \includegraphics[width=0.48\textwidth]{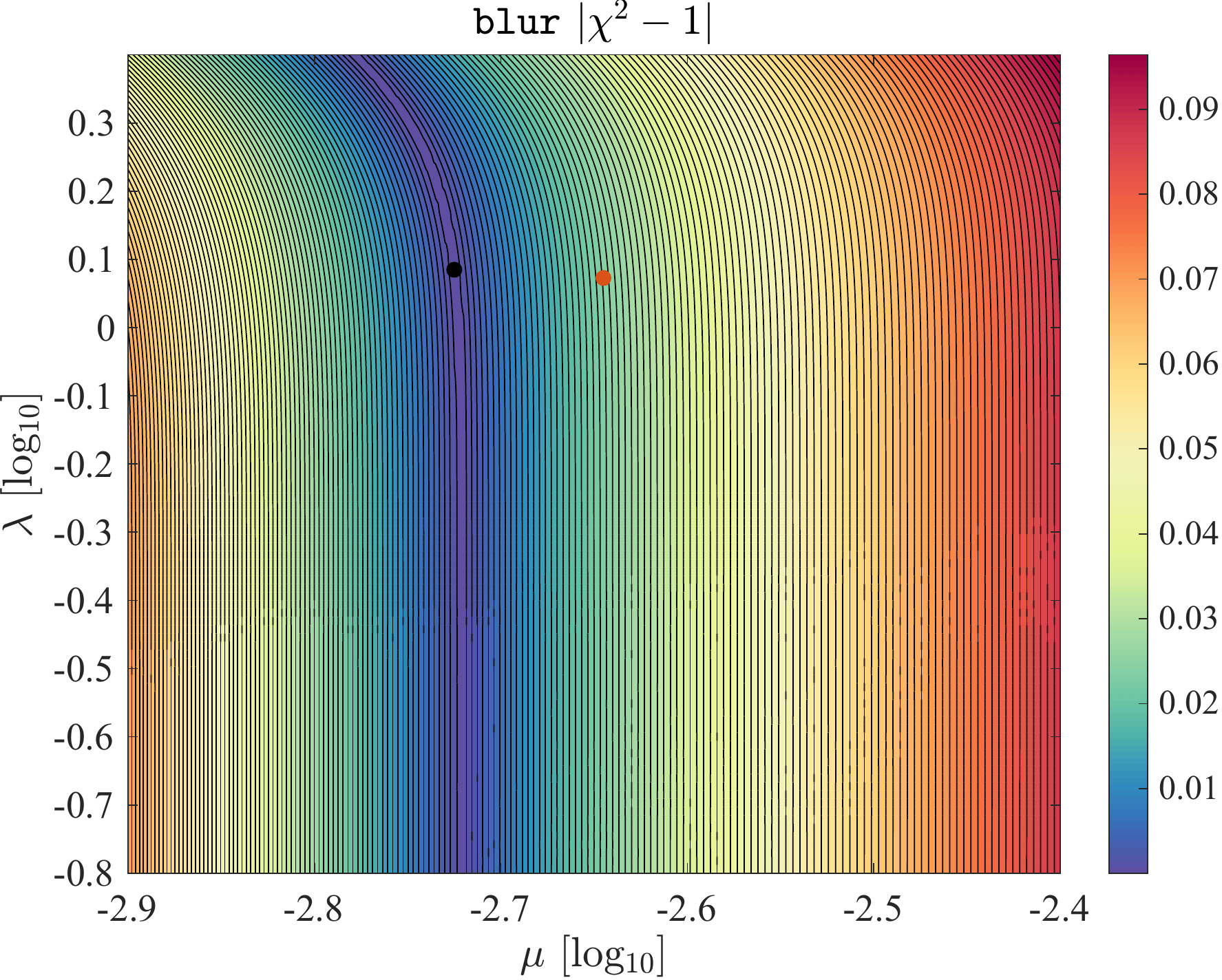}
 \caption{Problem {\tt blur}:  Calculating the relative error   and $\chi^2$ departure from $m\sigma^2$ for a logarithmically uniform grid of points $(\mu,\lambda)$. The minimum relative error is at the red dot and the minimum for the $\chi^2$ is at the black dot. The obtained minimum values for the relative errors at these points are $0.1820 $ and $0.1840$, and occur for shrinkage parameter $\gamma = 0.0016$, and $0.0013$  respectively.}
    \label{fig:blur_lambda_mu}
\end{figure}

\begin{figure}
    \centering
    \includegraphics[width=0.48\textwidth]{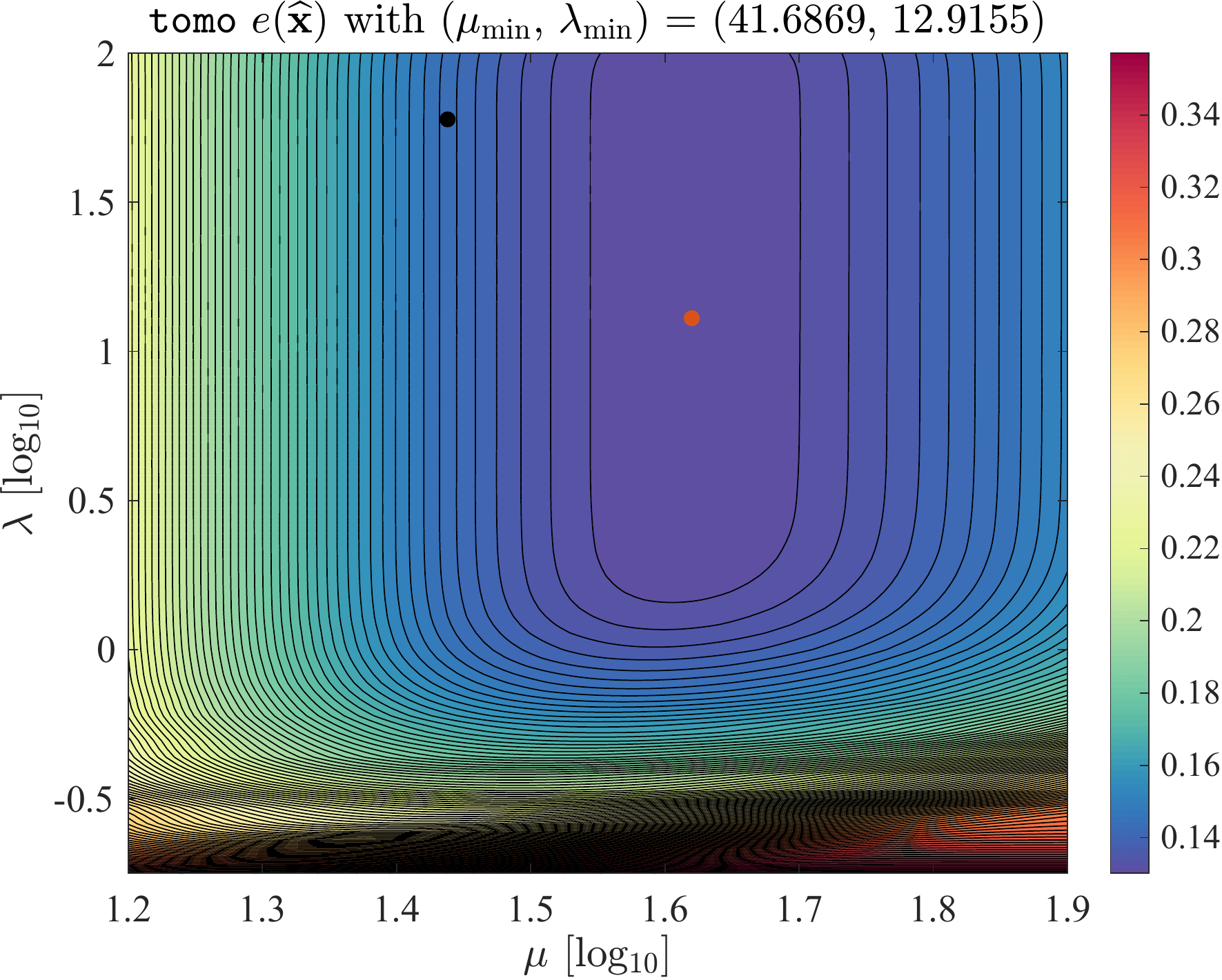}
    \includegraphics[width=0.48\textwidth]{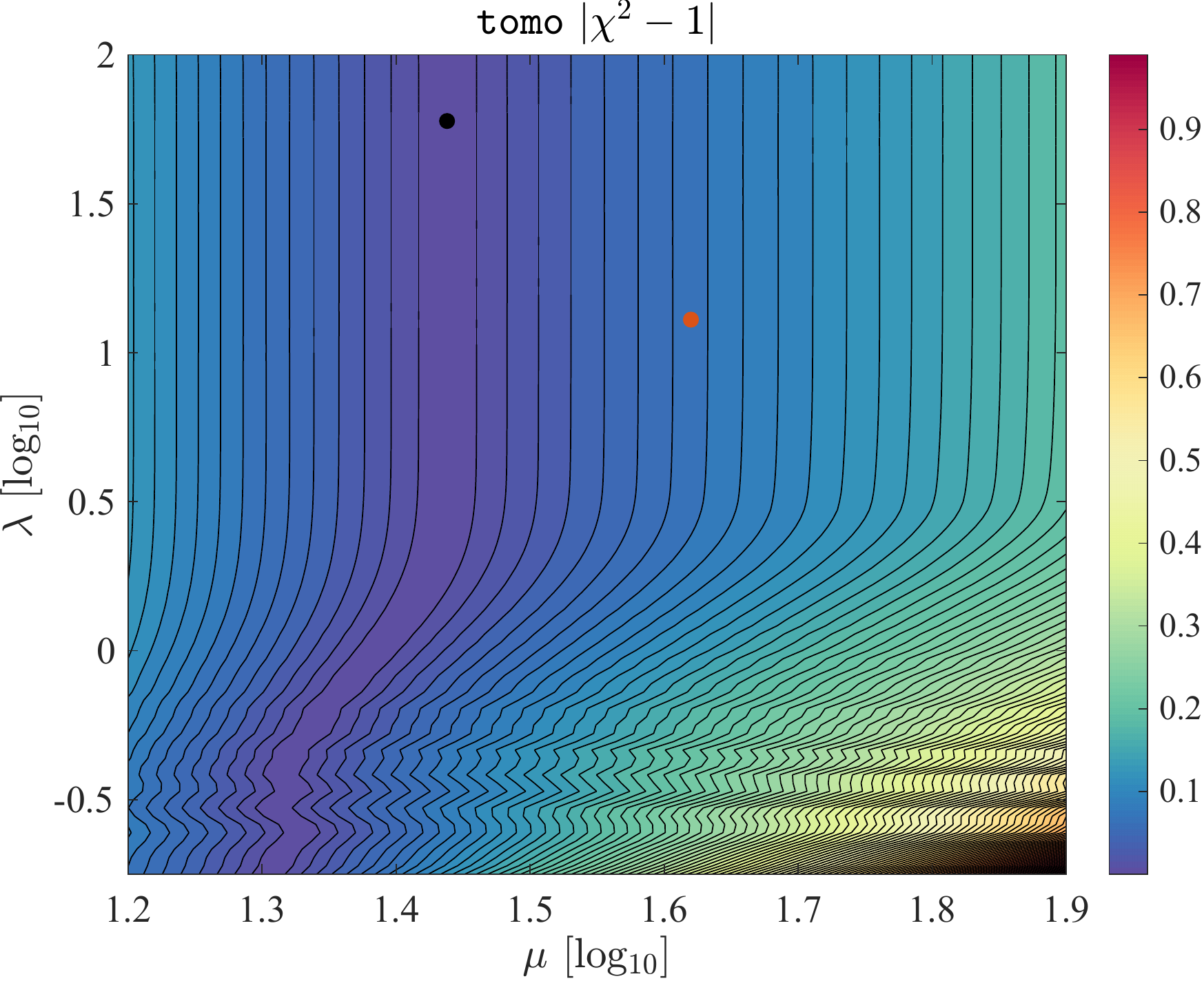}
     \caption{Problem {\tt tomo}:  Calculating the relative error   and $\chi^2$ departure from $m\sigma^2$ for a logarithmically uniform grid of points $(\mu,\lambda)$. The minimum relative error is at the red dot and the minimum for the $\chi^2$ is at the black dot. The obtained minimum values for the relative errors at these points are $0.1301 $ and $0.1449$, and occur for shrinkage parameter $\gamma = 0.2499$, and $0.0076$  respectively. }
    \label{fig:tomo_lambda_mu}
\end{figure}

\begin{figure}
    \centering
    \includegraphics[width=0.48\textwidth]{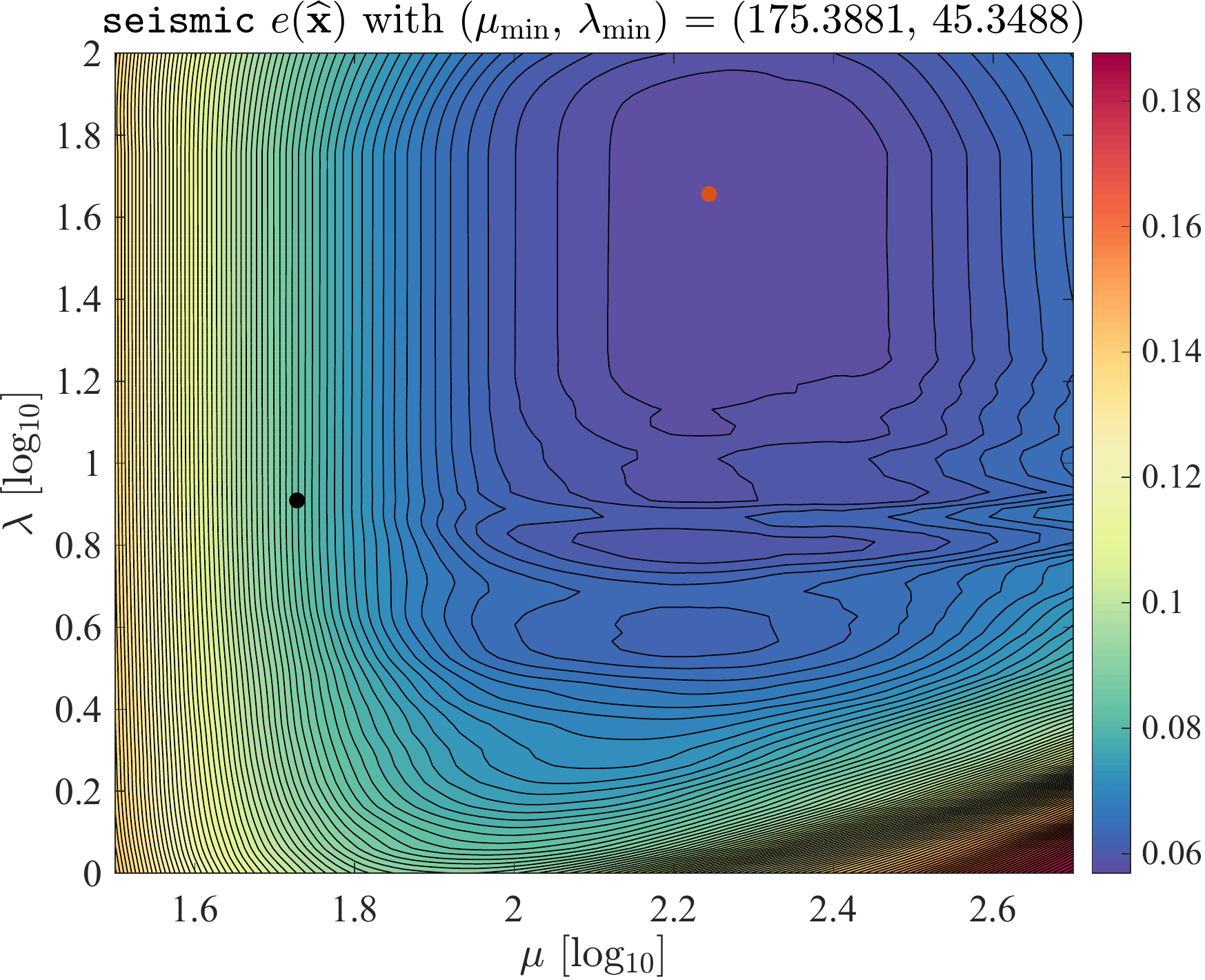}
    \includegraphics[width=0.48\textwidth]{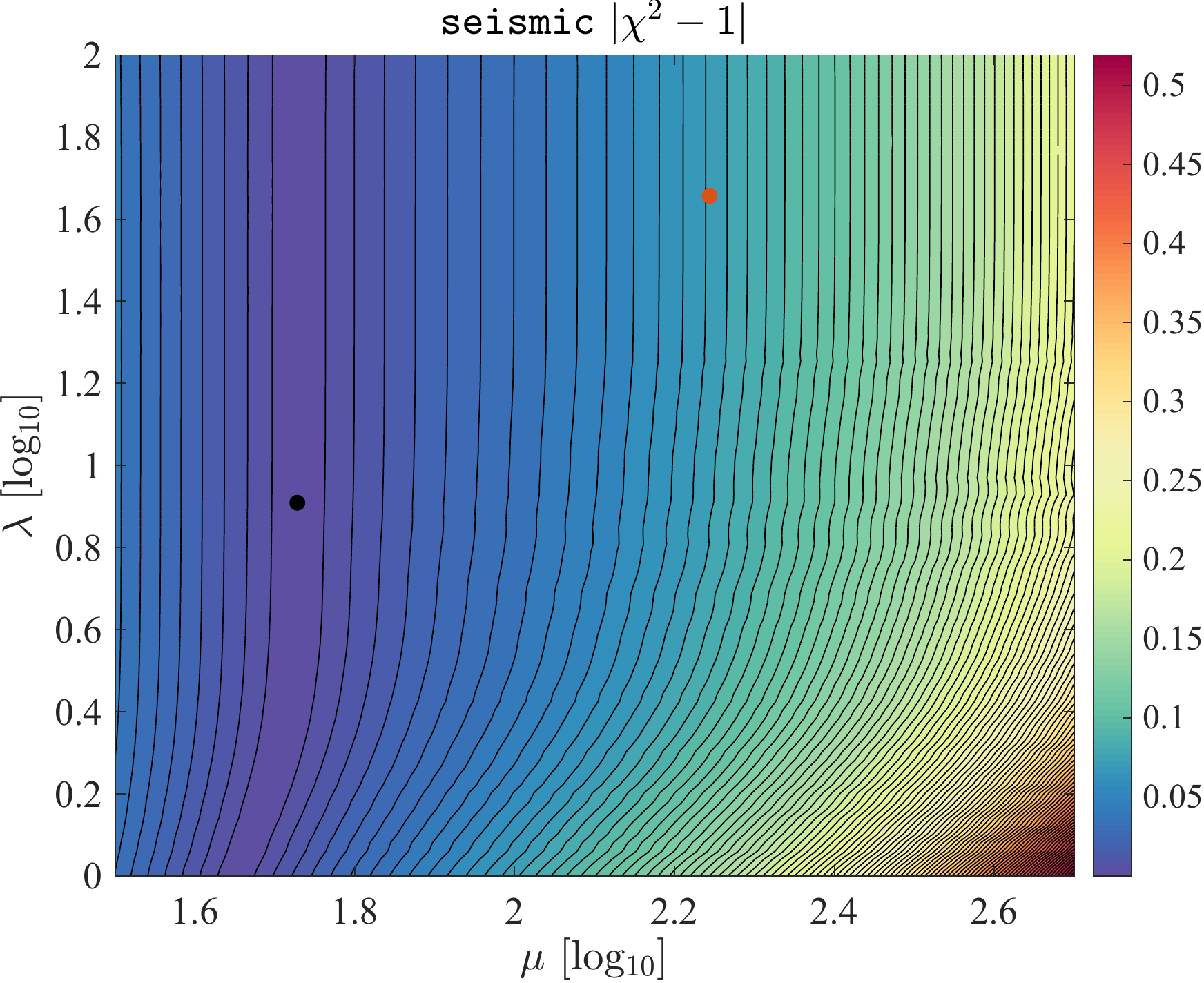}
    \caption{Problem {\tt seismic}:  Calculating the relative error   and $\chi^2$ departure from $m\sigma^2$ for a logarithmically uniform grid of points $(\mu,\lambda)$. The minimum relative error is at the red dot and the minimum for the $\chi^2$ is at the black dot. The obtained minimum values for the relative errors at these points are $0.0569 $ and $0.0846$, and occur for shrinkage parameter $\gamma = 0.0853$, and $0.8125$  respectively.}
    \label{fig:seismic_lambda_mu}
\end{figure}

Finally, to demonstrate the applicability of the  bisection algorithm for other differentially Laplacian operators $\bfD$, we present a small sample of results when $\bfD$ is the Laplace matrix. We use
the same  data sets {\tt blur}, {\tt tomo} and {\tt seismic}, with the same noise levels corresponding to SNRs of $20$ and $13.98$, and the same problem sizes, and replace the TV matrix by the Laplace matrix. In \cref{fig:experimentlaplace}, we collect the results in one plot per test, where in each plot we show the relative error curves for a range of $\lambda$ around the optimal $\mu$ for a specific shrinkage parameter $\gamma$ (solid symbols) found using the $\chi^2$-DF test. The results again demonstrate the ability of the bisection algorithm with the $\chi^2$-DF test to find solutions with near-optimal relative errors.

In summary, we observe empirically that \vpal~converges particularly fast with only a few iterations, when a near-optimal regularization parameter is selected using the $\chi^2$-DF test.
\begin{figure}
    \begin{center}
    \begin{tabular}{ccc}
         {\tt blur}  & {\tt tomo} & {\tt seismic}  \\
         \includegraphics[width=0.32\textwidth]{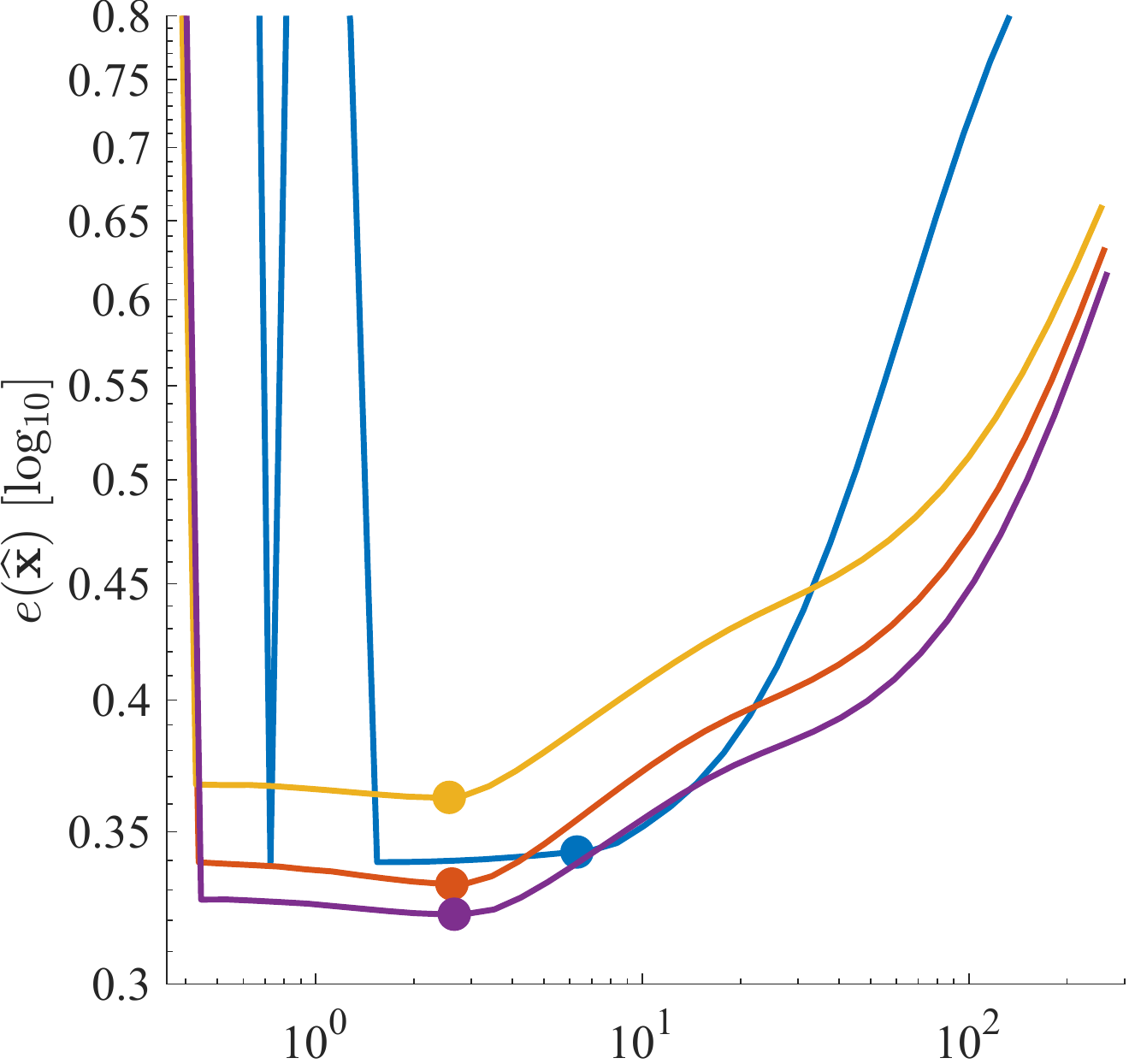} &
         \includegraphics[width=0.29\textwidth]{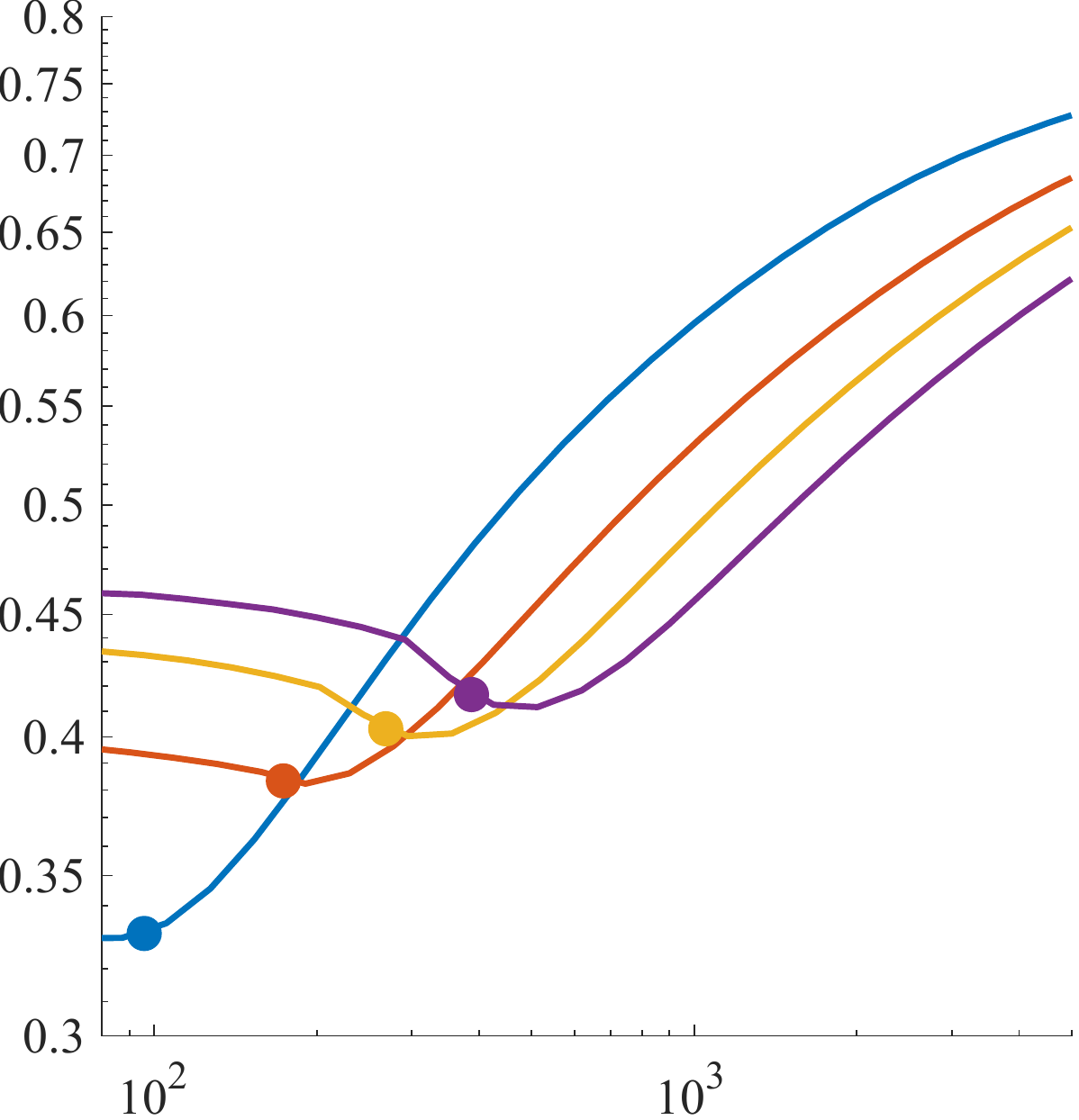} &
         \includegraphics[width=0.29\textwidth]{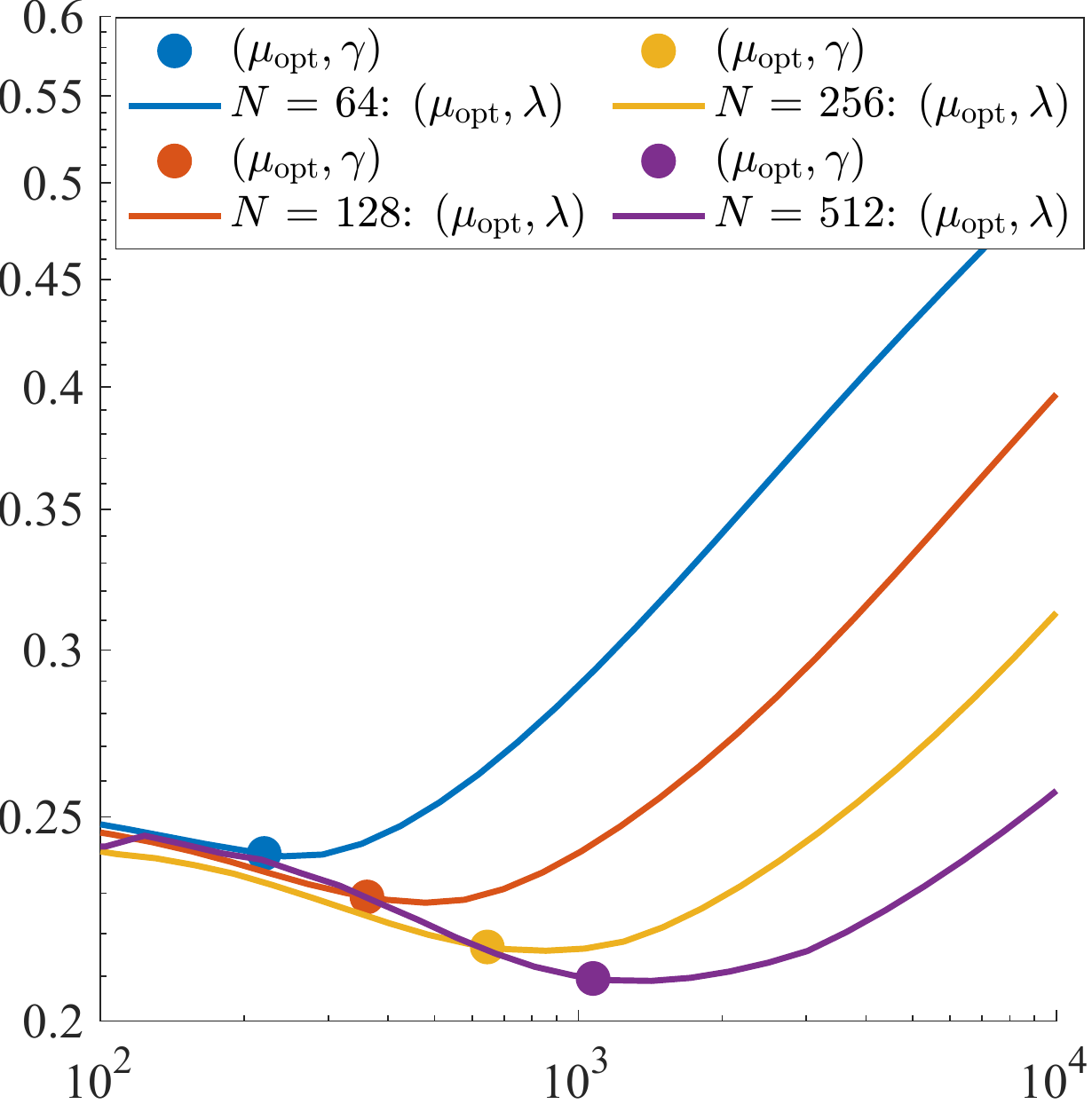} \\
         \includegraphics[width=0.32\textwidth]{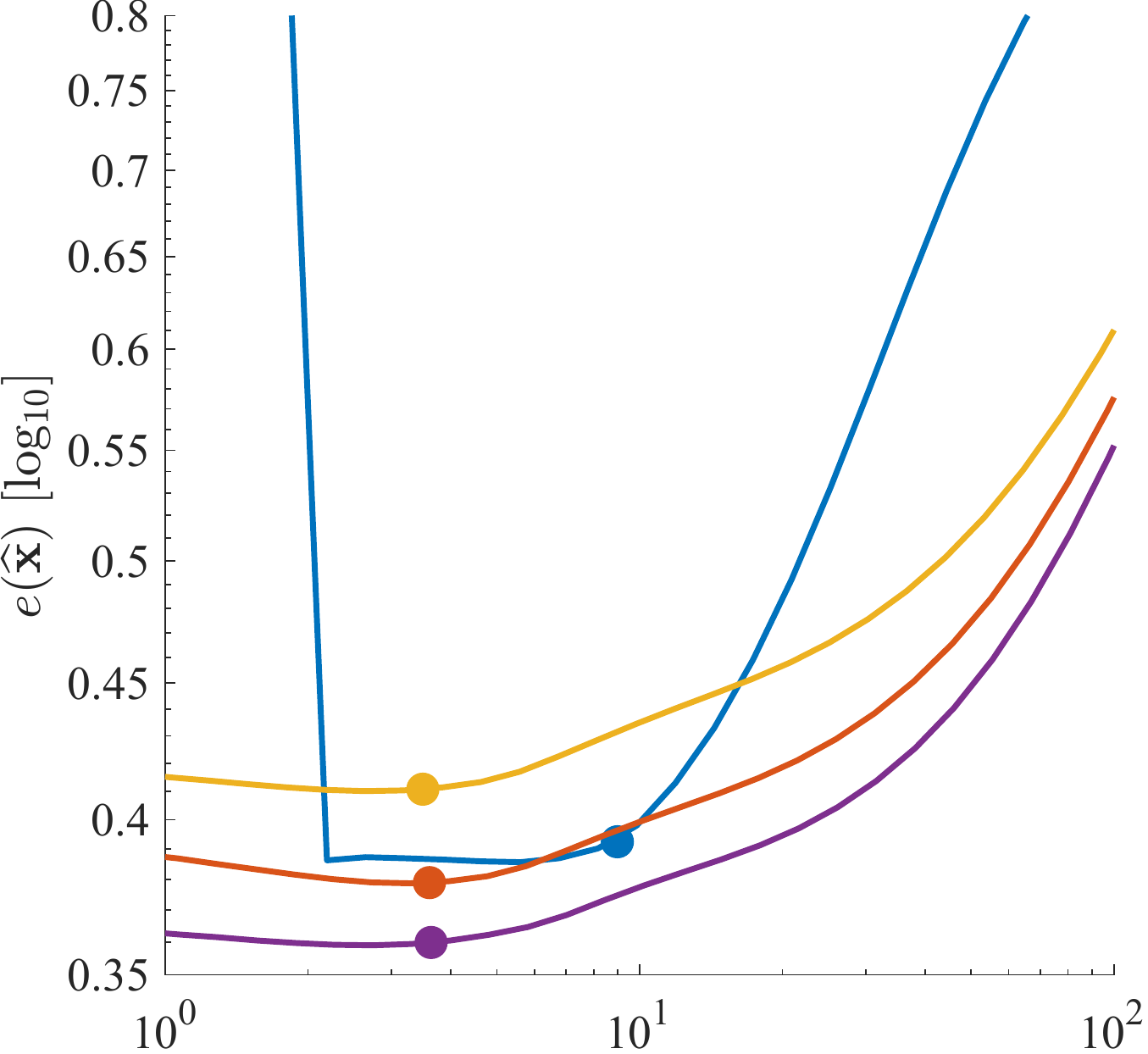} &
         \includegraphics[width=0.29\textwidth]{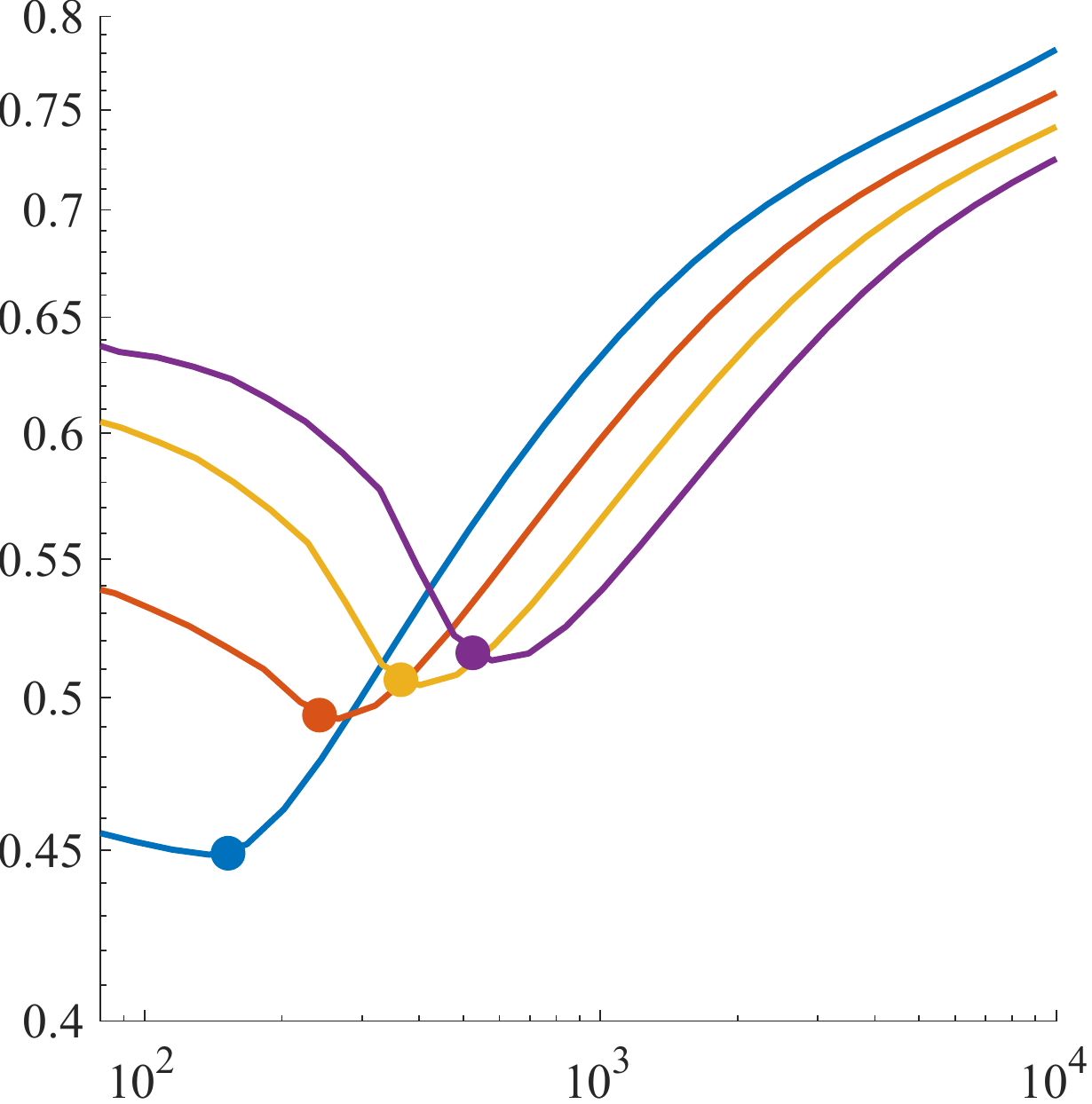} &
         \includegraphics[width=0.29\textwidth]{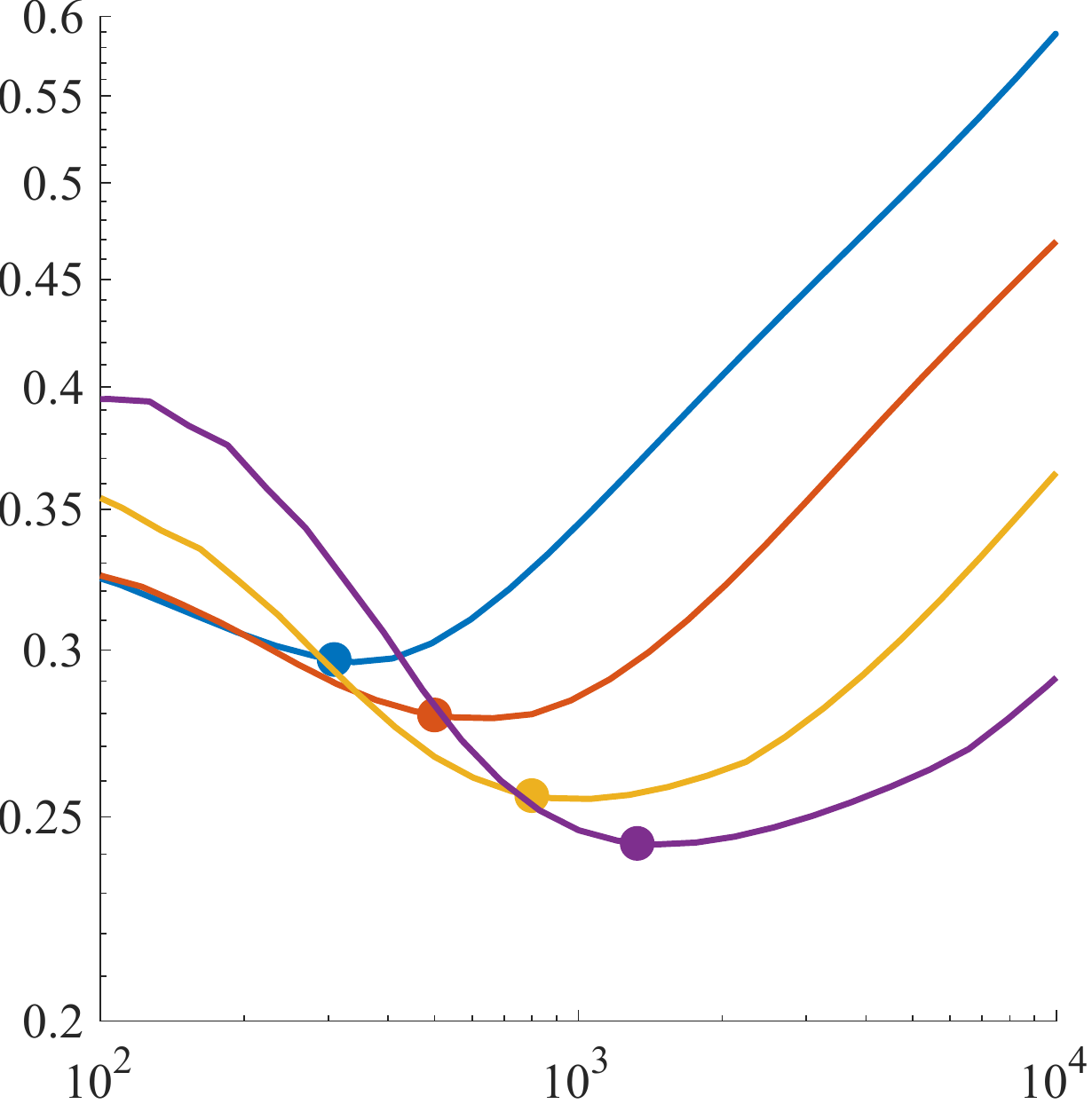}  \\[-1.5ex]
          \tiny $\lambda$ [log$_{10}$] & \tiny $\lambda$ [log$_{10}$]  & \tiny $\lambda$ [log$_{10}$] \\
    \end{tabular}
    \end{center}
    \caption{Results for the {\tt blur}, {\tt tomo} and {\tt seismic} problems using the Laplace operator for $\bfD$. Here blue, red, yellow and purple curves and solid symbols, correspond to results for problems of size $64$, $128$, $256$ and $512$, respectively. The solid symbols indicate the points selected by using the bisection algorithm applied to satisfy the $\chi^2$-DF test.  The  first row shows relative errors  with noise level of $10\%$ corresponding to a signal to noise ratio (SNR) of $20$. Second row shows relative errors with noise level of $20\%$ corresponding to an SNR  of $13.98$.
    \label{fig:experimentlaplace}}
\end{figure}

\section{Discussion \& conclusion}\label{sec:discussion}

In this work we presented a new method for solving generalized total variation problems using an augmented Lagrangian framework that utilizes variable projection methods to solve the inner problem with proven convergence. We further provided an automatic regularization selection method using a degrees of freedom argument. Our investigations included various numerical experiments illustrating the efficiency and effectiveness of our new {\rm vpal} method for different regularization operators $\bfD.$

Our work provides a first investigation into the variable projected augmented Lagrangian method, yet many research items remain open. Future research will take multiple directions. First, while \cref{thm:vpal} provides necessary convergence results, our implementation {\tt vpal} uses an inexact solve of the inner problem, i.e., a single CG update. We will investigate convergence properties of the inexact approach, utilizing results from inexact ADMM methods \cite{hager2019inexact} for the proof. Second, although we utilize a CG update for the inner iteration, other update strategies may be employed. Since this is a nonlinear problem we may for instance utilize LBFGS updates or nonlinear Krylov subspace methods. Third, estimating a good regularization parameter $\mu$ remains a costly task. In \cite{chung2021computational} iterative regularization approaches estimating $\mu$ on a subspace were investigated, and demonstrated a computational advantage. We will extend the DF~argument to be used to find $\mu$ using a standard bisection algorithm at relatively low cost, when the original model parameters can be assumed to be differentially Laplacian, as is the case for standard image deblurring problems.  Fourth, we consider $\ell^2-\ell^1$ norm regularization here; however, the developed approach extends also to other $\ell^p-\ell^q$ norm problems \cite{chung2021efficient} and even more general objective functions. We will investigate the convergence and numerical advantages and disadvantages of utilizing a variable projected approach for such $\ell^p-\ell^q$ problems and for supervised learning loss function fitting within this framework. Fifth, row action methods have been developed to solve least squares and Tikhonov type problems of extremely large-scale where the forward operator $\bfA$ is too large to keep in computer memory \cite{chung2020sampled,slagel2019sampled}. We will investigate how \vpal~can be extended to such settings and investigate convergence properties and sampled regularization approaches. Sixth, we will investigate and extend our variable projected optimization method to other suitable optimization problems such as for efficiently solving the Sylvester equations \cite{benner2009adi}.

\section*{Acknowledgments}
This work was initiated as a part of the SAMSI Program on Numerical Analysis in Data Science in 2020. Any opinions, findings, and conclusions or recommendations expressed in this material are those of the authors and do not necessarily reflect the views of the National Science Foundation. We would like to thank Michael Saunders and Volker Mehrmann for their many helpful suggestions and for their constructive feedback on an early draft of this paper.

\bibliographystyle{siamplain}
\bibliography{references}

\end{document}